\documentclass[11pt,reqno]{amsart}
\usepackage[T1]{fontenc}
\usepackage[utf8]{inputenc}
\usepackage{amsfonts, amssymb, amsmath, amsthm,latexsym,slashed}
\usepackage{mathtools, thm-restate}
\usepackage{stmaryrd}
\usepackage{hyperref}
\usepackage{amsrefs}
\usepackage{a4wide}
\usepackage{enumerate}
\usepackage{mathrsfs}
\usepackage{bbm}
\usepackage{upgreek}
\usepackage{enumitem}
\usepackage{nicefrac}
\usepackage{tikz} 
\usetikzlibrary{arrows}
\usepackage{caption}
\usepackage{ulem}
\hypersetup{
    colorlinks,
    linkcolor={green!50!black}, % green!50!black
    citecolor={blue}, % blue
    urlcolor={cyan!75!blue} % cyan!75!blue
}
\newcommand{\Hom}[1]{\mathrm{Hom} ({#1})}
\newcommand{\End}[1]{\mathrm{End} ({#1})}

\DeclareMathOperator{\tr}{tr}
\DeclareMathOperator{\rk}{rk}
\DeclareMathOperator{\Hess}{Hess}
\newcommand{\N}{\mathbb{N}}
\newcommand{\Z}{\mathbb{Z}}
\newcommand{\R}{\mathbb{R}}
\newcommand{\C}{\mathbb{C}}

\newcommand{\one}{\mathbbm{1}}

\newcommand{\supp}[1]{\mathrm{supp} \left({#1}\right)} 
\newcommand{\singsupp}[1]{\mathrm{singsupp} \left({#1}\right)}
\renewcommand{\Im}{\mathrm{Im}}

\newcommand{\grad}{\mathrm{grad}}
\renewcommand{\div}{\mathrm{div}}
\renewcommand{\Box}{\square}
\newcommand{\WF}[1]{\mathrm{WF} ({#1})}
\newcommand{\WFPrime}[1]{\mathrm{WF}^{\prime} ({#1})}
\newcommand{\ES}[1]{\mathrm{ES} ({#1})} 
\newcommand{\Feyn}{\mathrm{F}}
\newcommand{\aFeyn}{\mathrm{aF}}
\newcommand{\ret}{\mathrm{ret}}
\newcommand{\adv}{\mathrm{adv}}
\newcommand{\comSecE}{C_{\mathrm{c}}^{\infty} (M; E)}
\newcommand{\secE}{C^{\infty} (M; E)}
\newcommand{\dualComSecE}{\mathcal{D}^{\prime} (M; E)}
\newcommand{\dualSecE}{\mathcal{E}^{\prime} (M; E)}
\newcommand{\advGreenOp}{G^{\mathrm{adv}}} 
\newcommand{\retGreenOp}{G^{\mathrm{ret}}}
\newcommand{\leftParametrix}{G_{\mathrm{L}}}
\newcommand{\rightParametrix}{G_{\mathrm{R}}}
\newcommand{\FeynGreenOp}{\mathsf{G}^{\mathrm{F}}} 
\newcommand{\FeynParametrix}{G^{\mathrm{F}}}
\newcommand{\antiFeynParametrix}{G^{\mathrm{aF}}}

\newcommand{\dotCoTanM}{\dot{T}^{*} M}
\newcommand{\dotCoTanN}{\dot{T}^{*} N}

\newcommand{\dotCoTanMN}{\dot{T}^{*} (M \times N)}
\newcommand{\coLightBun}{\dot{T}_{0}^{*} M}
\newcommand{\dotCoTanRn}{\dot{T}^{*} \! \mathbb{R}^{n}}
\newcommand{\ms}{\scriptscriptstyle}
\newcommand{\pr}{\mathrm{pr}}
\newcommand{\Rn}{\mathbb{R}^{n}} 
\newcommand{\Rd}{\mathbb{R}^{d}}

\newcommand{\CN}{\mathbb{C}^{N}}
\newcommand{\xxiNot}{(x_{0}, \xi^{0})}

\newcommand{\PsiDO}[2]{\Psi \mathrm{DO}^{{#1}} ({#2})}
\newcommand{\PsiDONU}[1]{\Psi \mathrm{DO}^{{#1}} (U',\mathrm{Mat}(N))}
\newcommand{\PsiDON}[1]{\Psi \mathrm{DO}^{{#1}} (\mathbb{R}^{n},\mathrm{Mat}(N))}
\DeclareMathOperator{\ri}{i}
\newcommand{\re}{\mathrm{e}}
\newcommand{\rd}{\mathrm{d}} 
\newcommand{\Maslov}{\mathbb{M}}
\newcommand{\xtheta}{(x; \theta)}
\newcommand{\xytheta}{(x, y; \theta)}

\newcommand{\xthetaNot}{(x_{0}; \theta^{0})}
\newcommand{\xxi}{(x, \xi)}

\newcommand{\yeta}{(y, \eta)}
\newcommand{\xxiyeta}{(x, \xi; y, \eta)}
\newcommand{\xxiyetaNot}{(x_{0}, \xi^{0}; y_{0}, \eta^{0})}

\newcommand{\parDeri}[2]{\frac{\partial {#2}}{\partial {#1}}}
\newcommand{\sgn}{\mathrm{sgn}}
\newcommand{\Char}[1]{\mathrm{Char} (#1)}
\newcommand{\cD}{\mathcal{D}}
\newcommand{\cE}{\mathcal{E}}
\newcommand{\cU}{\mathcal{U}}
\newcommand{\cV}{\mathcal{V}}
\newcommand{\cW}{\mathcal{W}}
\newcommand{\sU}{\mathscr{U}}
\newcommand{\symb}[1]{\sigma_{\ms{#1}}}
\newcommand{\subSymb}[1]{\sigma^{\mathrm{sub}}_{\ms{#1}}}
\newcommand{\totSymb}[1]{\sigma^{\mathrm{tot}}_{\ms{#1}}}
\newcommand{\dVolh}{\mathrm{d} \mathsf{v}_{h}} 
\newcommand{\dVolg}{\mathrm{d} \mathsf{v}_{g}}
\newcommand{\dVol}{\mathrm{d} \mathsf{v}}
\newcommand{\nM}{n_{\ms{M}}} 
\newcommand{\nN}{n_{\ms{N}}}
\newcommand{\matN}{\mathrm{Mat} (N)}
\newtheorem{theorem}{Theorem}[section]
\newtheorem{lemma}[theorem]{Lemma}
\newtheorem{corollary}[theorem]{Corollary}
\newtheorem{proposition}[theorem]{Proposition}
\theoremstyle{definition}
\newtheorem{definition}[theorem]{Definition}

\newtheorem{remark}[theorem]{Remark}

\allowdisplaybreaks
\title[Feynman Propagators]{On microlocalisation and the construction of Feynman Propagators for normally hyperbolic operators}
\author[O.~Islam]{Onirban Islam}
\address{Institut f\"{u}r Mathematik, Universit\"{a}t Potsdam, Karl-Liebknecht-Stra\ss{}e 24-25, 14476 Potsdam, Germany.}
\email{onirban.islam@math.uni-potsdam.de}
\author[A.~Strohmaier]{Alexander Strohmaier}
\address{Institut f\"{u}r Analysis, Leibniz Universit\"{a}t Hannover, Welfengarten 1, 30167 Hannover, Germany.}
\email{a.strohmaier@math.uni-hannover.de}
\subjclass[2010]{}
\keywords{Wave-type operators, Dirac-type operators, Feynman propagators, Hadamard states}
\date{19 December 2024}
\begin{document}
% ================================================ 
% abstract 
\begin{abstract}
% ================================================ 
This article gives global microlocalisation constructions for normally hyperbolic operators on a vector bundle over a globally hyperbolic spacetime in geometric terms. 
As an application, this is used to generalise  the Duistermaat-H\"{o}rmander construction of Feynman propagators, therefore incorporating the most important non-scalar geometric operators.
It is shown that for normally hyperbolic operators that are selfadjoint with respect to a hermitian bundle metric, the Feynman propagators can be constructed to satisfy a positivity property that reflects the existence of Hadamard states in quantum field theory on curved spacetimes.
We also give a more direct construction of the Feynman propagators for Dirac-type operators on a globally hyperbolic spacetime. 
Even though the natural bundle metric on spinors is not positive-definite, in this case, we can give a direct microlocal construction of a Feynman propagator that satisfies positivity.
\end{abstract}
% ================================================ 
\maketitle
\setcounter{tocdepth}{1}
\tableofcontents
% 
% 
% 
% 
% 
% 
% 
% 
% 
% 
% ================================================ 
\section*{Introduction}
% ================================================ 
If $(M,g)$ is a globally hyperbolic spacetime then the Cauchy problem for any normally hyperbolic operator, i.e., any operator of the form 
\begin{equation*}
	\Box = - g^{ij} \partial_i \partial_j + \textrm{lower order terms}, \quad i,j = 1, \ldots, n := \dim M 
\end{equation*}
(in local coordinates using Einstein's summation convention) is well-posed. 
This implies the existence of unique advanced and retarded fundamental solutions (propagators). 
The properties of these fundamental solutions are extremely important for the understanding of classical wave propagation, such as the electromagnetic waves. 
They appear naturally because of causality: the retarded fundamental solution propagates to the future, whereas the advanced fundamental solution propagates to the past. 
In quantum field theory, the appearance of time ordering and the enforcement of positivity of energy led to the development of another type of fundamental solution, the Feynman propagator.  
It propagates positive energy solutions to the future and negative energy solutions to the past and thus combines causality with the notion of positivity of energy.
Explicit formulae for Lorentz invariant Feynman propagators for normally hyperbolic operators in Minkowski spacetime can be found in basic quantum field theory textbooks. 
The usual construction employs the Fourier transform. 
However, on a generic globally hyperbolic spacetime there is neither a Fourier transform nor any reasonable notion of energy in the absence of a global time-like Killing vector field. 
A priori, it is, therefore unclear what a Feynman propagator should be.

In the theory of partial differential equations, the notion of a parametrix is often useful in the first stage on the construction of a true fundamental solution. 
A parametrix, per se, is an inverse of the operator $\Box$ modulo smoothing operators. 
Parametrices are considered equivalent if they differ by smoothing operators.
It was a deep insight of 
Duistermaat and H\"{o}rmander~\cite{Duistermaat_ActaMath_1972}   
that there is a well-defined notion of Feynman parametrices and these parametrices are unique up to smoothing operators. 
In other words, they are unique as parametrices. 
In fact, the notion of distinguished parametrices for any general pseudodifferential operator of real principal type is given in their seminal article where they have identified a geometric notion of pseudoconvexity which allows to prove uniqueness of such parametrices. 
Roughly speaking, despite of not having any notion of energy due to the lack of a global timelike Killing vector field, there is still a microlocal notion of positivity of energy and a corresponding flow on the cotangent bundle. 
This microlocal notion can be used to characterise Feynman parametrices: they are distinguished by their wavefront sets rather than the support properties of the Fourier transform.

Feynman parametrices play an extremely important role in quantum field theory on curved spacetimes, and the theory has actually been developed to a certain extent first in the physics literature.
Canonical quantisation of linear fields can be done in two stages. 
In the first step, one constructs a field algebra from the space of solutions of the 
respective equation of motion~\cite{Dimock_CMP_1980, Dimock_AMS_1982} 
and next, some state is required to construct a Hilbert space representation of this field algebra. 
Whereas, in Minkowski spacetime, these two steps are usually combined into one owing to the existence of the vacuum state, it is more fruitful to separate them in curved spacetimes. 
The first step can be done without any problems, exactly the same way as in Minkowski spacetime but the second step necessitates a notion of a reasonable state. 
It has been realised that a state compels certain conditions in order to perform the usual operations in 
perturbative quantum field theory~\cite{Brunetti_CMP_2000, Hollands_CMP_2001, Hollands_CMP_2002}. 
One of the identified conditions is a restriction of the type of singularity that one obtains from the state, the so called Hadamard condition. 
Although Duistermaat and H\"{o}rmander were certainly aware of the developments in physics, it was realised only much later by 
Radzikowski~\cite{Radzikowski_CMP_1996} 
that the expectation values of the timeordered products with respect to states satisfying the Hadamard condition are Feynman propagators. 
In fact, the construction of a Hadamard state is equivalent to the construction of a Feynman propagator that satisfies a certain positivity property. 
The fact that this positivity property holds for parametrices was already shown by Duistermaat and H\"{o}rmander. 
This property for parametrices implies the existence of such Feynman propagators and hence of Hadamard states 
(see~\cite{Lewandowski_JMP_2022}).

Apart from its applications in physics, Feynman parametrices also play an increasingly important role in mathematics.
For instance, they appear in the context of Lorentzian index theory as inverses modulo compact operators of the Dirac operator with Atyiah-Patodi-Singer boundary conditions~\cite{Baer_AJM_2019}, 
local index theorem~\cite{Baer},   
and asymptotically static spacetimes~\cite{Shen}.  
These concepts also arise in Vasy's treatment of asymptotically hyperbolic problems, see for example~\cite{GellRedman_CMP_2016}  
and references therein, where even a non-linear problem is discussed in this context.
Moreover, the notion of Feynman propagator also comes up naturally in the Lorentzian generalisation of the Duistermaat-Guillemin-Gutzwiller trace formula~\cite{Strohmaier_AdvMath_2021}. 

So far, the Duistermaat and H\"{o}rmander's construction of distinguished parametrices for any pseudodifferential operator of real principal type has not been used much in the physics literature probably due to its great generality and the associated complex notation. 
Their main idea is, however, compelling and simple. 
The operator can be microlocally conjugated to a vector field and it is therefore sufficient to construct a parametrix for the operator of differentiation on the real line. 
There are two distinguished fundamental solutions for the operator of differentiation. 
A choice of parametrix for each connected component of the characteristic set of the operator results a distinguished parametrix. 
Thus, there are $2^N$ distinguished parametrices for such a pseudodifferential operator if $N$ is the number of connected components of its characteristic set. 
For the wave operator, this gives $4$ distinguished parametrices in dimensions $n \geq 3$ and $16$ distinguished parametrices in $n=2$.

The aim of this article is three fold. 
First, we would like to revise and simplify the construction of 
Duistermaat-H\"ormander~\cite{Duistermaat_ActaMath_1972} 
in the special case of any normally hyperbolic operator. 
The second aim is to fill a gap in the literature: microlocalisation and the corresponding construction of distinguished parametrices in Duistermaat-H\"ormander is covered in the literature only for scalar operators.
Several constructions in index theory, in trace formulae, and also in physics require the existence and uniqueness of Feynman parametrices for operators acting on vector bundles. 
It is known that most of the constructions carry over to the case of any normally hyperbolic operator, since its principal symbol is scalar. 
There are, however, also important differences that appear on the level of subprincipal symbol. 
In this article, we would like to give a precise statement of microlocalisation (Theorems~\ref{thm: microlocalisation_P} and~\ref{thm: microlocalisation_NHOp}) for these class of geometric operators. 
We then provide a detailed construction of Feynman parametrices (Theorem~\ref{thm: exist_unique_Feynman_parametrix_NHOp}) for vector bundles with complete proofs for their uniqueness and discuss the effect of curvature of the bundle connection. 
In case, the vector bundle has a hermitian inner product with respect to which the operator is formally selfadjoint, we show (Proposition~\ref{thm: positivity_Feynman_minus_adv_NHOp}) that the construction can be carried out in such a way that the above-mentioned positivity property holds. 
Third and finally, for any Dirac-type operator, we give (Theorem~\ref{thm: Hadamard_bisolution_Dirac_type_op}) a much more direct construction of Feynman parametrices satisfying a positivity property.
We also discuss some consequences, including the usual propagation of singularities theorems (Theorem~\ref{thm: propagation_singularity_Sobolev_WF}).
These are well known to hold for very general operators on vector bundles. 
For instance, the propagation of polarisation sets has been proven by Dencker using 
microlocalisation in the matrix setting~\cite{Dencker_JFA_1982}. 

As explained above the construction of Feynman propagators satisfying the positivity property is equivalent to the construction of Hadamard states.
There are by now several constructions of Hadamard states, recently, even in the 
analytic category~\cite{Gerard_CMP_2019}. 
Amongst the methods to construct them, there are direct ones using singularity expansions employing the 
Lorentzian distance function 
(the so-called Hadamard expansion)~\cite{Fulling_CMP_1978, Brown_JMP_1984, Sahlmann_RMP_2001, Marecki_master, Hollands_RMP_2008, Dappiaggi_RMP_2009, Lewandowski_JMP_2022},  
spectral methods that rely on frequency splitting 
(deformation method)~\cite{Fulling_AnnPhys_1981, Murro_AGAG_2021}, 
pseudodifferential methods~\cite{Junker_RMP_1996, Hollands_adiabatic_CMP_2001, Gerard_CMP_2014, Gerard_CMP_2015, Gerard, Gerard_bounded_geometry} 
(see also the recent monograph~\cite{Gerard_EMS_2019}), 
holography~\cite{Moretti_CMP_2008, Dappiaggi_JMP_2009, Dappiaggi_ATMP_2011, Gerard_AnalPDE_2016} 
(see also the exposition~\cite{Dappiaggi_Springer_2017}), 
and even 
global methods~\cite{GellRedman_CMP_2016, Vasy_AHP_2017, Derezinski_RMP_2018, Capoferri_JMAA_2020}.  
Many of these constructions can and have been generalised to the bundle case. 
Each comes with their advantages and disadvantages.
We are not trying to review the wealth of different methods but instead refer to the above literature for further references. 
% 
% 
% 
% 
% 
% 
% 
% 
% 
% 
% ================================================ 
\section{The setup and the main results} 
\label{sec: setup}
% ================================================ 
Suppose that $(M,g)$ is an $n \geq 2$-dimensional globally hyperbolic spacetime, i.e., an oriented and timeoriented Lorentzian manifold $M$ of metric $g$-signature $(-, +, \ldots, +)$ that possesses a smooth spacelike Cauchy hypersurface.  
Let $E \to M$ be a smooth complex vector bundle. 
A linear differential operator $\Box$ acting on smooth sections $\secE$ of $E$ is called \textit{normally hyperbolic} if its principal symbol $\symb{\square}$ is scalar and equals to the metric $g^{-1}$ on the cotangent bundle $T^{*} M$ over $M$, i.e., 
\begin{equation*}
	\symb{\square} \xxi := g_{x}^{-1} (\xi, \xi) \, \one_{\End{E}}   
\end{equation*} 
for any $\xxi \in  T^{*} M$. 
Then, there is a unique connection
$\nabla^{E}: C^\infty(M; E) \to C^\infty(M; E \otimes T^*M)$ and a unique potential $V \in C^\infty \big( M; \End{E} \big)$ such that $\Box$ takes the form 
(see e.g.~\cite[Prop. 3.1]{MR1418003},~\cite[Lem. 1.5.5 and 1.5.6]{Baer_EMS_2007}) 
\begin{equation} \label{eq: Weitzenboeck_connection_NHO}
	\Box = -\tr_{g} (\nabla^{T^{*} M \otimes E} \circ \nabla^{E}) + V, 
	\quad 
	\nabla^{T^{*} M \otimes E} := \nabla^{\mathrm{LC}} \otimes \one + \one \otimes \nabla^{E},  		
\end{equation}
where $\nabla^{\mathrm{LC}}$ is the Levi-Civita connection of $M$ and $\tr_{g} :  T^{*} M \otimes T^{*} M \to \R$ denotes the metric trace: $\tr_{g} \big( \xxi \otimes (x, \eta) \big) := g_{x}^{-1} (\xi, \eta)$. 
This formula is sometimes called the Weitzenb\"{o}ck formula and we will refer to the connection $\nabla^{E}$ as the \textit{Weitzenb\"{o}ck connection}.
If $E$ carries a bundle metric such that $\Box$ is formally selfadjoint then $\nabla^{E}$ is compatible with this metric.
We do not, however, assume this unless stated otherwise.

Recall that $C_{\mathrm{sc}}^{\infty} (M; E)$ denotes the space of spatially compact smooth sections of $E$. 
It is defined as the set of all $u \in \secE$ for which there exists a compact subset $K$ of $M$ such that $\supp{u} \subset J (K) := J^{+} (K) \cup J^{-} (K)$ where $J^\pm(K)$ denotes the causal future/past of $K$. 
It is a classical result that any normally hyperbolic operator on a globally hyperbolic spacetime admits unique retarded and advanced Green's operators $G^{\ret, \adv}: \comSecE \to C_{\mathrm{sc}}^{\infty} (M; E)$ with the characteristic property   
\begin{equation*}
	\mathrm{supp} \big( G^{\ret, \adv} (u) \big) \subset J^{\pm} \big( \supp{u} \big)
\end{equation*}
for any compactly supported smooth section $u \in \comSecE$ of $E$.  
Constructions in great generality are due to  
Hadamard~\cite{Hadamard_ActaMath_1908, Hadamard_Dover_2003} 
and to 
Riesz~\cite{Riesz_ActaMath_1949, Riesz_CPAM_1960}. 
Contemporary expositions include, for 
example~\cite{Friedlander_CUP_1975, Guenther_AP_1988, Baer_EMS_2007},  
and we refer here, for instance, the 
monographs~\cite[Cor. 3.1.4]{Baer_EMS_2007},~\cite[Prop. 4.1, Rem. 4.3 (b)]{Guenther_AP_1988}  
for a normally hyperbolic operator. 
We would also like to mentioned the treatments~\cite{Schapira_MPAG_2017, Jubin_LMP_2016} for the Cauchy problem on an analytic globally hyperbolic spacetime. 

Throughout this paper we will not distinguish notationally between a continuous linear map $C_{\mathrm{c}}^{\infty} \to \cD'$ and its distributional kernel. 
This means that the Green's operators $\advGreenOp, \retGreenOp$ are also to be understood as distributions in $\mathcal{D}'(M \times M; E \boxtimes E^*)$. 
On a general manifold $M$, operators will be thought of as maps between halfdensities. 
This allows pairings between functions without using a volume form. 
In this way, the space of distributions $\cD' (M; E \otimes \varOmega^{\nicefrac{1}{2}})$ is defined as the topological dual of $C_{\mathrm{c}}^{\infty} (M; E^{*} \otimes \varOmega^{\nicefrac{1}{2}})$, where $\varOmega^{\nicefrac{1}{2}} \to M$ is the bundle of halfdensities over $M$. 
In case a Lorentzian (resp. Riemannian) metric is given on $M$, we use the natural Lorentzian (resp. Riemannian) volume form to identify halfdensities with functions, i.e., in this case,  $\varOmega^{\nicefrac{1}{2}}$ has a canonical trivialisation in terms of the Lorentzian (resp. Riemannian) volume form. 
In both situations, we will be notationally suppressing $\varOmega^{\nicefrac{1}{2}}$, so all functions and sections will be considered halfdensity-valued.
We will denote the zero section removed cotangent bundle resp. the (co)lightcone bundle by $\dotCoTanM$ resp. $\coLightBun$, and use Einstein's summation convention.  

Our convention of Fourier transform is 
\begin{equation*}
	\hat{f} (\theta) := \int_{\Rn} f (x) \, \re^{ - \ri x \cdot \theta} \rd x
\end{equation*}
for any $f \in L^{1} (\Rn, \rd x)$, where $\cdot$ is the Euclidean inner product. 
Suppose that $U \subset \Rn$ is a non-empty open set.
Recall that $\dot{T}^{*} U \ni \xxiNot \notin \WF{u}$ is not in the wavefront set $\WF{u}$ of a distribution $u \in \mathcal{D}' (U; E)$ if and only if there exists a compactly supported section $f$ on $U$ non-vanishing at $x_{0} \in U$ such that the Fourier transform of $\widehat{fu}(\xi)$ is rapidly decreasing in a conic neighbourhood of $ \xi^{0}$~\cite[Prop. 2.5.5]{Hoermander_ActaMath_1971}. 
In fact, this definition makes sense on smooth manifolds as the wavefront set transforms covariantly under a change of coordinates. 
Details are available, for instance, in the 
monograph~\cite[Chap. VIII]{Hoermander_Springer_2003} 
and in the 
exposition~\cite{Strohmaier_Springer_2009} 
customised for quantum field theory on curved spacetimes. 
If $u \in \mathcal{D}' (M; E)$ then $\WF{u} \subset \dot T^* M$. 
The \textit{twisted wavefront set} of a bidistribution $A \in \cD'(M \times N ; E \boxtimes F^*)$ is given by 
\begin{equation*}
	\WFPrime{A} := \{ (x_{0}, \xi^{0}; y_{0}, - \eta^{0}) \in \dotCoTanM \times \dotCoTanN | \xxiyetaNot \in \WF{A} \}, 
\end{equation*} 
where $F \to N$ is a smooth complex vector bundle over a manifold $N$.  
The twisted wavefront  $\WFPrime{P}$ of the Schwartz kernel of a pseudodifferential operator $P$ is a subset of the diagonal in  $\dot T^* M \times  \dot T^* M$.  
Then, $\ES{P} := \{ \xxiNot \in \dotCoTanM | (x_{0}, \xi^{0}; x_{0}, \xi^{0}) \in  \WFPrime{P} \}$ is called the 
\textit{essential support}\footnote{It 
	is also known as the \textit{microsupport} or \textit{wavefront set} of a pseudodifferential operator.
} 
of $P$. 
This is the smallest conic set such that $P$ is of order $-\infty$ in $\dotCoTanM \setminus \ES{P}$ 
(see e.g.~\cite[Prop. 18.1.26]{Hoermander_Springer_2007}). 

Throughout the article, only the polyhomogeneous symbol class $S^{m}$ (see Definition~\ref{def: polyhomogeneous_symbol_mf}) will be used, so $\Psi \mathrm{DO}^{m}$ (resp. $I^{m}$) will be the set of pseudodifferential operators (resp. Lagrangian distributions) having polyhomogeneous total symbols.  
We will denote by $H_{\mathrm{loc}}^{s} (M; E)$, the space of sections on $E$ that are locally in the Sobolev space $H^{s}$ with respect to any smooth chart and smooth bundle chart. 
The space of sections in $H_{\mathrm{loc}}^{s} (M; E)$ of compact support is denoted by $H_{\mathrm{c}}^{s} (M; E)$. 
As usual, the space $H_{\mathrm{loc}}^{s} (M; E)$ is equipped with the locally convex topology of convergence locally in $H^s (M; E)$. 
The space $H_{\mathrm{c}}^{s} (M; E)$ is the union $\bigcup_{K \Subset M} H_{\mathrm{c}}^{s} (K; E)$ and is equipped with the inductive limit topology, where the union runs over all compact subset $K$ of $M$. 
For details, we refer, for example~\cite[App. B1]{Hoermander_Springer_2007}. 
% 
% 
% 
% 
% 
% 
% 
% 
% 
% 
% ================================================ 
\subsection{Feynman parametrices and propagators}
% ================================================  
A parametrix which is uniquely characterised by its wavefront set is called a 
distinguished parametrix~\cite[Sec. 6.6]{Duistermaat_ActaMath_1972}. 

% 
% 
% 
% ================================================  
\begin{definition}[Feynman parametrices] 
\label{def: Feynman_parametrix}
	Let $(M, g)$ be a globally hyperbolic spacetime and $\dotCoTanM, \coLightBun \to M$ the punctured cotangent bundle and the lightcone bundle over $M$, respectively. 
	The Feynman $\FeynParametrix$ and the anti-Feynman $\antiFeynParametrix$ parametrices of a normally hyperbolic operator acting on smooth sections of a vector bundle $E \to M$ over $M$, are parametrices $\FeynParametrix, \antiFeynParametrix : \comSecE \to \secE$ whose Schwartz kernels satisfy~\cite[p. 229]{Duistermaat_ActaMath_1972} 
	(see also~\cite[pp. 541-542]{Radzikowski_CMP_1996}) 
	\begin{equation} \label{eq: def_Feynman_parametrix}
		\WFPrime{\FeynParametrix} \subset \varDelta \,\dotCoTanM  \cup C^{+}, 
		\qquad 
		\WFPrime{\antiFeynParametrix} \subset \varDelta \,\dotCoTanM \cup C^{-},  
	\end{equation}
	where $\varDelta \, \dotCoTanM  := \{ (x, \xi; x, \xi) \in \dotCoTanM \times \dotCoTanM \}$ is the diagonal in $\dotCoTanM \times \dotCoTanM$,  
	\begin{equation} \label{eq: def_Feynman_bicharacteristic_relation_NHOp}
		C^{\pm} := \big\{ (x, \xi; y, \eta) \in \coLightBun \times \coLightBun | \exists s \in \R_{\gtrless 0} : \xxi = \varPhi_{s} (y, \eta) \big\}  
	\end{equation}
	are the forward and backward geodesic relations on $\coLightBun \times \coLightBun$, and $\varPhi_{s}$ is the "time $s$" geodesic flow on $T^{*} M$ restricted to $\coLightBun$ where $s$ is the flow parameter. 
\end{definition}
% ================================================ 
% 
% 
% 

In this exposition, we give a vector bundle version of the classic result by 
Duistermaat-H\"{o}rmander~\cite[Thm. 6.5.3]{Duistermaat_ActaMath_1972} 
(see also~\cite[Thm. 26.1.14]{Hoermander_Springer_2009}) 
for normally hyperbolic operators. 

% 
% 
% 
% ================================================ 
\begin{theorem}[Existence and uniqueness of Feynman parametrices] 
\label{thm: exist_unique_Feynman_parametrix_NHOp}
	Let $E \to M$ be a smooth complex vector bundle over a globally hyperbolic spacetime $(M, g)$ and $\square$ a normally hyperbolic operator on $E$. 
	Then, there exist unique (modulo smoothing kernels) Feynman $\FeynParametrix$ and anti-Feynman $\antiFeynParametrix$ parametrices of $\square$. 
	Moreover, for every $k \in \R$, $\FeynParametrix$ and $\antiFeynParametrix$ extend to   continuous maps from  $H_{\mathrm{c}}^{k} (M; E)$ to  $H_{\mathrm{loc}}^{k+1} (M; E)$ with   
	\begin{equation} \label{eq: diff_Feyn_anti_Feyn_parametrix}
		\FeynParametrix - \antiFeynParametrix \in I^{-3/2} \big( M \times M, C'; \Hom{E, E} \big), 
	\end{equation} 
	where $I^{-3/2} \big( M \times M, C'; \Hom{E, E} \big)$ is the space of Lagrangian distributions (see Definition~\ref{def: Lagrangian_distribution_bundle_local}) associated to the geodesic relation 
	\begin{equation} \label{eq: def_bicharacteristic_relation_NHOp}
		C' := \{ (x, \xi; y, -\eta) \in \coLightBun \times \coLightBun | \exists s \in \R : \xxi = \varPhi_{s} (y, \eta) \}, 
	\end{equation}
	where $\varPhi_{s}$ is the "time $s$" geodesic flow on the cotangent bundle restricted to the lightcone bundle $\coLightBun \to M$. 
	Furthermore, $\FeynParametrix - \antiFeynParametrix$ is non-characteristic 
	(see Definition~\ref{def: elliptic_FIO}) 
	at every point of $C$.  
\end{theorem}
% ================================================
%
%
%

A special case of this result for the massive Klein-Gordon operator was given by 
Radzikowski~\cite[Prop. 4.2-4.4]{Radzikowski_CMP_1996} 
as a direct consequence of~\cite[Thm. 6.5.3]{Duistermaat_ActaMath_1972}. 
Employing the distinguished global phase function approach of Fourier integrals operators~\cite{Laptev_CPAM_1994}, 
Capoferri \textit{et al}.~\cite[Thm. 5.2]{Capoferri_JMAA_2020} 
have constructed these parametrices for scalar wave operators with time-independent smooth potential in spatially compact globally hyperbolic ultrastatic spacetimes.  
Recently, 
Lewandowski~\cite[Prop. 3.5]{Lewandowski_JMP_2022} 
has published a direct construction for a real vector bundle utilising the Hadamard series expansion along with the presentation of~\cite{Baer_EMS_2007}. 
In contrast to these, our proof is purely microlocal as in the original 
treatment~\cite[Thm. 6.5.3]{Duistermaat_ActaMath_1972}.
This only requires a bundle version of microlocalisation as developed in due course (Theorems~\ref{thm: microlocalisation_P} and~\ref{thm: microlocalisation_NHOp}), which is along the lines of 
Dencker's~\cite{Dencker_JFA_1982} 
proof of propagation of the polarization sets albeit our presentation is more geometric. 

In fact, the following positivity property can be shown.

% 
% 
% 
% ================================================ 
\begin{restatable}{proposition}{positivityFeynmanMinusAdvNHOp} 
\label{thm: positivity_Feynman_minus_adv_NHOp}
	Let $E \to M$ be a smooth complex vector bundle over a globally hyperbolic spacetime $(M, g)$ and $\square$ a normally hyperbolic operator on $E$. 
	Suppose that $E$ is endowed with a (non-degenerate) sesquilinear form $(\cdot|\cdot)  \in C^\infty(\overline{E^*} \otimes E^*)$ such that $\square$ is formally selfadjoint with respect to $(\cdot|\cdot)$. 
	Then, there exists a Feynman parametrix $\FeynParametrix$ of $\square$ such that $W := - \ri (\FeynParametrix - \advGreenOp)$ is formally selfadjoint. 
	Additionally, if $(\cdot|\cdot)$ is positive-definite (hermitian) then $W$ can be chosen non-negative.
\end{restatable}
% ================================================ 
% 
% 
% 

This is essentially a bundle version of that in 
Duistermaat-H\"{o}rmander~\cite[Thm. 6.6.2]{Duistermaat_ActaMath_1972} 
which was proven by deploying a refined microlocalisation of scalar pseudodifferential operators~\cite[Lem. 6.6.4]{Duistermaat_ActaMath_1972}. 
We provide such a microlocalisation for $\square$ in~\eqref{eq: microlocalisation_NHOp_refined}. 

Finally, one can turn the Feynman parametrix $\FeynParametrix$ into a Feynman propagator $\FeynGreenOp$ utilising the well-posedness of the Cauchy problem for $\square$ on a globally hyperbolic spacetime. 

% 
% 
% 
% ================================================ 
\begin{restatable}[Existence of Hadamard bisolutions]{theorem}{existenceFeynmanGreenOpNHOp} 
\label{thm: existence_Feynman_propagator_NHOp}
	Let $E \to M$ be a smooth complex vector bundle over a globally hyperbolic spacetime $(M, g)$ and $\square$ a normally hyperbolic operator on $E$ that is formally selfadjoint with respect to the non-degenerate sesquilinear form $(\cdot|\cdot) \in C^\infty(\overline{E^*} \otimes E^*)$. 
	Then, there exists a Feynman propagator $\FeynGreenOp$ for $\square$ such that $\omega := - \ri (\FeynGreenOp - \advGreenOp)$ is formally selfadjoint with respect to $(\cdot|\cdot)$.   
	In addition, if $(\cdot|\cdot)$ is hermitian then $\FeynGreenOp$ can be chosen such that $\omega (\bar{u}^{*} \otimes u) \geq 0$ for any $u \in \comSecE$. 
	Thus, $\omega$ defines a Hadamard state. 
	Here $\secE \ni u \mapsto \bar{u}^{*} \in C^{\infty} (M; \bar{E}^{*})$ is the fibrewise linear mapping induced by $(\cdot|\cdot)$.
\end{restatable}
% ================================================ 
% 
% 
% 

Duistermaat-H\"{o}rmander~\cite[p. 229]{Duistermaat_ActaMath_1972} 
have actually considered a much wider class of operators, namely scalar pseudodifferential operators of real principal type on a smooth manifold. 
The pivotal step of determining the appropriate smoothing operators required to obtain Feynman propagators from respective parametrices was, however, left open.
This indeterminacy can be fixed in various ways on special spacetimes even in the absence of the timelike Killing vector field. 
Such constructions have appeared in the literature on microlocal analysis. 
For example, 
Gell-Redman \textit{et al}.~\cite[Thm. 3.6]{GellRedman_CMP_2016} 
have treated the scalar wave operator in spaces with non-trapping Lorentzian scattering matrices.  
Vasy has constructed Feynman propagators by making assumptions on global dynamics for 
(i) scalar formally selfadjoint operators with real principal symbol in closed manifolds~\cite[Thm. 1]{Vasy_AHP_2017}, 
and in spaces of Lorentzian scattering matrices for 
(ii) wave operators in Melrose's b-pseudodifferential algebraic framework~\cite[Thm. 7]{Vasy_AHP_2017} 
and 
(iii) Klein-Gordon operators in Melrose's scattering pseudodifferential algebraic formalism~\cite[Thm. 10 and 12]{Vasy_AHP_2017}.
His idea is to identify the appropriate spaces where these operators are invertible and then define the Feynman propagators as the inverse of those operators satisfying the required properties --- a generalisation of Feynman's original "$\pm \ri \epsilon$" prescription. 
As a consequence, he has also achieved respective positivity properties for Feynman parametrices~\cite[Cor. 5, 9, 11,13]{Vasy_AHP_2017},  in the same spirit as the ones  
obtained by 
Duistermaat-H\"{o}rmander~\cite[Thm. 6.6.2]{Duistermaat_ActaMath_1972}. 
A special case of Theorem~\ref{thm: existence_Feynman_propagator_NHOp} for the Klein-Gordon operator minimally coupled to a static electromagnetic potential on a static spacetime has been proven by 
Derezi\'{n}ski-Siemssen~\cite[Thm. 7.7]{Derezinski_RMP_2018}. 
In the spirit of the limiting absorption principle, they have shown that the Feynman propagator can be considered as the boundary value of the resolvent of the Klein-Gordon operator.
Assuming Proposition~\ref{thm: positivity_Feynman_minus_adv_NHOp} and ideas from~\cite[Sec. 3.3]{Gerard_CMP_2015}, 
Lewandowski~\cite[Thm. 4.3]{Lewandowski_JMP_2022} 
has recently given a construction of Hadamard states for Riemannian vector bundles. 

Since the square of any Dirac-type operator is normally hyperbolic, Theorem~\ref{thm: existence_Feynman_propagator_NHOp} can be used to show that Hadamard states exist for such an operator whenever the underlying vector bundle is equipped with a hermitian form. 
However, the existence of a hermitian inner product on which Dirac-type operators are formally selfadjoint is rather exceptional. 
The natural inner product on spinors on a spacetime is not positive-definite rather \textit{indefinite} and therefore the positivity property for Feynman propagators cannot be concluded directly for a Dirac-type operator on a globally hyperbolic spin-spacetime employing Theorem~\ref{thm: existence_Feynman_propagator_NHOp}. 
Let $\big( E \to M, (\cdot|\cdot) \big)$ be a vector bundle over a globally hyperbolic (not necessarily spin) spacetime $(M, g)$, endowed with a sesquilinear form $(\cdot|\cdot)$ and let  $D$ be a Dirac-type operator on $E$ that is formally selfadjoint with respect to this sesquilinear form. 
Suppose that $\Sigma$ is a smooth spacelike Cauchy hypersurface of $M$ and that $N$ is a future-directed unit normal vector field on $M$ along $\Sigma$. 
We assume that 
\begin{equation} \label{eq: positive_definite_sesequilinear_form_Dirac_type_op}
	\langle \cdot | \cdot \rangle := \big( \symb{D} (N^{\flat}) \cdot \big| \cdot \big)
\end{equation}
is a fibrewise hermitian form on $E$, where $\symb{D}$ is the principle symbol of $D$ and ${\cdot}^{\flat} : C^{\infty} (M; TM) \to C^{\infty} (M; T^{*} M)$ is the unique linear pointwise bijection. 
We now provide another construction of Feynman propagators $\mathsf{S}^{\Feyn}$ for $D$ together with positivity, employing a direct microlocal decomposition (see~\eqref{eq: microlocal_decomposition_causal_propagator_Dirac_type_op}) of the Pauli-Jordan operator for $D$, in 

% 
% 
% 
% ================================================ 
\begin{restatable}{theorem}{HadamardBisolutionDiracTypeOp}
	\label{thm: Hadamard_bisolution_Dirac_type_op}	
	Let $\big( E \to M, (\cdot|\cdot), \symb{D} \big)$ be a smooth bundle of Clifford modules over a globally hyperbolic spacetime $(M, g)$ and $D$ a Dirac-type operator on $E$ that is formally selfadjoint with respect to the sesquilinear form $(\cdot|\cdot)$. 
	Then, there exists a Feynman propagator $\mathsf{S}^{\Feyn}$ for $D$ such that $\omega := - \ri (\mathsf{S}^{\Feyn} - S^{\adv})$ is formally selfadjoint. 
	In addition, if there exists a hermitian form $\langle \cdot | \cdot \rangle$ on $E$ satisfying~\eqref{eq: positive_definite_sesequilinear_form_Dirac_type_op}, then $\mathsf{S}^{\Feyn}$ can be chosen such that $\omega$ is non-negative with respect to $(\cdot|\cdot)$ and hence defines a Hadamard bisolution of $D$.
\end{restatable}
% ================================================ 
% 
% 
% 
% 
% 
% 
% 
% 
% 
% ================================================ 
\section{Microlocalisation}
% ================================================ 
\subsection{Fourier integral operators}  
\label{sec: FIO}
% ================================================ 
Fourier integral operators originated from the study of the singularities of solutions of hyperbolic differential equations by 
Lax~\cite{Lax_DukeMathJ_1957} 
and in the context of geometrical optics by 
Maslov~\cite{Maslov_Moscow_1965}. 
Their local formulation has been later systematically developed and globalised in the seminal articles by    
H\"{o}rmander~\cite{Hoermander_ActaMath_1971} 
and by 
Duistermaat and H\"{o}rmander~\cite{Duistermaat_ActaMath_1972} 
for scalar operators. 
The vector bundle version is available in  
H\"{o}rmander's monograph~\cite{Hoermander_Springer_2009}.  
In this report, we will adopt a symbolic calculus viewpoint, as described below.  

Let $E \to M, F \to N$ be two smooth complex vector bundles over smooth manifolds $M, N$ and let $C$ be a homogeneous canonical relation from the punctured cotangent bundle $\dotCoTanN$ of $N$ to that $\dotCoTanM$ of $M$. 
In a nutshell, a Fourier integral operator associated with $C$ is a continuous linear map from $C_{\mathrm{c}}^{\infty} (N; F)$ to $\dualComSecE$ whose Schwartz kernel $A$ is a Lagrangian distribution~\cite[Def. 25.2.1]{Hoermander_Springer_2009}. 
We will now explain briefly what this means. 
Suppose that $\varLambda \subset \dotCoTanM$ is a smooth, closed and conic Lagrangian submanifold.  
The space $I^m (M, \Lambda)$ of Lagrangian distributions of order $m \in \R$ can be defined with respect to local coordinate charts. 
To do this, we first explain the concept of a clean phase function~\cite[p. 71]{Duistermaat_JInventMath_1975}. 
The reader is referred to the 
monograph~\cite[pp. 414-428]{Treves_Plenum_1980} 
for more details.

% 
% 
% 
% ================================================  
\begin{definition} \label{def: clean_phase_function}
	Let $M$ be a smooth manifold and $d$ a natural number not necessarily equal to the dimension $n$ of $M$. 
	A real-valued smooth function $\varphi$ on an open conic set $\cU \subset M \times \dot{\R}^{d}$, homogeneous of degree one in $\theta \in \dot{\R}^{d}$, is called a clean phase function of excess $0 \leq e \leq n$ if $\rd \varphi \neq 0$ and its fibre-critical set 
	\begin{equation*}
		\varSigma := \big\{ \xthetaNot \in \cU ~|~ \grad_{\theta} \varphi \,  \xthetaNot = 0 \big\}
	\end{equation*} 
	is an $n+e$-dimensional smooth submanifold of $M \times \dot{\R}^{d}$, whose tangent space is given by the kernel of $\rd (\grad_{\theta} \varphi)$. 
	The number of linearly independent differentials $\rd (\partial \varphi / \partial \theta_{i}), i = 1, \dots, d$ is equal to $d - e$ on $\varSigma$.  
	The phase function $\varphi$ is called non-degenerate if $e = 0$    
	(see e.g.~\cite[Def. 21.2.15]{Hoermander_Springer_2007}).
\end{definition}
% ================================================ 
% 
% 
% 

A clean phase function on an open set $\cU \subset U \times \dot{\R}^{d}$ where $U \subset \Rn$, defines an immersed conic Lagrangian submanifold $\varLambda_\varphi \subset \dot{T}^*U$ via the homogeneous Lagrangian fibration~\cite[Lem. 7.1]{Duistermaat_JInventMath_1975}
(see also, e.g.~\cite[pp. 416-417]{Treves_Plenum_1980},~\cite{Meinrenken_ReptMathPhys_1992}) 
\begin{equation} \label{eq: map_critical_set_Lagrangian_submf}
	\varSigma \ni \xtheta \mapsto (x, \rd_{x} \varphi) \in \varLambda_{\varphi}.
\end{equation}
A Lagrangian distribution $u \in I^m(U, \varLambda_\varphi)$ is, by definition, a distribution $u \in \cD' (U)$ locally given by an oscillatory integral of the form~\cite[(1.2.1) and Def. 3.2.2]{Hoermander_ActaMath_1971} 
(see also, e.g.~\cite[Prop. 25.1.5']{Hoermander_Springer_2009},~\cite[pp. 433-439]{Treves_Plenum_1980}) 
\begin{equation} \label{eq: def_IML_local}
		u(x) := (2\pi)^{-(n + 2d - 2e) / 4} \int_{\Rd} \re^{\ri \varphi \xtheta} \, a \xtheta \, \rd \theta, 
\end{equation}
where $\rd \theta$ is the Lebesgue measure on $\R^{d}$ and $a \in S^{m + (n - 2d - 2e) /4} (U \times \dot{\R}^{d})$ is a symbol (see Definition~\ref{def: symbol_mf}) with support in the interior of a sufficiently small  conic neighbourhood of $\varSigma$ contained in the domain of definition of $\varphi$. 
It then follows that $\WF{u} \subset \varLambda_\varphi$~\cite[Thm. 3.2.6]{Hoermander_ActaMath_1971} 
(see also, e.g.~\cite[Prop. 3.1 (p. 438)]{Treves_Plenum_1980}). 

As usual, the oscillatory integral is to be understood as a formal expression that does not make sense for a fixed $x$ rather it defines a distribution in the sense that for any test function 
$f \in C^\infty_{\mathrm{c}} (U)$ we have
\begin{equation*}
	u(f) = (2\pi)^{-(n + 2d - 2e) / 4} \int_{\Rd} \int_{U} \re^{\ri \varphi \xtheta} \, a \xtheta \, f(x) \, \rd x \, \rd \theta.  
\end{equation*}

% 
% 
% 
% ================================================  
\begin{definition} \label{def: Lagrangian_distribution_bundle_local}
	Let $E \to M$ be a smooth complex vector bundle over a smooth manifold $M$. 
	Suppose that $\varLambda$ is a smooth, closed and conic Lagrangian submanifold of the punctured cotangent bundle $\dotCoTanM$ of $M$.  
	A distribution $u \in \mathcal{D}' (M; E)$ is called an element in the space $I^{m} (M, \varLambda; E)$ of Lagrangian distributions of order $m \in \R$, if it can be written as 
	\begin{equation*}
		u = \sum_{\alpha} u_{\alpha} 
	\end{equation*}
	with locally finite $\supp{u_{\alpha}} \subset U_{\alpha} \subset M$ such that the components $(u^{1}, \ldots, u^{\rk E})$ of $u_{\alpha}$ with respect to a local chart and a bundle chart are in $I^{m} (U_{\alpha}, \varLambda_{\varphi_\alpha})$ with $\varLambda_{\varphi_\alpha} \subset \varLambda$.	
\end{definition}
% ================================================ 
% 
% 
% 

The above of course implies that~\cite[Lem. 25.1.2]{Hoermander_Springer_2009} 
\begin{equation} \label{eq: WF_Lagrangian_dist}
	\WF{u} \subset \varLambda 
\end{equation} 
for any $u \in I^{m} (M, \varLambda; E)$. 
The original 
definition~\cite[Sec. 1.4]{Hoermander_ActaMath_1971} 
was based on representations with non-degenerate phase functions. 
Subsequently, it has been 
globalised~\cite[Def. 3.2.2]{Hoermander_ActaMath_1971} 
and extended for sections of vector bundles~\cite[Def. 25.1.1]{Hoermander_Springer_2009}.  
We refer to the monographs~\cite[pp. 4-10]{Hoermander_Springer_2009},~\cite[pp. 433-439]{Treves_Plenum_1980}  
for a systematic study.  

% 
% 
% 
% ================================================ 
\begin{remark} \label{rem: parametrisation_FIO_canonical_relation}
	The parametrisation of a generic Lagrangian submanifold, $\varLambda_{\varphi} = {\{} (x, \rd_{x} \varphi) {\}}$ by a phase function $\varphi$ cannot be performed globally due to topological restriction given by some cohomology class, known as the Maslov class. 
	Furthermore, a phase function does not uniquely characterise a Lagrangian submanifold, rather one must consider an equivalence class of phase functions satisfying certain necessary and sufficient conditions, as originally analysed for 
	non-degenerate phase functions~\cite[Thm. 3.1.6]{Hoermander_ActaMath_1971},  
	albeit it holds true more generally for clean phase functions. 

	Technically speaking, let $\tilde{\cU} \subset M \times \dot{\R}^{\tilde{d}}$ be an open conic set such that there exists a diffeomorphism $\upsilon: \cU \in \xtheta \mapsto (x; \tilde{\theta}) \in \tilde{\cU}$ which is homogeneous with respect to $\theta$ of degree one and fibre-preserving, where $\tilde{\theta}$ is a smooth function of $\xtheta$.  
	Then, $\varphi$ is said to be (locally) \textit{equivalent} to a phase function $\tilde{\varphi}$ on $\tilde{\cU}$ if $\upsilon^{*} \tilde{\varphi} = \varphi$~\cite[p. 134]{Hoermander_ActaMath_1971}. 
	If $\varphi$ resp. $\tilde{\varphi}$ are two clean phase functions in conic neighbourhoods of their fibre-critical points $\xthetaNot \in U \times \dot{\R}^{d}$ resp. $(x_{0}, \tilde{\theta}^{0}) \in U \times \dot{\R}^{\tilde{d}}$ with the same excess $e$, then they are equivalent in some open conic neighbourhoods of these points if and only if~\cite[Prop. 1.5 (p. 421)]{Treves_Plenum_1980} 
	\begin{subequations} \label{eq: necessarily_sufficient_equivalent_clean_phase_function}
		\begin{eqnarray} 
			\varSigma_{\varphi} \ni \xtheta \mapsto (x, \rd_{x} \varphi) 
			& = & 
			(x, \rd_{x} \tilde{\varphi}) \mapsfrom (x, \tilde{\theta})  \in \varSigma_{\tilde{\varphi}}, 
			\\ 
			d 
			& = & 
			\tilde{d}, 
			\\ 
			\sgn \big( \Hess_{\theta} \varphi \, \xtheta \big) 
			& = & 
			\sgn \big( \Hess_{\tilde{\theta}} \varphi \, (x; \tilde{\theta}) \big),   
		\end{eqnarray}
	\end{subequations}
	where $\sgn$ denotes the signature of the Hessian matrix $\mathrm{Hess}_{\theta}$ (resp. $\mathrm{Hess}_{\tilde{\theta}}$) with respect to the fibre variable $\theta$ (resp. $\tilde{\theta}$). 

	When $M$ is compact, the aforementioned obstructions require cohomological and $k$-theoretic language to formulate which are somehow tangential to the subject matter. 
	So we refrain those precise formulae and refer~\cite{Latour_ASENS_1991} together with the earlier references cited therein. 
	The non-compact case is still an open issue.  
	Notwithstanding, a global parametrisation can be achieved, for instance, by an equivalence class of complex-valued non-degenerate phase functions in the particular case whenever conic Lagrangian submanifolds are given by the graphs of homogeneous symplectomorphisms~\cite[Lem. 1.2 and 1.7]{Laptev_CPAM_1994}.
\end{remark}
% ================================================ 
% 
% 
% 

In order to define the class of Fourier integral operators, one substitutes the manifold $M$ by a product manifold $M \times N$ and replaces the Lagrangian submanifold $\varLambda$ by a homogeneous canonical relation $C \subset \dotCoTanM \times \dotCoTanN$ which is closed in $\dotCoTanMN$. 
In addition, the role of the vector bundle $E$ will be played by the vector bundle $E \boxtimes F^{*} \to M \times N$. 
We will usually abbreviate  $E \boxtimes F^{*} \to M \times N$ by $\Hom{F, E}$ when it is clear from the context what the base manifold is. 
A Fourier integral operator $C_{\mathrm{c}}^{\infty} (N; F) \to \cD' (M; E)$ associated with $C$, of order $m \in \R$, is defined as the continuous linear mapping whose Schwartz kernel $A$ is an element in $I^{m} \big( M \times N, C'; \Hom{F, E} \big)$. 
As with the wavefront set the ``twisted''  relation $C'$ is defined by $(x,\xi;y,\eta) \in C' \Leftrightarrow (x,\xi;y,-\eta) \in C$.
The bidistribution $A$ can be described in local coordinates as follows. 
Given a local chart and a local trivialisation, we can locally identify $A$ with a $\rk E \times \rk F$-matrix of entries $A_k^r \in I^{m} (U \times V, C'_{\varphi})$, where $U \subset \R^{\nM}, V \subset  \R^{\nN}$ are appropriate open subsets and~\cite[p. 134]{Hoermander_ActaMath_1971},~\cite[Lem. 7.1]{Duistermaat_JInventMath_1975} 
(see also, e.g.~\cite{Meinrenken_ReptMathPhys_1992},~\cite[(21.2.9)]{Hoermander_Springer_2007})   
\begin{equation} \label{eq: local_twisted_canonical_relation_phase_function}
	C'_{\varphi} = \{ \left( x, \rd_{x} \varphi; y, \rd_{y} \varphi \right) \in \dotCoTanM \times \dotCoTanN | (x, y; \theta) \in \varSigma \}, 
	\quad 
	\varSigma := (\grad_{\theta} \varphi)^{-1} (0) 
\end{equation}
is a local representative of $C$, generated by a clean phase function $\varphi$ on $U \times V \times \dot{\R}^{d}$ with excess $e$. 
By definition, 
\begin{equation}
	\WFPrime{A} \subset C, \quad  \WFPrime{A_k^r} \subset C_{\varphi}. 
\end{equation}
In a small conic neighbourhood of $\xxiyetaNot \in C'_{\varphi}$, the singularity of $A_k^r$ is described by a model of the form~\cite[Prop. 25.1.5']{Hoermander_Springer_2009}  
\begin{equation} \label{eq: 25_1_3_clean_Hoemander}
	\tilde A_{k}^{r} (x, y) \equiv (2\pi)^{-(\nM + \nN + 2d - 2e) / 4} \int_{\Rd} \re^{\ri \varphi \xytheta} a_{k}^{r} \xytheta \, \rd \theta 
\end{equation}
in the sense that in local coordinates $\xxiyetaNot \notin \WF{A_k^r - \tilde A_k^r}$. 

We list some rudimentary properties of Lagrangian distributions below. 
% 
% 
% 
% 
% 
% 
% 
% 
% 
% 
% ================================================ 
\subsubsection{Principal symbol} 
\label{sec: principal_symbol_FIO}
% ================================================ 
Loosely speaking, this is the highest order contribution in the asymptotic sense for a Lagrangian distribution $A \in I^{m} (M, \varLambda)$. 
The concept was introduced by 
H\"{o}rmander~\cite[Thm. 3.2.5]{Hoermander_ActaMath_1971} 
for vanishing excess. 
Subsequently, its generalisation for non-zero excess was given by
Duistermaat and Guillemin~\cite[Lem. 7.2]{Duistermaat_JInventMath_1975}. 
Later, a geometric description of principal symbol of an arbitrary distribution on a manifold was provided by 
Weinstein~\cite{Weinstein_BullAMS_1976, Weinstein_TransAMS_1978} 
who has also extended H\"{o}rmander's analysis for vector bundles, $I^{m} (M, \varLambda; E)$. 
A compendia of these results is availbale 
in the 
monographs~\cite[pp. 13-16]{Hoermander_Springer_2009},~\cite[pp. 439-450]{Treves_Plenum_1980} 
(see also, e.g.~\cite{Meinrenken_ReptMathPhys_1992}). 
We start with the notation 
\begin{equation*}
	I^{m - [1]} (\ldots) 
	:= 
	I^{m} (\ldots) / I^{m - 1} (\ldots).
\end{equation*}  
The principal symbol map is be defined by the 
isomorphism~\cite[Thm. 25.1.9]{Hoermander_Springer_2009} 
\begin{subequations}
	\begin{equation} \label{eq: def_symbol_map}
		\sigma : I^{m - [1]} \big( M \times N, C'; \Hom{F, E} \big) \to S^{m + \frac{\nM + \nN}{4} - [1]} \big( C; \Maslov \otimes \widetilde{\mathrm{Hom}} (F, E) \big), [A] \mapsto \big[ \symb{[A]}\big], 
	\end{equation}
	where $\widetilde{\mathrm{Hom}} (F, E) \to C$ is the pullback of the bundle $\Hom{F, E} \to M \times N$ to $T^{*} (M \times N)$ followed by restriction to $C$, $\Maslov \to C$ is the Keller-Maslov bundle (see Appendix~\ref{sec: FIO_symplecto_bundle}) over $C$, and 
	$S^{m + (\nM + \nN) / 4} \big( C; \Maslov \otimes \widetilde{\mathrm{Hom}} (F, E) \big)$ 
	is the $\Maslov \otimes \widetilde{\mathrm{Hom}} (F, E)$-valued symbol space (see Appendix~\ref{sec: symbol}) on $C$. 
	The principal symbol is then given explicitly in local coordinates below.  	

	Since $A \equiv (A_{k}^{r})$, the principal symbol $\symb{A}$ is locally identified with the matrix $(\symb{A})^{r}_{k}$. 
	Therefore, the result in~\cite[Prop. 25.1.5', p. 15]{Hoermander_Springer_2009}  
	for scalar Lagrangian distributions translates to vector bundles as 
	\begin{eqnarray}
		(\symb{A})_{k}^{r} \xxiyeta 
		& := &  
		\sqrt{|\rd x| |\rd \xi|} \int_{\mathfrak{C}_{\xi, \eta}} \mathsf{a}_{k}^{r} (x, y; \theta', \theta'')  \dfrac{\re^{\ri \pi \sgn (\Hess_{x,y; \theta'} \varphi) / 4}}{\sqrt{|\det (\Hess_{x,y; \theta'} \varphi)|}} \rd \theta'' 
		\nonumber \\ 
		&& 
		\mod S^{m + \frac{\nM + \nN}{4} - 1} (C_{\varphi}; \Maslov_{\varphi}), 
	\end{eqnarray}
\end{subequations}
where $\mathsf{a}_{k}^{r}$ is the top-order homogeneous term of $a$ in~\eqref{eq: 25_1_3_clean_Hoemander}, $\Hess_{x,y; \theta'} \varphi$ is the Hessian of the phase function $\varphi$ parametrising $C_{\varphi} := \eqref{eq: local_twisted_canonical_relation_phase_function}$, and 
$
	\mathfrak{C}_{\xi, \eta} 
	:= 
	\{\xytheta \in \varSigma | \rd_{x} \varphi := \xi, \rd_{y} \varphi := \eta\}
$
is the $e$-dimensional fibre over the corresponding homogeneous Lagrangian fibration $\varSigma \ni \xytheta \mapsto (x, \rd_{x} \varphi; y, \rd_{y} \varphi) \in C'_{\varphi}$. 
Here, $\rd \theta''$ is the Lebesgue measure on $\dot{\R}^{e}$ and the variable $\theta''$ is defined by the 
splitting\footnote{Such 
	a splitting is always possible due to the Thom splitting (also known as the parametrised Morse) lemma 
	(see e.g.~\cite[App. C. 6]{Hoermander_Springer_2007}). 
} 
$\dot{\R}^{d} \in \theta = (\theta', \theta'') \in \dot{\R}^{d - e} \times \dot{\R}^{e}$ such that the projection $\mathfrak{C}_{\xi, \eta} \ni (x, y; \theta', \theta'') \mapsto \theta'' \in \dot{\R}^{e}$ has a bijective differential so that, for a fixed $\theta'' = \mathrm{cst}$, $\varphi (x, y; \theta', \mathrm{cst})$ is non-degenerate. 
% 
% 
% 
% 
% 
% 
% 
% 
% 
% 
% 
% 
% 
% ================================================ 
\subsubsection{Adjoint} 
\label{sec: adjoint_FIO}
% ================================================ 
To describe the adjoint of a Fourier integral operator we recall some standard notions from linear algebra for our setup. 
Let $E^{*} \to M$ be the dual vector bundle of $E \to M$ and let $x$ be any point in $M$. 
Then, $\bar{E}_{x}$ is the complex conjugate of the vector space $E_{x}$, defined as the set of all conjugate-linear maps from $E_{x}^{*}$ to $\C$ and the identity $E_{x} \to \bar{E}_{x}$ is conjugate-linear. 

A Fourier integral operator $A : C_{\mathrm{c}}^{\infty} (N; F) \to  C^{\infty} (M; E)$ has a unique formal adjoint in the sense that, there exists a unique Fourier integral operator
$
	\bar{A}^{*} : C_{\mathrm{c}}^{\infty} (M; \bar{E}^{*}) \to C^{\infty}  (N; \bar{F}^{*})
$
with the property 
\begin{equation}
	\int_{N} (\bar{A}^{*} \phi) (y) \, v (y) := \int_{M} \phi (x) \, (Av) (x)  
\end{equation}
for any $\phi \in C_{\mathrm{c}}^{\infty} (M; \bar{E}^{*})$ any $v \in C_{\mathrm{c}}^{\infty} (N; F)$. 
If $C^{-1}$ denotes the inverse relation of $C$ obtained by interchanging $T^{*} M$ and $T^{*} N$ then its Schwartz kernel and principal symbol are given by~\cite[Thm. 25.2.2]{Hoermander_Springer_2009}   
\begin{subequations}
	\begin{eqnarray}
		&& 
		\bar{A}^{*} \in I^{m} \big( N \times M, C^{-1 \prime};  \Hom{\bar{E}^{*}, \bar{F}^{*}} \big), 
		\\ 
		&& 
		\symb{[\bar{A}^{*}]} = \mathsf{s}^{*} (\overline{\symb{[A]}}^{*}) \in S^{m + \frac{\nM + \nN}{4} - [1]} \big( C^{-1}; \Maslov_{\ms C^{-1}} \otimes \widetilde{\mathrm{Hom}} (\bar{E}^{*}, \bar{F}^{*}) \big), 
	\end{eqnarray}
\end{subequations}
where $\mathsf{s} : N \times M \to M \times N$ is the interchanging map and $\Maslov_{\ms C^{-1}} \to C^{-1}$ is the Keller-Maslov bundle over $C^{-1}$. 
In any orthonormal frames of $E$ and $F$, this yields 
\begin{equation}
	A^{k}_{r} (y, x) \equiv (2\pi)^{-(\nM + \nN + 2d - 2e) / 4} \int_{\Rd} \re^{-\ri \varphi (y, x; \theta)} \bar{a}_{r}^{k} (y, x; \theta) \, \rd \theta 
\end{equation}
whenever $A$ is given by~\eqref{eq: 25_1_3_clean_Hoemander}, where as before $\equiv$ means modulo smoothing kernels. 
% ================================================
% 
% 
% 
% 
% 
% 
% 
% 
% 
% 
% ================================================ 
\subsubsection{Algebra of Fourier integral operators} 
\label{sec: alg_FIO}
% ================================================ 
A necessary assumption for the product (composition) of two Fourier integral operators to be well-defined is that the first operator must be properly supported. 
Then the defined composition may, however, still fail to be a Fourier integral operator.
For instance, the composition of two canonical relations does not necessarily have the required properties to define another Fourier integral operator.  
In order to have a well-defined product that is again a Fourier integral operator, we are obliged to assume that their Schwartz kernels (Lagrangian distributions) are properly supported and the composition of canonical relations is clean, proper and connected~\cite[Sec. 4.1]{Weinstein_Nice_1975},~\cite[Thm. 5.4]{Duistermaat_JInventMath_1975} 
(see also, e.g.~\cite[Thm. 21.2.14]{Hoermander_Springer_2007},~\cite[Def. 5.2 $($p. 458$)$]{Treves_Plenum_1980}).   
Given a complex vector bundle $\tilde{F} \to \tilde{N}$ over a manifold $\tilde{N}$, if 
$A \in I^{m} \big( M \times \tilde{N}, C'; \Hom{\tilde{F}, E} \big)$ 
and 
$B \in I^{m'} \big( \tilde{N} \times N, \varLambda'; \Hom{F, \tilde{F}} \big)$ 
with the required restrictions, 
then~\cite[Thm. 25.2.3]{Hoermander_Springer_2009} 
(see also, e.g.~\cite[Thm. 5.3 (p. 461)]{Treves_Plenum_1980})  
\begin{equation} \label{eq: product_Lagrangian_dist_bundle}
	AB := A \circ B \in I^{m + m' + e/2} \big( M \times N, (C \circ \varLambda)'; \Hom{F, E} \big), 
\end{equation}
where $e$ is the excess of the clean composition $C \circ \varLambda$ and its principal symbol is given by 
\begin{equation} \label{eq: product_symbol_Lagrangian_dist_bundle}
	\symb{AB} = \symb{A} \diamond \symb{B}.  
\end{equation} 
The symbol $\diamond$ is to be understood as follows. 
One obtains the exterior tensor product $\symb{A} \boxtimes \symb{B}$ followed by intersecting with the diagonal. 
Then the resulting quantity is integrated over the compact $e$-dimensional fibre $\mathfrak{F}_{\xxiyeta}$ over $\xxiyeta \in C \circ \varLambda$. 
Finally, the endomorphism trace is taken of the hindmost expressed.

The space of Fourier integral operators is not an algebra unless the canonical relation is symmetric and transitive~\cite[Ex. 1 and comment on it in the following page]{Guillemin_JFA_1993}. 
In case, the composition of the canonical relation $C$ by itself is clean with excess zero, proper and connected the space of Fourier integral operators $\big( I^\bullet (M \times M, C'; \Hom{E, E}), \circ \big)$ is an associative graded algebra over the field $\C$.
The composition~\eqref{eq: product_symbol_Lagrangian_dist_bundle} of principal symbols constitutes a product which is commutative only in the scalar case but non-commutative in general. 
We remark that the product formula~\eqref{eq: product_symbol_Lagrangian_dist_bundle} becomes simpler (see~\eqref{eq: product_symbol_FIO_symplecto}) when both $C$ and $\varLambda$ are graphs of some homogeneous symplectomorphisms. 
The product is then always defined; see Appendix~\ref{sec: FIO_symplecto_bundle} for details. 
% ================================================ 
% 
% 
% 
% 
% 
% 
% 
% 
% 
% 
% ================================================ 
\subsection{Connections and the subprincipal symbols of pseudodifferential operators}
% ================================================ 
The first important observation is that the subprincipal symbol of a pseudodifferential operator with scalar principal symbol transforms like a (partial-) connection $1$-form under change of bundle charts 
(see e.g.~\cite[Prop. 3.1]{Jakobson_CMP_2007}). 
This is perhaps not surprising given that the subprincipal symbol appears as a constant term in transport equations. 
Connections that are naturally defined from transport equations have appeared first in the work of Dencker on propagation of polarization sets~\cite[p. 365-366]{Dencker_JFA_1982}. 
They are often referred to as Dencker connections in the mathematical physics literature. 
For the Maxwell system, Dencker found that this connection equals to the Levi-Civita connection. 
For the spin-Dirac operator, it was verified, for example, 
in~\cite{Hollands_adiabatic_CMP_2001} 
that the Dencker connection is indeed the spin-connection.
We will now explain the precise relation between geometrically defined connections and the parallel transport induced by the subprincipal symbol in a systematic way. 

For a scalar $P \in \PsiDO{m}{M}$, the subprincipal symbol $\subSymb{P}$ is a well defined function on $\dot T^*M$, given by~\cite[$(5.2.8)$ and Prop. 5.2.1]{Duistermaat_ActaMath_1972}  
(see also, e.g.~\cite[Thm. 18.1.33 and $(18.1.33')$]{Hoermander_Springer_2007}) 
\begin{equation} \label{eq: subprincipal_symbol}
	 \subSymb{P} = p_{m-1} + \frac{\ri}{2} \frac{\partial^{2} \symb{P}}{\partial x^{i} \partial \xi_{i}}, \quad i = 1, \ldots, n,  
\end{equation}
where $\symb{P}$ is the principal symbol and $p_{m-1}$ is the next (cf. Definition~\ref{def: polyhomogeneous_symbol_mf}) to the leading order homogeneous term in the total symbol of $P$  when expressed as an operator acting on halfdensities.
If $P \in \PsiDO{m}{M; E}$ then it can be locally represented by a $\rk{E} \times \rk{E}$-matrix of pseudodifferential operators, if we fix a local bundle frame $(e_1,\ldots,e_{\rk E})$ of $E$. 
The subprincipal symbol of $P$ is then defined with respect to this local frame as the $\rk{E} \times \rk{E}$-matrix-valued function on $\dot T^*M$ given by the subprincipal symbols of the elements of the matrix representing $P$. 
We will see below that under a change of bundle frames, this matrix will transform like a connection $1$-form along the Hamiltonian vector field in the sense explained below.

Let us summarise the transformation properties of the subprincipal symbol. 
In case $P \in \PsiDO{m}{M; E}, m \in \R$ has a scalar principal symbol and $Q \in \PsiDO{s}{M; E}, s \in \R$ properly supported, we have the following multiplication formula~\cite[(1.4)]{Duistermaat_JInventMath_1975} 
\begin{equation} \label{eq: subprincipal_symbol_product}
	\subSymb{PQ} = \subSymb{P} \, \symb{Q} + \symb{P} \, \subSymb{Q} + \frac{1}{2 \ri} X_{P} (\symb{Q}),   
\end{equation}
where $X_P = \{ \symb{P},\cdot \}$ is the Hamiltonian vector field generated by the principal symbol $\symb{P}$ of $P$.
We also have~\cite[(1.3)]{Duistermaat_JInventMath_1975} 
\begin{equation} \label{eq: subprincipal_symbol_power}
	\subSymb{P^{k}} = k \, \symb{P}^{k-1} \, \subSymb{P}  
\end{equation}
for any $k \in \N$. 
Moreover, if $P$ is elliptic and $Q$ is a parametrix of $P$, then we have  
\begin{equation} \label{eq: subprincipal_symbol_inverse}
	\subSymb{Q} = - \, \symb{P}^{-2} \, \subSymb{P}  
\end{equation}
and in this sense the above formula also holds for any negative integer $k$.

Now, let us show that the subprincipal symbol indeed has the claimed transformation property under a change of bundle charts. In fact we show a microlocal version of this statement 
(see also~\cite{Hintz_JST_2017}). 

% 
% 
% 
% ================================================ 
\begin{proposition} \label{prop: P_connection_transformation}
	Let $M$ be a manifold and let $m, s \in \R, N \in \N$. 
	Assume $P \in \PsiDO{m}{M,\matN}$ having a scalar principal symbol $\symb{P}$. 
	Suppose that $Q \in \PsiDO{s}{M,\matN}$ is non-characteristic at some $\xxiNot \in \dotCoTanM$ and that $\tilde Q \in \PsiDO{-s}{M,\matN}$ is a microlocal parametrix for $Q$ in an open conic neighbourhood $\cU$ of $\xxiNot$. 
	Then the subprincipal symbol of $\tilde{Q} P Q$ is 
	\begin{equation*}
		\subSymb{\tilde Q P Q} \xxi 
		= (\symb{Q}^{-1} \, \subSymb{P} \, \symb{Q}) \xxi 
		- \ri \big( \symb{Q}^{-1} \, X_P(\symb{Q}) \big) \xxi 
	\end{equation*}
	for any $\xxi \in \cU$, where $X_{P}$ is the Hamiltonian vector field generated by $\symb{P}$. 
\end{proposition}
% ================================================ 

% 
% 
% 
% ================================================ 
\begin{proof}
	First note that $\xxiNot \notin \ES{\tilde Q P Q - P - \tilde Q [P,Q]}$. 
	Since $P$ has a scalar principal symbol, the order of $[P,Q]$ is $m+s-1$, and hence the order of $\tilde Q [P,Q]$ is $m-1$.
 	Thus, $\subSymb{\tilde Q P Q} = \subSymb{P}  + \symb{\tilde Q [P,Q]}$ on $\cU$. 
 	It remains to compute the principal symbol of $[P,Q]$. Since the principal symbol of $P$ is scalar-valued, we obtain the commutator relation 
 	(see e.g.~\cite[(14)]{Jakobson_CMP_2007}) 
 	\begin{equation} \label{eq: symbol_commutator_PsiDO}
 		\symb{[P,Q]} = - \ri \{\symb{P},\symb{Q}\} + [\subSymb{P},\symb{Q}],  
 	\end{equation}
 	as a consequence of the multiplication formula~\eqref{eq: subprincipal_symbol_product}. 
	Therefore, on $\cU$, we get $\symb{Q^{-1} [P,Q]} = - \ri \symb{Q}^{-1} \{ \symb{P}, \symb{Q}\} + \symb{Q}^{-1}  \symb{P}  \symb{Q} -  \symb{P}$ and the result follows.
\end{proof}
% ================================================ 
% 
% 
% 

Let $P \in \PsiDO{m}{M,\matN}, Q \in \PsiDO{m'}{M,\matN}, R \in \PsiDO{m''}{M,\matN}$ be properly supported. 
Then, an application of the product formula~\eqref{eq: subprincipal_symbol_product} yields 
\begin{align}
	\subSymb{PQR}
	& = 
	\subSymb{P} \symb{Q} \symb{R} + \symb{P} \subSymb{Q} \symb{R} + \symb{P} \symb{Q} \subSymb{R} 
	\nonumber \\ 
	& + \frac{1}{2 \ri} \Big( \symb{P} \{ \symb{Q}, \symb{R} \} + \{ \symb{P}, \symb{Q} \} \symb{R} 
	+ 
	\frac{\partial \symb{P}}{\partial x^{i}} \symb{Q} \frac{\partial \symb{R}}{\partial \xi_{i}} - \frac{\partial \symb{P}}{\partial \xi_{i}} \symb{Q} \frac{\partial \symb{R}}{\partial x^{i}} \Big)
\end{align}
If $P$ and $Q$ have scalar principal symbols then $\subSymb{QPQ} = 2 \symb{P} \symb{Q} \subSymb{Q} + \symb{Q}^{2} \subSymb{P}$. 
This entails the following result. 

% 
% 
% 
% ================================================ 
\begin{proposition} \label{prop: P_orderonerescale}
	Let $M$ be a manifold and let $m, s \in \R, N \in \N$.  
	Suppose that $P \in \PsiDO{m}{M,\matN}, Q \in \PsiDO{s}{M,\matN}$ having scalar principal symbols and that $Q$ has vanishing subprincipal symbol.
	Then $\subSymb{QPQ} = \symb{Q}^2 \, \subSymb{P}$. 
\end{proposition}
% ================================================ 

%
%
% 
% ================================================ 
\begin{definition} \label{def: P_compatible_connection}
	Let $E \to M$ be a smooth complex vector bundle over a manifold $M$ and let $P \in \PsiDO{m}{M; E}, m \in \R$. 	
	A connection $\nabla^{E}$ on $E$ will be called $P$-\textit{compatible} if and only if 
	\begin{equation*}
		\mathsf{\Gamma} \big( (\rd_{x} \pi) X_{P} \big) \xxi = \ri \subSymb{P} \xxi    
	\end{equation*}
	for all $\xxi \in \Char{P}$, where $\subSymb{P}$ resp. $\Char{P}$ are the subprincipal symbol resp. the characteristic set of $P$, $\mathsf{\Gamma}$ is the connection $1$-form of $\nabla^{\pi^{*} E}$ and $\pi: \dotCoTanM \to M$ is the punctured cotangent bundle. 
	In other words, a $P$-compatible connection $\nabla^{E}$ indices the covariant derivative 
	\begin{equation*}
		\nabla_{X_{P}}^{\pi^{*} E} = X_{P} + \mathsf{\Gamma} \big( (\rd \pi) X_{P} \big) 
	\end{equation*}
	on the bundle $\pi^{*} E \to \dotCoTanM$ along the Hamiltonian vector field $X_{P}$ generated by the principal symbol of $P$. 
\end{definition}
% ================================================ 
% 
% 
% 

This definition makes sense because both quantities, $\mathsf{\Gamma}$ and $\subSymb{P}$ have the same transformation law under change of bundle frames as shown in Proposition~\ref{prop: P_connection_transformation}. 
Hence $X_P + \ri \subSymb{P}$ has an invariant meaning. 

% 
%
%
%
% ================================================ 
\begin{proposition} \label{prop: L_P_L_compatible_connection}
	Let $E \to M$ be a smooth complex vector bundle over a manifold $M$ and $m, s \in \R$.  
	Suppose that $P \in \PsiDO{m}{M; E}$ has a scalar principal symbol $\symb{P}$ and that $\nabla^{E}$ is a $P$-compatible connection. 
	Assume that $Q \in \PsiDO{s}{M; E}$ has a scalar principal symbol. 
	Assume further that for every point in $M$ there is an open neighbourhood and a local trivialisation of $E$ such that $Q$ has vanishing subprincipal symbol with respect to this local trivialisation. 
	Then $\nabla^{E}$ is $QPQ$-compatible.  
\end{proposition}
% ================================================ 
% 
% 
% 
% ================================================ 
\begin{proof}
	We fix one bundle frame to check this.
	Define $\tilde P = Q P Q$. 
	Then 
	$
		X_{\tilde P} = \{\symb{\tilde P} , \cdot \} = \{\symb{Q}^2 \, \symb{P}, \cdot \} =  \symb{Q}^2 \{\symb{P} , \cdot \} + 2 \symb{P} \, \symb{Q}  \{\symb{Q} , \cdot \}
	$.
	When restricted to covectors in $\Char{P}$ this equals $\symb{Q}^2 \{\symb{P} , \cdot \} = \symb{Q}^2 X_{P}$.
	Hence, on $\Char{P}$ we have 
	\begin{equation*}
		\nabla_{X_{\tilde P}}^{\pi^{*} E} - X_{\tilde P} = \symb{Q}^2  \Big( \nabla_{X_{P}}^{\pi^{*} E} - X_{P} \Big) = \ri \symb{Q}^2 \, \subSymb{P} = \ri \subSymb{Q P Q} =\ri \subSymb{\tilde P} 
	\end{equation*}
	as an application of Proposition~\ref{prop: P_orderonerescale}. 
\end{proof}
% ================================================ 
% 
% 
% 

The main observation is now that the Weitzenb\"{o}ck connection defined by $\square$ is compatible with this operator in the above sense. 

%
%
% ================================================ 
\begin{theorem} \label{thm: NHOp_compatible_connection}
	Let $E \to M$ be a smooth complex vector bundle over a spacetime $(M, g)$. 
	Suppose that $\square$ is a normally hyperbolic operator on $E$ and that $\nabla^{E}$ is the Weitzenb\"ock connection.
	Then $\nabla^{E}$ is $\square$-compatible.
\end{theorem}
% ================================================ 
% 
% 
% 
% ================================================ 
\begin{proof}
	Since the potential $V$ in~\eqref{eq: Weitzenboeck_connection_NHO} is of order zero, it does not contribute to the subprincipal symbol. 
	We will check compatibility in a local frame $(e_1, \ldots, e_{\rk E})$. 
	In local coordinates $(x^{\mu}), \mu = 1, \ldots, n$, one has $\nabla_{\mu}^{E} = \partial_{\mu} + \Gamma_{\mu}$ where $\Gamma$ is the connection $1$-form of the Weitzenb\"{o}ck connection: $\nabla_{\mu}^{E} e_{i} = \Gamma_{\mu i}^{j} e_{j}$. 
	The Weitzenb\"{o}ck formula~\eqref{eq: Weitzenboeck_connection_NHO} then reads 
	$
		- \tr_{g} \big( \nabla^{T^{*} M \otimes E} \circ \nabla^{E} \big) + V  
		=
		- g^{\mu \nu} \nabla_{\mu}^{E} \nabla_{\nu}^{E} 
		+ g^{\mu \nu} \Gamma_{\quad \mu \nu}^{\mathrm{LC} \, \rho} \partial_{\rho} +V_1
	$
	for some potential $V_1$ and where $\Gamma^{\mathrm{LC}}$ is the Levi-Civita connection. 
	This expression is the formula for the operator acting on functions. 
	The formula~\eqref{eq: subprincipal_symbol} for the subprincipal symbol needs to be applied to the full symbol of the operator acting on halfdensities. 
	Since we have the canonical Lorentzian volume form $\sqrt{|\det g|} \rd x^{1} \wedge \ldots \wedge \rd x^{n}$ on $M$, the formula for $\square$ on halfdensities in local coordinates is given by
	\begin{equation}
		|\det g|^{\frac{1}{4}} \; \square \; |\det g|^{-\frac{1}{4}}
		= 
		- g^{\mu \nu} \nabla_{\mu}^{E} \nabla_{\nu}^{E} - \frac{\partial g^{\mu \nu}}{\partial x^\mu} \partial_\nu + V_2  
	\end{equation}
	for some potential $V_{2}$. 
	Using this representation, we see that the subprincipal symbol of $\square$ is 
	\begin{equation} \label{eq: subprincipal_symbol_NHOp}
		\subSymb{\square} \xxi = - 2 \ri g^{\mu \nu} \Gamma_{\mu} \xi_\nu. 
	\end{equation}
	The Hamiltonian vector field of the principal symbol of $\square$ in local coordinates is  given by 
	\begin{equation} \label{eq: HVF_NHOp}
		X_{\square} = 2 g^{\mu \nu} \xi_\mu \frac{\partial}{\partial x^\nu} - \frac{\partial g^{\mu \nu}}{\partial x^\alpha} \xi_\mu \xi_\nu \frac{\partial}{\partial \xi_\alpha}.
	\end{equation}
	Hence, $(\rd_{x} \pi) X_{\square} = 2 g^{\mu \nu} \xi_\mu \partial_{x^{\nu}}$ and we can see that
	$
		\Gamma \big( (\rd_{x} \pi) X_{\square}) \big) = 2 g^{\mu \nu} \Gamma_{\mu} \xi_\nu = \ri \subSymb{\square} \xxi
	$.
\end{proof}
% ================================================ 
% 
% 
% 

As a simple application of this theorem and Proposition~\ref{prop: L_P_L_compatible_connection} we have

%
%
%
% ================================================ 
\begin{corollary} \label{cor: L_NHOp_L_compatible_connection}
	Let $E \to M$ be a smooth complex vector bundle over a spacetime $(M, g)$. 
	Suppose that $\square$ is a normally hyperbolic operator on $E$ and that $Q \in \PsiDO{-1/2}{M; E}$ is properly supported with a scalar principal symbol.  
	Assume that for each point in $M$, there is an open neighbourhood over which $E$ admits a local trivialisation such that the subprincipal symbol of $Q$ vanishes with respect to this local trivialisation. 
	Then $\nabla^{E}$ is $Q \square Q$-compatible, i.e., the Weitzenb\"ock connection is compatible with the first-order operator $Q \square Q$.
\end{corollary}
% ================================================ 
% 
% 
% 
% 
% 
% 
% 
% 
% 
% 
% 
% 
%  
% ================================================ 
\subsection{Microlocal conjugate of a normally hyperbolic operator} 
\label{sec: microlocalisation_NHOp}
% ================================================
Two pseudodifferential operators are called microlocally conjugate if they can be conjugated to one another by an elliptic Fourier integral operator once they have been appropriately localised in cotangent space. 
The key point is that under some natural assumptions any first-order pseudodifferential operator can be microlocally conjugated to a vector field. 
This is originally due to 
Duistermat-H\"{o}rmander~\cite[Prop. 6.1.4]{Duistermaat_ActaMath_1972} 
for scalar operators of real principal type, which has been extended to vector bundles by 
Dencker~\cite{Dencker_JFA_1982} 
who formulated microlocal conjugate of a system of classical pseudodifferential operators locally of real principal type.  
More precisely, Dencker transformed the system of operators to a scalar pseudodifferential operator $P$ (modulo smoothing operators) with vanishing subprincipal symbol by conjugating with a system of elliptic pseudodifferential operators.  
Then, he depicted the microlocal conjugation of $P$ with $D_{1} := - \ri \one_{\matN} \partial_{1}$ by appropriate elliptic Fourier integral operators associated with the graph of a symplectomorphism locally connecting $\Char{P}$ to $\Char{D_{1}}$.  

In this section, we will explain microlocalisation in an intrinsic geometric language and present Dencker's result in a slightly more general form. 
Let us now consider compactly supported complex-matrix valued smooth functions $C_{\textrm{c}}^{\infty} \big( \Rn, \CN \big)$ on Euclidean space $\Rn$ and let $D_{1} : C_{\textrm{c}}^{\infty} \big( \Rn, \CN \big) \to C_{\textrm{c}}^{\infty} \big( \Rn, \CN \big)$. 
We will show now that any first order pseudodifferential operator $P$ of real principle type on $E$ having a scalar principal symbol is microlocally conjugate to $D_{1}$ in the sense of~\eqref{eq: microlocal_conjugate_P} in the theorem below. 

% 
% 
% 
% ================================================ 
\begin{center}
	\begin{tikzpicture}
		\node (a) at (0,0) {$\cU$};
		\node (b) at (4,0) {$\cU'$}; 
		\node (c) at (0,2) {$\pi^{*} \Hom{E, E}_{\cU}$};
		\node (d) at (4,2) {$\matN$}; 
		\node[below] at (2, 0) {$\varkappa$}; 
		\node[left] at (0,1) {$\symb{P}$}; 
		\node[above] at (2.1, 2) {$\hat{\varkappa}$}; 
		\node[right] at (4,1) {$\symb{D_{1}}$};  
		\draw[->] (b) -- (a); 
		\draw[->] (d) -- (c); 
		\draw[->] (a) -- (c); 
		\draw[->] (b) -- (d);
	\end{tikzpicture}
	\hfil
	\begin{tikzpicture}
		\node (a) at (0,0) {$C_{\textrm{c}}^{\infty} (M; E)$};
		\node (b) at (5,0) {$C_{\textrm{c}}^{\infty} (\Rn, \CN)$}; 
		\node (c) at (0,2) {$C_{\textrm{c}}^{\infty}(M; E)$};
		\node (d) at (5,2) {$C_{\textrm{c}}^{\infty} (\Rn, \CN)$}; 
		\node[below] at (2.2, 0) {$B$}; 
		\node[left] at (0,1) {$P$}; 
		\node[above] at (2.2, 2) {$\tilde{B}$}; 
		\node[right] at (5,1) {$D_{1}$}; 
		\draw[->] (b) -- (a); 
		\draw[->] (a) -- (c); 
		\draw[->] (c) -- (d); 
		\draw[->] (b) -- (d);
	\end{tikzpicture}
	\captionof{figure}{A 
		schematic diagram of microlocalisation. 
		The diagram on the right commutes in a microlocal sense near the point $\xxiNot$ and the map $\hat{\varkappa}$ is defined by 
		$\symb{P} = \hat{\varkappa} (\symb{D_{1}}) = \symb{B} \circ \symb{D_{1}} \circ \symb{\tilde{B}}$.
		}
	\label{fig: microlocalisation}
\end{center}
% ================================================ 

% 
% 
% 
% ================================================
\begin{theorem}[Microlocalisation] 
\label{thm: microlocalisation_P}
    Let $E \to M$ be a smooth complex vector bundle of rank $N$ over an  $n$-dimensional manifold $M$.	
    Suppose that $P \in \PsiDO{1}{M; E}$ is a properly supported first-order pseudodifferential operator on $E$ with real scalar principal symbol $\symb{P}$ such that: 
	\begin{enumerate}[label=(\alph*)]
		\item \label{con: p_x_xi_zero}
		$\symb{P} \xxiNot = 0$ for some element $\xxiNot$ in the punctured cotangent bundle $\dotCoTanM$ and  
		\item \label{con: HVF_radial}
		the Hamiltonian vector field $X_{P}$ of $\symb{P}$ and the radial direction are linearly independent at $\xxiNot$.
        \end{enumerate}
	Then for any $m \in \R$, there exist 
	\begin{enumerate}[label=(\roman*)] 
		\item
		a homogeneous symplectomorphism $\varkappa$ from an open conic neighbourhood $\cU'$ of $(0, \eta_{1} \rd y^{1})$ in $\dotCoTanRn$ to an open conic coordinate chart $\big( \cU, (x^{1}, \ldots, x^{n}$; $\xi_{1}, \ldots, \xi_{n}) \big)$ of $\xxiNot$ in $\dotCoTanM$ such that 
		\begin{equation} \label{eq: choice_symplecto}
			\varkappa^{*} \symb{P} = \xi_{1} \mathbbm{1}_{\Hom{E, E}} 
		\end{equation} 
		and 
		\item 
		properly supported Lagrangian distributions 
		$B \in I^{m} \big( M \times \Rn, \varGamma'; \Hom{\C^{N}, E} \big)$   
		and 
		$\tilde{B} \in I^{-m} \big( \Rn \times M, \varGamma^{-1 \prime}; \Hom{E, \C^{N}} \big)$ 
		so that $B \tilde{B}, \tilde{B} B$ both are zero-order pseudodifferential operators and  
		\begin{subequations}
			\begin{align}
				& 
				\WFPrime{B} \subset \cU_{(x_{0}, \xi^{0}; 0, \eta_{1} \rd y^{1})}, 
				& 
				\WFPrime{\tilde{B}} \subset \cU'_{(0, \eta_{1} \rd y^{1}; x_{0}, \xi^{0})},
				\label{eq: microlocalisation_1}
				\\ 
				& 
				(x_{0}, \xi^{0}) \notin \ES{B \tilde{B} - I_{E}}, 
				& 
				(0, \eta_{1} \rd y^{1}) \notin \ES{\tilde{B} B - I}, 
				\label{eq: microlocalisation_2}
				\\ 
				& 
				(x_{0}, \xi^{0}) \notin \ES{B D_{1} \tilde{B} - P}, 
				& 
				(0, \eta_{1} \rd y^{1}) \notin \ES{\tilde{B} P B - D_{1}},       
				\label{eq: microlocal_conjugate_P}
			\end{align}
		\end{subequations}
		where $\varGamma$ is the graph of $\varkappa$, $D_{1} := - \ri \one_{\matN} \partial / \partial y^{1} : C_{\mathrm{c}}^{\infty} (\Rn, \CN) \to C_{\mathrm{c}}^{\infty} (\Rn, \CN)$, and $\cU_{(x_{0}, \xi^{0}; 0, \eta_{1}\rd y^{1})}$ resp. $\cU'_{(0, \eta_{1}\rd y^{1}; x_{0}, \xi^{0})}$ are small conic neighbourhoods of $(x_{0}, \xi^{0}; 0, \eta_{1} \rd y^{1}) \in \dotCoTanM \times \dotCoTanRn$ resp. $(0, \eta_{1}\rd y^{1}; x_{0}, \xi^{0}) \in \dotCoTanRn \times \dotCoTanM$. 
	\end{enumerate}

	In addition, if $m=0$ and $E \to M$ is endowed with a sesquilinear form $(\cdot|\cdot)$ with respect to which $P$ is formally selfadjoint then $\tilde{B}$ can be chosen as the adjoint of $B$ provided that $\C^N$ is endowed with a standard sesquilinear form of the same signature as $(\cdot|\cdot)$.  

	In case,  $m=0$ and $E \to M$ is equipped with a $P$-compatible connection $\nabla^{E}$ then the principal symbols of $B$ resp. $\tilde{B}$ can be chosen $\one$ near $(x_{0}, \xi^{0}; 0, \eta_{1} \rd y^{1})$ resp. $(0, \eta_{1} \rd y^{1}; x_{0}, \xi^{0})$ with respect to a frame that is parallel along $X_{P}$. 

	Whenever, $\big( E \to M, (\cdot|\cdot), \nabla^{E} \big)$ is a vector bundle with a $P$-compatible connection $\nabla^{E}$ and a sesquilinear form $(\cdot|\cdot)$ such that $P$ is formally selfadjoint with respect to $(\cdot|\cdot)$, and $m=0$, then we can choose $B$ such that the principal symbol of $B$ equals $\one$ near $(x_{0}, \xi^{0}; 0, \eta_{1} \rd y^{1})$ with respect to a frame that is unitary and parallel along $X_{P}$, and  $\tilde{B}=B^*$.
\end{theorem}
% ================================================ 
% 
% 
% 

A schematic of this notion has been portrayed in Figure~\ref{fig: microlocalisation}.  

% 
% 
% 
% ================================================ 
\begin{proof}
	We will prove the proposition imitating the strategy used for the scalar version~\cite[Prop. 6.1.4 and Lem. 6.6.4]{Duistermaat_ActaMath_1972} 
	(see also~\cite[Prop. 26.1.3]{Hoermander_Springer_2009}). 
	The existence of $\varkappa$ satisfying~\eqref{eq: choice_symplecto} is guaranteed by a 
	homogeneous Darboux theorem~\cite[Lem. 6.6.3]{Duistermaat_ActaMath_1972} 
	(see also~\cite[Thm. 21.3.1]{Hoermander_Springer_2007}) 
	which prerequisites our hypotheses~\ref{con: p_x_xi_zero} and~\ref{con: HVF_radial}.  
	Suppose that $b \in S^{m} \big( \varGamma; \Maslov \otimes \widetilde{\mathrm{Hom}} (\CN, E) \big)$ has an inverse in a conic neighbourhood of $(x_{0}, \xi^{0}; 0, \eta_1 \rd y^{1}) \in \varGamma$. 
	Then we can obtain a properly supported  $B_{1} \in I^{m} \big( M \times \Rn, \varGamma'; \Hom{\C^{N}, E} \big)$ such that $\WFPrime{B_{1}} \subset \cU_{(x_{0}, \xi^{0}; 0, \eta_1 \rd y^{1})}$ and $B_{1}$ is non-characteristic (see Definition~\ref{def: elliptic_FIO}) at $(x_{0}, \xi^{0}; 0, \eta_1 \rd y^{1})$ from the construction given in 
	Appendix~\ref{sec: FIO_symplecto_bundle}. 

	By Theorem~\ref{thm: existence_parametrix_FIO}, there exists a unique microlocal parametrix $\tilde{B}_{1} \in I^{- m} \big( \Rn \times M, \varGamma^{-1 \prime}$; $\Hom{E, \C^{N}} \big)$ such that~\eqref{eq: microlocalisation_2} is satisfied . 
	Since $B_{1}$ and $\tilde{B}_{1}$ have reciprocal principal symbols to each other on $\cU_{(x_{0}, \xi^{0}; 0, \eta_1 \rd y^{1})}$, $\tilde{B}_{1} P B_{1}$ has the principal symbol $\eta_{1} \one_{\matN}$ on $\cU'$ due to~\eqref{eq: choice_symplecto} and the Egorov Theorem~\ref{thm: Egorov}. 
	Furthermore 
	\begin{equation} \label{eq: WF_B1_P_A1_D_Q}
		(0, \eta_1 \rd y^{1}) \notin \ES{\tilde{B}_{1} P B_{1} - D_{1} - Q}, 
	\end{equation}
	for some $Q \in \PsiDON{0}$. 

	To find the operators $B$ and $\tilde{B}$ with the claimed properties, we construct properly supported elliptic $B_{2}, \tilde{B}_{2} \in \PsiDON{0}$ such that (recall, $\equiv$ means modulo smoothing operators) 
	\begin{eqnarray}
		&& 
		\tilde{B}_{2} B_{2} \equiv I, 
		\label{eq: 26_1_03_Hoermander}
		\\ 
		&& 
		\tilde{B}_{2} (D_{1} + Q) B_{2} \equiv D_{1},    
		\label{eq: 26_1_3_Hoermander}
	\end{eqnarray}
	and we set $B := B_{1} B_{2}$ and $\tilde{B} := \tilde{B}_{2} \tilde{B}_{1}$.  
	Then~\eqref{eq: 26_1_03_Hoermander} and~\eqref{eq: 26_1_3_Hoermander} together with~\eqref{eq: WF_B1_P_A1_D_Q} imply 
	\begin{eqnarray*}
		(0, \eta_1 \rd y^{1}) \notin \ES{\tilde{B}_{2} (\tilde{B}_{1} B_{1} - I) B_{2}} 
		= 
		\ES{\tilde{B} B - I},  
		\\ 
		(0, \eta_1 \rd y^{1}) \notin \ES{\tilde{B}_{2} (\tilde{B}_{1} P B_{1} - D_{1} - Q) B_{2}} 
		= 
		\ES{\tilde{B} P B - D_{1}}, 
	\end{eqnarray*}
	which proves the second half of~\eqref{eq: microlocalisation_2} and~\eqref{eq: microlocal_conjugate_P}, and the first half follows immediately after we multiply from left and right by $B$ and by $\tilde{B}$.

	It therefore remains to construct $B_2$ and $\tilde{B}_2$ such that~\eqref{eq: 26_1_03_Hoermander} and~\eqref{eq: 26_1_3_Hoermander} hold.
	By the existence of a parametrix 
	(see e.g.~\cite[Thm. 18.1.24]{Hoermander_Springer_2007}), 
	for every elliptic operator $B_{2} \in \PsiDON{0}$ there exists $\tilde{B}_{2} \in \PsiDON{0}$ such that $\tilde{B}_{2} B_{2} - I$ and $B_{2} \tilde{B}_{2} - I$ are smooth. 
	Multiplying~\eqref{eq: 26_1_3_Hoermander} by $B_{2}$ from the left we arrive at the equivalent condition for~\eqref{eq: 26_1_3_Hoermander} that 
	\begin{equation} \label{eq: 26_1_3_Hoermander_equiv} 
		[D_{1}, B_{2}] \equiv - Q B_{2}   
		\tag{\ref{eq: 26_1_3_Hoermander}$^{\prime}$} 
	\end{equation}
	for some elliptic $B_{2}$. 
	We will now construct a solution of~\eqref{eq: 26_1_3_Hoermander_equiv} order by order, starting with the principal symbol. 
	The principal symbol of~\eqref{eq: 26_1_3_Hoermander_equiv} vanishes provided  
	\begin{equation}
		- \ri \left\{ \symb{D_{1}}, \symb{B_{2}} \right\} + \symb{Q} \symb{B_{2}} = 0,   
	\end{equation} 
	as the the subprincipal symbol of $D_{1}$ vanishes; cf.~\eqref{eq: symbol_commutator_PsiDO}.  
	If $q$ is the principal symbol of $Q$, then the preceding equation yields 
	\begin{equation} \label{eq: microlocalisation_first_transport}
		\parDeri{y^{1}}{b_0} = - \ri q b_0    
	\end{equation}
	for the principal symbol $b_0$ of $B_2$. 
	This is a first-order differential equation, hence a unique solution exists, given the initial condition $b_0 (y^{1}=0, \cdot) = \one_{\CN}$ and this solution depends smoothly on $q$.
	By construction, $b_0 \yeta$ is homogeneous of degree zero and $\det (b_0 \yeta)$ is non-vanishing. 
	Therefore $b_0^{-1}$ exists and is homogeneous of degree zero.  
	Defining a properly supported $B_{2, 0} \in \PsiDON{0}$ with homogeneous principal symbol $b_0$, we now successively construct $B_{2, k} \in \PsiDON{-k}$ 
	so that, for every $k \in \N$:  
	\begin{equation} \label{eq: microlocalisation_k_transport} 
		[D_{1}, B_{2, 0} + \ldots + B_{2, k}] 
		+ 
		Q (B_{2, 0} + \ldots + B_{2, k}) = R_{k+1} \in \PsiDON{- (k+1)}.   
	\end{equation} 
	This is equivalent to the corresponding homogeneous principal symbols $b_{k}$ of $B_{2, k}$ and $r_{k}$ of $R_{k}$ of degree $-k$, to satisfy	
	\begin{equation}
		- \ri \parDeri{y^{1}}{b_{k}} + q b_{k} = - r_{k}. 
	\end{equation} 	
	This equation can be solved by the Duhamel principle and the solution reads 
	\begin{equation} \label{eq: microlocalisation_sol_k_transport}
		b_{k} \yeta = - \ri b_0 \yeta \int_{0}^{y^{1}} 
		b_0^{-1} (t, y^{2}, \ldots, y^{n}; \eta) \, r_{k} (t, y^{2}, \ldots, y^{n}; \eta) \, \rd t.   
	\end{equation}
	Then, using asymptotic summation of the symbols of $B_{2,k}$ (see Definition~\ref{def: polyhomogeneous_symbol_mf}), we can now construct an operator $B_2$ satisfying \eqref{eq: 26_1_3_Hoermander_equiv}. 	
	\newline 

	\noindent 
	\textbf{The case when $P$ is formally selfadjoint:} 
	\\ 
	We are going to prove that the operators $B$ and $\tilde{B}$ can be chosen microlocally unitary in case $P$ is formally selfadjoint with respect to $(\cdot|\cdot)$. 
	We endow the space $C_{\mathrm{c}}^{\infty} (\Rn, \CN)$ with a standard sesquilinear scalar product of the same signature as $(\cdot|\cdot)$. 
	The operator $D_{1}$ is then formally selfadjoint.  
	Futhermore, we make the choice $m = 0$. 
	Acting with $B \in I^{0} \big( M \times \Rn, \varGamma'; \Hom{\CN, E} \big)$ from the left of~\eqref{eq: microlocal_conjugate_P} gives the equivalent microlocal conjugate relation between   
	\begin{equation}
		(x_{0}, \xi^{0}; 0, \eta_1 \rd y^{1}) \notin \WFPrime{PB - B D_{1}}  
	\end{equation}
	$P$ and $D_{1}$. 
	Taking its adjoint we have $(0, \eta_1 \rd y^{1}; x_{0}, \xi^{0}) \notin \WFPrime{B^{*} P - D_{1} B^{*}}$ where  $B^{*} \in I^{0} \big( \Rn \times M, \varGamma^{-1 \prime}; \Hom{E^{*}, \CN} \big)$ and consequently   
	\begin{eqnarray*}
		&& (0, \eta_1 \rd y^{1}) \notin \ES{B^{*} P B - D_{1} B^{*} B}, 
		\\ 
		&& (0, \eta_1 \rd y^{1}) \notin \ES{B^{*} P B - B^{*} B D_{1}}.
	\end{eqnarray*}  
	Thus $(0, \eta_1 \rd y^{1}) \notin \ES{[B^{*} B, D_{1}]}$ and $B$ is non-characteristic at $(0, \eta_1 \rd y^{1})$. 
	Now the operator $B^{*} B$ is a pseudodifferential operator and since $\ri D_{1}$ is differentiation with respect to $y^1$, the total symbol of the commutator $[B^{*} B, D_{1}]$ is the $-\ri$ times the $y^1$-derivative of the total symbol of $B^{*} B$.
	Thus, on the the level of symbols, the above implies that the total symbol of $B^{*} B$ is the sum of a term independent of $y^1$ and a term that is rapidly decaying in a conic neighbourhood $\cU'$ of $(0, \eta_1 \rd y^{1})$. 
	It is bounded below on $\cU'$ because of the ellipticity of $B^{*} B$ there. 
	By Proposition~\ref{prop: FIO_2_2_2_matrix}, one can find a selfadjoint properly supported $\varPsi \in \PsiDONU{0}$ whose principal symbol is the same as that of $B$ and such that  
	\begin{subequations}
		\begin{eqnarray}
			&& 
			(0, \eta_1 \rd y^{1}) \notin \Char{\varPsi},
			\label{eq: microlocalisation_varPsi_1}
			\\ 
			&& 
			(0, \eta_1 \rd y^{1}) \notin \ES{\varPsi^{*} \varPsi - B^{*} B}, 
			\label{eq: microlocalisation_varPsi_2}
			\\ 
			&& 
			(0, \eta_1 \rd y^{1}) \notin \ES{[\varPsi, D_{1}]}, 
			\label{eq: microlocalisation_varPsi_3} 
		\end{eqnarray}
	\end{subequations}
	where $U'$ is the projection of $\cU'$ on $\Rn$. 
	Note that Proposition~\ref{prop: FIO_2_2_2_matrix} (stated for operators acting on the same bundle) can be used only in a local trivialisation of $E$ by an orthonormal frame that identifies the fibre of $E$ with $\C^N$ in such a way that the sesquilinear forms are identified. It is at this stage that we must require the sesquilinear forms to have the same signature.
	Here, the last property follows from the fact that the construction of the full symbol of $\varPsi$ in Proposition~\ref{prop: FIO_2_2_2_matrix} involves only multiplication and asymptotic summation of symbols. 
	The property of a symbol being a sum of two terms, one independent of $y_1$ and another rapidly decaying on $\cU'$, is preserved under these operations. 
	The full symbol of $\varPsi$ is therefore also of this form.

	The ellipticity~\eqref{eq: microlocalisation_varPsi_1} entails a unique microlocal parametrix $\varPhi$ for $\varPsi$. 
	In other words, there exists a properly supported $\varPhi \in \PsiDONU{0}$ such that 
	\begin{subequations}
		\begin{eqnarray}
			&&  
			(0, \eta_1 \rd y^{1}) \notin \Char{\varPhi},
			\label{eq: microlocalisation_varPhi_1} 
			\\ 
			&& 
			(0, \eta_1 \rd y^{1}) \notin \ES{\varPhi \varPsi - I} 
			\Leftrightarrow (0, \eta_1 \rd y^{1}) \notin \ES{\varPsi \varPhi - I}.
			\label{eq: microlocalisation_varPhi_2} 
		\end{eqnarray}
	\end{subequations}
	Note,~\eqref{eq: microlocalisation_varPhi_1},~\eqref{eq: microlocalisation_varPsi_2} and~\eqref{eq: microlocalisation_varPsi_3} imply that 
	\begin{subequations}
		\begin{eqnarray*}
			&& (0, \eta_1 \rd y^{1}) \notin \ES{\varPhi^{*} \varPsi^{*} \varPsi \varPhi - \varPhi^{*} B^{*} B \varPhi}, 
			\\ 
			&& (0, \eta_1 \rd y^{1}) \notin \ES{\varPhi [\varPsi, D_{1}] \varPhi}.   
		\end{eqnarray*}
	\end{subequations}
	This, accounting~\eqref{eq: microlocalisation_varPhi_2} entails that 
	\begin{subequations}
		\begin{eqnarray*}
			&& (0, \eta_1 \rd y^{1}) \notin \ES{I - (B \varPhi)^{*} B \varPhi}, 
			\\ 
			&& (0, \eta_1 \rd y^{1}) \notin \ES{[\varPhi, D_{1}]},  
		\end{eqnarray*}
	\end{subequations} 
	which completes the proof since $(x_{0}, \xi^{0}; 0, \eta_1 \rd y^{1}) \notin \WFPrime{B [D_{1}, \varPhi]}$ and therefore with $B^{\backprime} := B \varPhi$ we have $(x_{0}, \xi^{0}; 0, \eta_1 \rd y^{1}) \notin \WFPrime{P B^{\backprime} - B^{\backprime} D_{1}}$. 
	\newline 

	\noindent 
	\textbf{The case of connection $P$-compatibility:} 
	\\
	We assume that $\nabla^{E}$ is $P$-compatible. 
	Since the construction of the symbols of $B$ and $\tilde{B}$ are local, we can fix a local frame and local coordinates. 
	We will reduce the general situation to the case when the subprincipal symbol of $P$ vanishes near $\xxiNot$. 
	\begin{enumerate}
		\item[(i)] 
		\textit{The case of vanishing subprincipal symbol}: 
		We assume that $P$ is given in local coordinates with respect to some bundle frame and that in this frame $\subSymb{P} = 0$ near $\xxiNot$. 
		Let $\tilde P$ be a scalar operator with the same principal symbol as $P$ and with vanishing subprincipal symbol.  
		Using Weyl-quantisation, the operators $B_1$ and $\tilde{B}_1$ can be constructed in such a way that Egorov's theorem holds up to the subprincipal symbol level~\cite[Thm. 1]{Silva_PublMat_2007}. 
		We can choose these operators in this fashion that their principal symbols are constant $\one$ and that the subprincipal symbol of $\tilde{B}_1 \tilde P B_1$ (resp. $B_1 D_{1} \tilde{B}_1$) vanishes near $(0, \eta_1 \rd y^{1})$ (resp. $\xxiNot$). 
		We will now use the same construction as before starting with $B_1$ and $\tilde B_1$. Then the principal symbol $q$ of the remainder term $Q$ vanishes. The construction of $B_2$ and $\tilde{B}_2$ then yields operators with total symbols that are scalar and constant principal symbols equal to $\one$. 
		We conclude that, in case the subprincipal symbol of $P$ vanishes near $\xxiNot$, the operators $B$ and $\tilde{B}$ can be chosen as scalar operators with principal symbols that are constant along the flow lines of $X_P$ and $X_{D_1}$.
		\item[(ii)] 
		\textit{The case of non-vanishing subprincipal symbol}:
		We will microlocally transform $P$ to a scalar pseudodifferential operator $\tilde{P}$. 
		To be specific, for any properly supported $\tilde{P} \in \PsiDO{1}{M}$, we want to have a  $\hat{B} \in \PsiDO{0}{M; E}$ such that $\hat{B}$ is non-characteristic at $\xxiNot$ and $\xxiNot \notin \ES{P \hat{B} - \hat{B} \tilde{P} I}$. 
		We construct this operator locally and therefore fix a local frame for the bundle $E$ near the point $x_0$. 
		In the pullbacked bundle $\pi^* E \to \dot T^* M$, we can also construct a local frame that is parallel along $X_P$ with respect to $\nabla_{X_{P}}^{\pi^{*} E}$. 
		This local parallel frame is constructed by choosing a local transverse to $X_P$ and use the original frame on this transverse. 
		Parallel transport along the flow lines of $X_P$ then gives the desired frame. 
		The change of frame matrix from the original frame to the parallel frame is then a local section $b$ of $\Hom{\pi^{*} E, \pi^{*} E} \to \dotCoTanM$. 
		By construction, this frame is homogeneous of degree zero. 
		We now choose an elliptic zero-order pseudodifferential operator $\hat B$ whose principal symbol $\symb{\hat B}$ equals $b$ on $\cU$. 
		Let $\check{B}$ be a parametrix for $\hat B$. 
		The subprincipal symbol of $\check{B} P \hat B$ equals
		$
		- \ri X_{P} \symb{\hat{B}} + [\subSymb{P}, \symb{\hat{B}}] 
		$ 
		on $\cU$, using~\eqref{eq: symbol_commutator_PsiDO}. 
		By $P$-compatibility, this is exactly the formula for the connection $1$-form in the parallel local bundle frame and it therefore vanishes.
		By Proposition~\ref{prop: P_connection_transformation}, this is precisely the formula for the subprincipal symbol of $\check{B} P \hat{B}$ on $\cU$ and hence, $\tilde P \equiv \check{B} P \hat{B}$ has a vanishing subprincipal symbol on $\cU$. 

		This reduces the problem to the case of vanishing subprincipal symbol discussed earlier and let $B_3$ and $\tilde{B}_3$ are the corresponding scalar Fourier integral operators. 
		This means $B$ and $\tilde{B}$ are of the form $B := \hat{B} B_3$ and $\tilde{B} := \tilde{B}_3 \check{B}$. 
		Since the principal symbol of $\hat{B}$ is the transition function to a parallel frame and $B_3$ is a scalar operator, these imply that the principal symbols of $B$ and $\tilde{B}$ can be chosen $\one$ with respect to a parallel frame along $X_P$.  
	\end{enumerate}
	\noindent 
	\textbf{The case of selfadjoint $P$ with connection $P$-compatibility:} 
	\\
	Finally, suppose that $P$ is formally selfadjoint with respect to $(\cdot|\cdot)$ and that $\nabla^{E}$ is $P$-compatible. 
	The above construction of $\hat{B}$ can be repeated using an orthonormal frame of $E$ and another orthonormal frame of $\pi^{*} E$ that is parallel along $X_P$. 
	One obtains an operator that has principal symbol $\one$ near $\xxiNot$ with respect to the parallel orthonormal frame. 
	Now one repeats the construction of $\varPsi$ as before and note that the principal symbol can be kept as $\one$ in that way to turn $B$ into a microlocally unitary operator in the sense that $\tilde{B} = B^{*}$ with the desired properties.
\end{proof}
% ================================================ 
% 
% 
% 

We will now show microlocalisation of $\square$ as a consequence of the preceding result by replacing the generic manifold $M$ by a Lorentzian (not necessarily globally hyperbolic) manifold $(M, g)$. 
In particular, the role of the characteristic point $\xxiNot$ of $P$ in Theorem~\ref{thm: microlocalisation_P}~\ref{con: p_x_xi_zero} will be played by any lightlike covector on $(M, g)$ and thus, the integral curves of $X_{P}$ on $\Char{P}$ are actually lightlike geodesics on the cotangent bundle (cf.~\eqref{eq: HVF_NHOp}). 

% 
% 
% 
% ================================================ 
\begin{theorem} 
\label{thm: microlocalisation_NHOp}
	Let $E \to M$ be a smooth complex vector bundle of rank $N$ over an $n$-dimensional Lorentzian manifold $(M, g)$ and $\Box : \comSecE \to \comSecE$ a normally hyperbolic operator. 
	Denote by $\nabla^{E}$, the Weitzenb\"{o}ck connection associated with $\Box$ and by $\varGamma$, the graph of the homogeneous symplectomorphism from an open conic neighbourhood $\cU'$ of $(0, \eta_1 \rd y^{1})$ in $\dotCoTanRn$ to an open conic coordinate chart $\big( \cU, \xxi \big)$ centered at any lightlike covector $\xxiNot$ on $(M, g)$.
	Then, for any $m \in \R$, one can find properly supported Lagrangian distributions 
	$A \in I^{-1/2 + m} \big( M \times \Rn, \varGamma'; \Hom{\C^{N}, E} \big)$  
	and 
	$\tilde{A} \in I^{-1/2 - m} \big( \Rn \times M, \varGamma^{-1 \prime}; \Hom{E, \C^{N}} \big)$ so that $A \tilde{A}, \tilde{A} A$ both are pseudodifferential operators of order $-1$ and 
	\begin{subequations}
		\begin{align}
			& 
			\WFPrime{A} \subset \cU_{(x_{0}, \xi^{0}; 0, \eta_1 \rd y^{1})}, 
			\quad 
			& 
			\WFPrime{\tilde{A}} \subset \cU'_{(0, \eta_1 \rd y^{1}; x_{0}, \xi^{0})},
			\label{eq: microlocalisation_NHOp_1}
			\\ 
			& 
			\xxiNot \notin{\Char{A \tilde{A}} }, 
			\quad 
			& 
			(0, \eta_1 \rd y^{1}) \notin {\Char{\tilde{A} A} }, 
			\label{eq: microlocalisation_NHOp_2}
			\\ 
			& 
			(0, \eta_{1} \rd y^{1}) \notin \ES{\tilde{A} \square A - D_{1}}.      
			\label{eq: microlocal_conjugate_NHOp}
		\end{align}
	\end{subequations}
	Here, {$\Char{A \tilde{A}}$ is the characteristic set of $A \tilde{A}$, } $D_{1} = - \ri \one_{\matN} \partial / \partial y^{1} : C_{\mathrm{c}}^{\infty} \big( \Rn, \CN \big) \to C_{\mathrm{c}}^{\infty} \big( \Rn, \CN \big)$, and $\cU_{(x_{0}, \xi^{0}; 0, \eta_1 \rd y^{1})}$ resp. $\cU'_{(0, \eta_1 \rd y^{1}; x_{0}, \xi^{0})}$ are small conic neighbourhoods of $(x_{0}, \xi^{0}; 0, \eta_{1} \rd y^{1}) \in \dotCoTanM \times \dotCoTanRn$ resp. $(0, \eta_1 \rd y^{1}; x_{0}, \xi^{0}) \in \dotCoTanRn \times \dotCoTanM$. 

	In addition, if {$m = 0$ and} $E \to M$ is endowed with a sesquilinear form $(\cdot|\cdot)$ with respect to which $\square$ is formally selfadjoint then $A$ can be chosen as a scalar operator with respect to a unitary bundle frame that is parallel along the geodesic flow and microlocally $A = \tilde{A}^*$ if $\C^N$ is endowed with a standard sesquilinear form of the same signature as $(\cdot|\cdot)$. 
\end{theorem}
% ================================================ 
% 
% 
% 
% ================================================ 
\begin{proof}
	The strategy is, as usual, to reduce $\square$ to a first-order operator so that  Theorem~\ref{thm: microlocalisation_P} can be applied. 
	To do so, choose a properly supported elliptic formally selfadjoint pseudodifferential operator $L$ on $E$ of order $-1/2$ having a homogeneous scalar principal symbol $l$. This operator is chosen such that
	near each point in $M$ there exists a local trivialisation of $E$ such that the subprincipal symbol of $L$ vanishes with respect to this trivialization. 
	Such an operator always exists, as it can be constructed locally and then patched to a global operator using a suitable partition of unity.  
	Then $L \square L$ is a first-order pseudodifferential operator whose principal symbol  $l^{2} \xxi \, g_{x}^{-1} (\xi, \xi) \, \one_{\End{E}}$ vanishes on $\xxiNot$. 
	By Theorem~\ref{thm: NHOp_compatible_connection} and Corollary~\ref{cor: L_NHOp_L_compatible_connection}, $\nabla^{E}$ is compatible with both $\square$ and $L\square L$, respectively. 
	Therefore the hypotheses of Theorem~\ref{thm: microlocalisation_P} are satisfied  
and we have
	\begin{equation} \label{eq: microlocal_conjugate_L_Box_L}
		{\xxiNot \notin \ES{B D_{1} \tilde{B} - L \square L}, 
		\quad 
		(0, \eta_{1} \rd y^{1}) \notin \ES{\tilde{B} (L \square L) B - D_{1}},}
	\end{equation}	
	{where the Lagrangian distributions $B$ and $\tilde{B}$ are constructed as in Theorem~\ref{thm: microlocalisation_P}.}  
	The conclusion entails by putting $A := LB$ and $\tilde{A} := \tilde{B} L$ where $\tilde{L}$ is a parametrix for $L$. 
	In case $\square$ is formally selfadjoint, we choose $L$ formally selfadjoint as well.  
\end{proof}
% ================================================ 
% 
% 
% 

% ================================================ 
\begin{remark}
\label{rem: global_microlocalisation_NHOp} 
	We have microlocally conjugated $\square$ to $D_{1}$ around a lightlike covector in the preceding theorem~\eqref{eq: microlocal_conjugate_NHOp}. 
	It is also possible to refine the construction as follows.
	Let $I$ be a compact interval in $\R$ and $\gamma : I \to \dotCoTanM$ an integral curve of $X_{P}$:    
	\begin{equation}
		\symb{P} \circ \gamma = 0. 
	\end{equation} 
	If the composition of $\gamma$ and the projection $\dotCoTanM \to \dotCoTanM / \R_{+}$ is injective then one can find a conic neighbourhood $\cV'$ of $I \times \{(0, \eta_{1} \rd y^{1}) \}$ and a smooth homogeneous symplectomorphism $\varrho$ from $\cV'$ to an open conic neighbourhood $\varrho (\cV') \subset \dotCoTanM$ of $\gamma (I)$ such that~\cite[Prop. 26.1.6]{Hoermander_Springer_2009}  
	\begin{equation}
		\varrho \big( I \times \{ (0, \eta_{1} \rd y^{1}) \} \big) = \gamma (I), 
		\quad \varrho^{*} \symb{P} = \xi_{1} \one_{\End{E}}. 
	\end{equation}
	Imitating the proof of Theorem~\ref{thm: microlocalisation_P}, one can microlocalise $P$ to $D_{1}$ on $\gamma (I)$ 
	(see~\cite[Prop. 26.1.3$'$]{Hoermander_Springer_2009} for the scalar version). 

	As a consequence, on a Lorentzian manifold $(M, g)$, if $\Gamma$ is the graph of $\varrho$ then the proof of Theorem~\ref{thm: microlocalisation_NHOp} shows that, for any $m \in \R$, there exists Lagrangian distributions 
	$A \in I^{{-1/2} + m} \big( M \times \Rn, \Gamma'; \Hom{\C^{N}, E} \big)$  
	and 
	$\tilde{A} \in I^{{-1/2} - m} \big( \Rn \times M, \Gamma^{-1 \prime}; \Hom{E, \C^{N}} \big)$ so that $A \tilde{A}, \tilde{A} A$ both are pseudodifferential operators of order $-1$ and   
	\begin{subequations}
		\begin{align}
			& 
			\WFPrime{A} \subset \cV_{\Gamma}, 
			&  
			\WFPrime{\tilde{A}} \subset \cV'_{\Gamma^{-1}},
			\label{eq: microlocalisation_NHOp_1_global}
			\\ 
			& 
			\gamma (I) \cap {\Char{A \tilde{A}}} = \emptyset, 
			&
			(I \times \{ (0, \eta_{1} \rd y^{1}) \}) \cap {\Char{\tilde{A} A}} = \emptyset, 
			\label{eq: microlocalisation_NHOp_2_global}
			\\ 
			& 
			{(I \times \{ (0, \eta_{1} \rd y^{1}) \}) \cap \ES{\tilde{A} \square A - D_{1}}} = \emptyset,      
			\label{eq: microlocal_conjugate_NHOp_global}
		\end{align}
	\end{subequations}
	where $\cV_{\Gamma}, \cV'_{\Gamma^{-1}}$ are small conic neighbourhoods of $\Gamma$ restricted to $\gamma (I)$ and its inverse, respectively. 
\end{remark}
% ================================================ 
% 
% 
% 

We close this section by a simple application of Theorem~\ref{thm: microlocalisation_P} to derive a bundle version of 
H\"{o}rmander's propagation of singularity theorem~\cite{Hoermander_Nice_1970},~\cite[Thm. 6.1.1']{Duistermaat_ActaMath_1972}, 
as for instance, stated in  
Taylor~\cite[Thm. 4.1, p. 135]{Taylor_PUP_1981} (for Sobolev wavefront sets)
and
Dencker~\cite[Thm. 4.2]{Dencker_JFA_1982} (for the polarisation set).  
Since $\secE = \bigcap_{s \in \R} H_{\mathrm{loc}}^{s} (M; E)$, such a refinement of the usual notion of (smooth) wavefront set is captured by the Sobolev wavefront set. 
For any $s \in \R$, the Sobolev wavefront set $\mathrm{WF}^{s} (u)$ of a distribution $u \in \cD' (M; E)$ relative to Sobolev space $H_{\mathrm{loc}}^{s} (M; E)$ is defined by~\cite[p. 201]{Duistermaat_ActaMath_1972}  
\begin{equation*} \label{eq: def_Sobolev_WF}
	\mathrm{WF}^{s} (u) 
	:= 
	\bigcap_{\substack{\Psi \in \PsiDO{0}{M; E} \\ \Psi u \in H_{\mathrm{loc}}^{s} (M; E)}} \Char{\Psi} 
	= 
	\bigcap_{\substack{\Psi \in \PsiDO{s}{M; E} \\ \Psi u \in L_{\mathrm{loc}}^{2} (M; E)}}   \Char{\Psi},  
\end{equation*}
where $\Char{\Psi}$ is the characteristic set of $\Psi$ and the intersection is over all properly supported $\Psi$. 
Locally this means, for any open subset $U$ of $\Rn$, $\dot{T}^{*} U \ni \xxiNot \notin \mathrm{WF}^{s} (u)$ is not in the Sobolev wavefront set of a distribution $u \in \mathcal{D}' (U; E)$ if and only if there exists a compactly supported section $f$ on $U$ non-vanishing at $x_{0} \in U$ such that the weighted Fourier transform $(1 + | \xi |^{2})^{s} \, |\widehat{fu}(\xi)|^{2}$ is integrable in a conic neighbourhood of $ \xi^{0}$. 
An accumulation of important properties of this finer wavefront set is available, for example, in~\cite[App. B]{Junker_AHP_2002}. 
A propagation of $H^{s}$-regularity is now presented below. 
The result is similar to that of Taylor but formulated in a geometric fashion and proved in a different way employing microlocalisation on vector bundles developed in Theorem~\ref{thm: microlocalisation_P}.    

% 
% 
% 
% ================================================ 
\begin{theorem}[Propagation of Sobolev regularity] 
\label{thm: propagation_singularity_Sobolev_WF}
	Let $E \to M$ be a smooth complex vector bundle over a manifold $M$ and $u \in \cD' (M; E)$. 
	Suppose that, for some $m \in \R$, a pseudodifferential operator $\Phi \in \PsiDO{m}{M; E}$ on $E$ of order $m$ satisfies the hypotheses~\ref{con: p_x_xi_zero} and~\ref{con: HVF_radial} of Theorem~\ref{thm: microlocalisation_P} and that $\mathcal{I}$ be an interval on an integral curve (on cotangent bundle) of the Hamiltonian vector field generated by the principal symbol of $\Phi$ such that $\mathcal{I} \cap \mathrm{WF}^{s} (\Phi u) = \emptyset$. 
	Then either $\mathcal{I} \cap \mathrm{WF}^{s + m -1} (u) = \emptyset$ or $\mathcal{I} \subset \mathrm{WF}^{s + m -1} (u)$. 
\end{theorem}
% ================================================ 
% 
% 
% 

% ================================================ 
\begin{proof}
    Once Theorem~\ref{thm: microlocalisation_P} is at our disposal, the rest of the proof is the same as its scalar version~\cite[Thm. 6.1.1']{Duistermaat_ActaMath_1972}  
	(see also~\cite[Thm. 26.1.4]{Hoermander_Springer_2009}). 	
	For completeness we give the details here. 
	By conjugating $\Phi$ with appropriate pseudodifferential operators we reduce the statement to the case $m = 1$. 
	In other words, one can just consider $P$ in Theorem~\ref{thm: microlocalisation_P} instead of $\Phi$ to conclude the assertion by utilising the Sobolev continuity properties of pseudodifferential 
	(see e.g.~\cite[p. 92]{Hoermander_Springer_2007}) 
	and Fourier integral operators 
	(see e.g.~\cite[Cor. 25.3.2]{Hoermander_Springer_2009}). 
	Via microlocalisation (Theorem~\ref{thm: microlocalisation_P}), the analysis further boils down to $P = D_{1}, s = 0$ and $\xxiNot = (0, \eta_1 \rd y^{1})$ by choosing the order of the Fourier integral operator $B$ in Theorem~\ref{thm: microlocalisation_P} equals to $-s$. 
	Then, the claim follows from the form (see~\eqref{eq: adv_ret_fundamental_sol_partial_derivative}) of the advanced and retarded fundamental solutions of $D_{1}$ and these map from $L_{\mathrm{c}}^{2} (\Rn, \CN)$ to $L_{\mathrm{loc}}^{2} (\Rn, \CN)$. 
\end{proof}
% ================================================ 
% 
%
% 
% 
% 
% 
% 
% 
% 
%  
% ================================================ 
\subsection{Proof of Theorem~\ref{thm: exist_unique_Feynman_parametrix_NHOp}}
% ================================================ 
We will make use of the strategy used by 
Duistermaat-H\"{o}rmander~\cite[Thm. 6.5.3]{Duistermaat_ActaMath_1972} 
(see also~\cite[Thm. 26.1.14]{Hoermander_Springer_2009}) 
for a scalar pseudodifferential operators of real principal type and simplify their arguments by exploiting the global hyperbolicity of spacetime $(M, g)$. 
To begin with, we show that the Feynman parametrix is unique, if it exists. 
% 
% 
% 
% 
% 
% 
% 
% 
% 
% 
% uniqueness 
% ================================================ 
\subsubsection{Uniqueness} 
\label{sec: uniqueness_Feynman_parametrix}
% ================================================ 
Suppose that $\leftParametrix$ (resp. $\rightParametrix$) is a left (resp. right) Feynman parametrix of $\Box$, i.e.,  
the off-diagonal contribution of $\WFPrime{\leftParametrix, \rightParametrix}$ is given by $ C^{+}$ (defined in Definition~\ref{def: Feynman_parametrix}). 
To prove that $\leftParametrix - \rightParametrix$ is a smoothing operator, we would like to argue that $\leftParametrix \Box \rightParametrix$ is congruent both to $\leftParametrix$ and to $\rightParametrix$ modulo smoothing operators. 
But $\leftParametrix$ and $\rightParametrix$ are not properly supported which makes this a non-trivial task. 
To circumvent this difficulty, one employs the fact that $\leftParametrix H  \rightParametrix$ is defined when $H$ is a pseudodifferential operator having Schwartz kernel of compact support in $M \times M$ and then $H : \dualComSecE \to \dualSecE$. 

If $(x, \xi; y, \eta) \in \WFPrime{\leftParametrix H \rightParametrix}$ but $(x, \xi), (y, \eta)$ are in the complement of $\ES{H}$, then $(x, \xi; z, \zeta), (z, \zeta; y, \eta) \in C^{+}$ for some $(z, \zeta) \in \ES{H}$. 
This follows from the behaviour of wavefront sets under composition of kernels, as available, for instance 
in~\cite[Theorem 8.2.10]{Hoermander_Springer_2003},
and the fact that $\WFPrime{H} \subset \varDelta \dotCoTanM$. 
This entails that $(x, \xi), (y, \eta)$, and $(z, \zeta)$ are on the same lightlike geodesic (strip) $\gamma (s)$ on $\dotCoTanM$ with $(z, \zeta)$ in between the other two points. 

Since $(M, g)$ is assumed to be globally hyperbolic, $J (K)$ is compact for any compact $K \subset M$. 
Therefore, if the projections of endpoints of $\gamma$ on $M$ are in $K$ then the projection $c (s)$ of $\gamma (s)$ on $M$ stays over $J (K)$, as shown in Figure~\ref{fig: uniqueness_Feynman_parametrix}.  
Consequently, if $\singsupp{H} \cap J (K) = \emptyset$ then $\WFPrime{\leftParametrix H \rightParametrix}$ cannot have any point over $K \times K$. 
Let $\chi \in C_{\mathrm{c}}^{\infty} (M; E)$ such that $\chi$ is identically $1$ on $J (K)$ and define $\mathrm{M}_{\chi}$ be the corresponding multiplication operator. 
We observe that $[\mathrm{M}_{\chi}, \square]$ vanishes identically on $J (K)$. 
Thus, $\WFPrime{\leftParametrix [\mathrm{M}_{\chi}, \square] \rightParametrix}$ contains no  point over $K \times K$ and so it is true for 
\begin{eqnarray}
	(\dot{T}^{*} K \times \dot{T}^{*} K) \cap \WFPrime{\leftParametrix - \rightParametrix}
	& = & 
	(\dot{T}^{*} K \times \dot{T}^{*} K) \cap \WFPrime{\leftParametrix \mathrm{M}_{\chi} \square \rightParametrix - \leftParametrix \square \mathrm{M}_{\chi} \rightParametrix} 
	\nonumber \\ 
	& = &    
	(\dot{T}^{*} K \times \dot{T}^{*} K) \cap \WFPrime{\leftParametrix [\mathrm{M}_{\chi}, \square] \rightParametrix} 
	\nonumber \\ 
	& = & 
	\emptyset. 
\end{eqnarray}
Since $K$ is arbitrary, we conclude that $\leftParametrix - \rightParametrix$ is a smoothing operator and similarly for the anti-Feynman parametrix.  

% 
% 
% 
% figure 
% ================================================ 
\begin{center}
	\begin{tikzpicture}
		\filldraw[color = cyan, fill = cyan!5!white] (0,0) -- (2,2) -- (4,0) -- (2,-2) -- (0,0); 
		\draw[blue, thick] (0,0) -- (4,0); 
		\draw[cyan, thick] (0,0) .. controls (2,1.5) .. (4,0); 
		\draw[cyan, thick] (0,0) .. controls (2,-1.5) .. (4,0);
		% label 
		\node[left] at (0,0) {$\textcolor{blue}{K}$}; 
		\node[left] at (1.85,1.75) {$\textcolor{cyan}{J (K)}$}; 
		\node at (2,0.75) {\textcolor{cyan}{$c (s)$}};
	\end{tikzpicture}
	\hfil
	\begin{tikzpicture}
		\draw[cyan, ultra thick] (1,0) -- (4,0); 
		\draw[red, ultra thick] (0,0) -- (1,0); 
		\draw[red, ultra thick] (4,0) -- (5,0); 
		\draw[thick] (1,2) -- (4,2);
		\draw[thick] (1,2) to [out=190, in=20] (0,0); 
		\draw[thick] (4,2) to [out=355, in=160] (5,0); 
		\draw[->] (0,0) -- (0, 2.5);
		% label 
		\node[above] at (2.5, 2) {$\chi$};
		\node[below] at (2.5,0) {$\textcolor{cyan}{c (s)}$}; 
		\node[left] at (0,2) {$1$};
	\end{tikzpicture}
	\captionof{figure}{A 
		schematic visualisation of a consequence of global hyperbolicity of spacetime $(M, g)$ to entail the uniqueness of Feynman parametrices. 
		Here, $c$ is a lightlike geodesic on $M$ and time flows from left to right.
	}
    \label{fig: uniqueness_Feynman_parametrix}	
\end{center}
% ================================================ 
%
% 
% 
% construction 
% ================================================ 
\subsubsection{Construction} 
% ================================================ 
Notice that $\square \rightParametrix = I + R$ is equivalent to $(\rightParametrix)^{*} \square^{*} = I + R^{*}$ and, since the adjoint $\square^{*}$ of $\square$ is also a normally hyperbolic operator, the existence of left parametrices with the listed properties in this theorem follows from the existence of right parametrices for $\square$. 
It is, therefore, sufficient to construct a right parametrix for $\square$ with the required regularities. 

To begin with, let us denote $\Rn \ni x = (x^{1}, x') \in \R \times \R^{n-1}$ and note that the Schwartz kernels of the advanced and retarded fundamental solutions 
\begin{eqnarray} 
	&& 
	F_{1}^{\ret, \adv} ={\pm} \ri \Theta \big( \pm (y^{1} - z^{1}) \big) \otimes \delta (y' - z'), 
	\label{eq: adv_ret_fundamental_sol_partial_derivative} 
	\\ 
	&&  
	F_{1} := F_{1}^{\ret} - F_{1}^{\adv} = \ri \delta (y' - z') 
	= \frac{\ri}{(2 \pi)^{n-1}} \int_{\R^{n-1}} \re^{\ri (y' - z') \theta'} \rd \theta' 
	\label{eq: causal_fundamental_sol_partial_derivative} 
\end{eqnarray}
of $D_{1}$ satisfy~\cite[Prop. 6.1.2]{Duistermaat_ActaMath_1972}  
(see also~\cite[Prop. 26.1.2]{Hoermander_Springer_2009})
\begin{eqnarray} 
	&& 
	\WFPrime {F_{1}^{\adv, \ret}}  
	= 
	\varDelta \, \dot{T}^{*} \Rn \bigcup C_{1}^{\adv, \ret},  
	\label{eq: Hoermander_Prop_26_1_2_i}
	\\ 
	&&	
	C_{1}^{\adv, \ret} := \{ \xxiyeta \in C_{1} \, | \, x^{1} \lessgtr y^{1} \},  
	\\ 
	&& 
	C_{1} := \{ \xxiyeta \in T^{*} \Rn \times T^{*} \Rn \, | \, x' = y', \xi' = \eta' \neq 0, \xi_{1} = 0 = \eta_{1} \}, 
	\\ 
	&& 
	F_{1}, \chi F_{1}^{\adv, \ret} \in I^{- 1/2} ( \Rn \times \Rn, C'_{1}),  
	\label{eq: Hoermander_Prop_26_1_2_ii} 
\end{eqnarray}
where $\Theta$ is the Heaviside step function and $\chi$ is a smooth function on $\Rn \times \Rn$ vanishing near the diagonal. 
As before, the integral~\eqref{eq: causal_fundamental_sol_partial_derivative} is an oscillatory integral that needs to be understood in the sense of a distribution.

We are now going to present a bundle version of the following result by 
Duistermaat-H\"{o}rmander~\cite[Lem. 6.5.4]{Duistermaat_ActaMath_1972}  
(see also~\cite[Lem. 26.1.15]{Hoermander_Springer_2009}). 
The proof simply carries over to the setting of bundles yet we include the details here for completeness. 

% 
% 
% 
% ================================================ 
\begin{lemma} \label{lem: Lem_26_1_15_Hoermander}
	We use the set up and terminology as in Theorem~\ref{thm: microlocalisation_NHOp}: $\square$ and $D_{1}$ are microlocally conjugate to each other, $A, \tilde{A}$ are the conjugating Lagrangian distributions with $m = 0$, and $\cU$ is a conic neighbourhood of any lightlike covector $\xxiNot$ in $T^*M$. 
	Suppose that $F_{1}^{\ret, \adv}$ are the Schwartz kernels of the retarded and advanced fundamental solutions of $D_{1}$ and that $\chi \in C_{\mathrm{c}}^{\infty} (\Rn \times \Rn, \R)$ is identically $1$ in a neighbourhood of the diagonal and vanishes outside another sufficiently small neighbourhood of the diagonal.  
	If $T \in \PsiDO{0}{M; E}$ with $\ES T \subset \cU$ and $F^{\pm} := A (\chi F_{1}^{\ret, \adv}) \tilde{A} T$, then  
	\begin{subequations}
		\begin{eqnarray}
			&& 
			\WFPrime{F^{\pm}} \subset \varDelta \dotCoTanM \cup C^{\pm}, 
			\label{eq: Hoermander_Lem_26_1_15_i}
			\\ 
			&& 
			\square F^{\pm} = T + R^{\pm}, 
			\quad 
			R^{\pm} \in I^{-1/2} \big( M \times M, C^{\pm \prime}; \Hom{E, E} \big), 
			\label{eq: Hoermander_Lem_26_1_15_ii}
			\\ 
			&& 
			F^{+} - F^{-} \in I^{-3/2} \big( M \times M, C'; \Hom{E, E} \big).   
			\label{eq: Hoermander_Lem_26_1_15_iii} 
		\end{eqnarray}
	\end{subequations}
\end{lemma}
% ================================================  
% 
% 
% 
% ================================================ 
\begin{proof}
	The properties~\eqref{eq: Hoermander_Prop_26_1_2_i} and~\eqref{eq: Hoermander_Prop_26_1_2_ii} immediately entail~\eqref{eq: Hoermander_Lem_26_1_15_i} and~\eqref{eq: Hoermander_Lem_26_1_15_iii}. 
	To show~\eqref{eq: Hoermander_Lem_26_1_15_ii}, the definition of $F^{\pm}$ implies   
	\begin{equation} \label{eq: Hoermander_26_1_11}
		\square F^{\pm} 
		= (\square A -{\tilde{L} B} D_{1}) (\chi F_{1}^{\ret, \adv}) \tilde{A} T + {\tilde{L} B} D_{1} (\chi F_{1}^{\ret, \adv}) \tilde{A} T,  
	\end{equation} 
	{where $\tilde{L}$ and $B$ are those appearing in~\eqref{eq: microlocal_conjugate_L_Box_L}.} 
	Since $\square$ and $D_{1}$ are microlocal conjugates to one another,~\eqref{eq: microlocal_conjugate_NHOp} can be re-expressed as (cf.~\eqref{eq: microlocal_conjugate_L_Box_L})  
	\begin{equation} 
		(x_{0}, \xi^{0}; 0, \eta_{1} \rd y^{1}) \notin \WFPrime{\square A - {\tilde{L} B} D_{1}}. 
		\tag{\ref{eq: microlocal_conjugate_NHOp}$^{\prime}$}   
	\end{equation}
	Then, there is a conic neighbourhood $\cV$ of $(0, \eta_{1} \rd y^{1})$ such that $(\square A - {\tilde{L} B} D_{1}) v$ is smooth for any $v \in \cD' (\Rn, \CN)$ when $\WF{v} \subset \cV$. 
	Since $\WFPrime{\chi F_{1}^{\ret, \adv}}$ can be made arbitrary close to the diagonal in $\dotCoTanRn \times \dotCoTanRn$ by choosing the support of $\chi$ close to the diagonal in $\Rn \times \Rn$, we can pick an appropriate $\chi$ and a conic neighbourhood $\mathcal{W}$ of $(0, \eta_{1} \rd y^{1})$ such that $\mathrm{WF} \big( (\chi F_{1}^{\ret, \adv}) v \big) \subset \cV$ provided $\WF{v} \subset \mathcal{W}$. 
	If $\ES{T} \subset \varkappa (\mathcal{W})$ then the first term in the right-hand side of~\eqref{eq: Hoermander_26_1_11} is smooth and we are left with the last term of that  equation. 
	One computes $D_{1} (\chi F_{1}^{\ret, \adv}) = I + \tilde{R}^{\pm}$ where 
	$\tilde{R}^{\pm} := D_{x^{1}} \big(\chi (x, y) \big) \, F_{1}^{\adv, \ret} 
	\in 
	I^{- 1/2} \big( \Rn \times \Rn, C'_{1}; \matN \big)$. 
	Since {$\tilde{L} B \tilde{A} T - T = (\tilde{L} B \tilde{B} L - I) T$} is smooth as long as $\ES{T}$ is sufficiently close to $\xxiNot$, it follows that $\square F^{\pm} - T = R^{\pm}$ where $R^{\pm} - {\tilde{L} B} \tilde{R}^{\pm} \tilde{A} T$  is smooth, which concludes the proof. 
\end{proof} 
% ================================================ 
% 
% 
%  

To end the proof of Theorem~\ref{thm: exist_unique_Feynman_parametrix_NHOp}, we choose a locally finite covering $\{ \sU_{\alpha} \}$ of $\dotCoTanM$ by open cones $\sU_{\alpha}$ such that Lemma~\ref{lem: Lem_26_1_15_Hoermander} is applicable when $\ES{T} \subset \sU_{\alpha}$. 
Denoting the projection of $\sU_{\alpha}$ on $M$ by $U_{\alpha}$ and picking  $\sU_{\alpha}$ so that $U_{\alpha}$ are also locally finite, we set 
\begin{equation}
	I = \sum_{\alpha} T_{\alpha}, 
	\quad 
	\ES{T_{\alpha}} \subset \sU_{\alpha}, 
	\quad 
	\supp{T_{\alpha}} \subset U_{\alpha} \times U_{\alpha}. 
\end{equation}
Then, for every $\alpha$, $A (\chi F_{1}^{\ret, \adv}) \tilde{A} T_{\alpha}$ can be chosen in accord with the Lemma~\ref{lem: Lem_26_1_15_Hoermander} so that $\mathrm{supp} \big( A (\chi F_{1}^{\ret, \adv}) \tilde{A} T_{\alpha} \big) \subset U_{\alpha} \times U_{\alpha}$, and hence the sum 
\begin{equation} 
	F^{\pm} := \sum_{\alpha} A (\chi F_{1}^{\ret, \adv}) \tilde{A} T_{\alpha} 
\end{equation}
is well-defined which satisfies the claimed properties~\eqref{eq: def_Feynman_parametrix} and~\eqref{eq: diff_Feyn_anti_Feyn_parametrix}.  
Note that, if $\chi$ is taken as a function of $y-z$ then $\chi F_{1}^{\ret, \adv}$ is a convolution by a measure of compact support and therefore it continuously maps $H^s_{\mathrm{c}}(\Rn, \CN)$ to $H^s_{\mathrm{c}}(\Rn, \CN)$.
By the mapping properties of pseudodifferential and Fourier integral operators, ~\cite[Thm. 4.3.1]{Hoermander_ActaMath_1971}   
(see also~\cite[Cor. 25.3.2]{Hoermander_Springer_2009}), $A, \tilde{A}, T_{\alpha}$ are 
continuous maps from $H^s_{\mathrm{c}}$ to $H^s_{\mathrm{c}}$. Since the sum of kernels is locally finite this shows that
$F^{\pm}$ continuously map $H_{\mathrm{c}}^{s} (M; E)$ into $H_{\mathrm{loc}}^{s} (M; E)$ for all $s \in \R$.  

Until now, we have just shown that $\square F^{\pm} = I + R^{\pm}$. 
By Lemma~\ref{lem: Hoermander_lem_26_1_16}, we can choose $G^{\pm} \in I^{-3/2} \big( M \times M, C^{\pm \prime}; \Hom{E, E} \big)$ so that $\square G^{\pm} - R^{\pm}$ is smooth. 
Moreover, $G^{\pm}$ extend to continuous mappings from $H_{\mathrm{c}}^{s} (M; E)$ to $H_{\mathrm{loc}}^{s} (M; E)$ for every $s \in \R$. 
This follows from mapping properties of Fourier integral operators, for example, Theorem 25.3.8 in~\cite{Hoermander_Springer_2009},  
bearing in mind that the corank of the symplectic form $\upsigma_{\ms \varGamma}$ (see Appendix~\ref{sec: FIO_symplecto_bundle}) on $\varGamma$ is two. 
Since locally a Fourier integral operator on a bundle is a matrix of scalar Fourier integral operators, this theorem can be applied here directly.
Therefore, 
\begin{equation} \label{eq: Feynman_parametrices_NHOp}
	G^{\Feyn, \aFeyn} := F^{\pm} - G^{\pm}
\end{equation}
is a right parametrix which has this continuity property. 

Furthermore, the construction shows that $F^{+} - F^{-}$ and thus $G^{\Feyn} - G^{\aFeyn}$ is non-characteristic on $\varDelta \, \coLightBun$, as can be seen from the integral representation~\eqref{eq: causal_fundamental_sol_partial_derivative} of $F_{1}$ and wavefront set properties of $A, \tilde{A}$.  
Since $\square (\FeynParametrix - \antiFeynParametrix)$ is smooth, it follows that its principal symbol satisfies a first-order homogeneous differential equation along the lightlike geodesic by Theorem~\ref{thm: Hoermander_thm_25_2_4} and~\eqref{eq: subprincipal_symbol_NHOp}.   
Therefore $G^{\Feyn} - G^{\aFeyn}$ is non-characteristic everywhere: $\mathrm{WF}' (G^{\Feyn} - G^{\aFeyn}) = C$ and then it can be concluded that $\mathrm{WF}' (G^{\Feyn, \aFeyn}) \supset C^{\pm}$ because of $\mathrm{WF}' (G^{\Feyn, \aFeyn}) \subset \varDelta \, \dotCoTanM \cup C^{\pm}$. 
Finally, the fact 
\begin{equation*}
	\varDelta \, \dotCoTanM = \WFPrime{I} = \WFPrime{\square G^{\Feyn, \aFeyn}} \subset  \WFPrime{G^{\Feyn, \aFeyn}}
\end{equation*}
together with Lemma~\ref{lem: Hoermander_lem_26_1_16}, complete the proof. 
% 
% 
% 
% 
% 
% 
% 
% 
% 
% 
% 
% 
%  

% ================================================ 
\subsection{Pauli-Jordan operator as a Lagrangian distribution} 
\label{sec: causal_propagator_NHOp_FIO}
% ================================================ 
Let $E \to M$ be a smooth complex vector bundle over a globally hyperbolic spacetime $(M, g)$ and $\square$ a normally hyperbolic operator on $E$. 
As mentioned in Section~\ref{sec: setup}, $\square$ admits unique advanced $G^{\adv}$ and retarded $G^{\ret}$ propagators whose wavefront sets are given by 
\begin{equation} \label{eq: WF_advanced_retarded_Green_op_NHOp}
	\WFPrime{G^{\ret, \, \adv}} \subset \varDelta \, \dotCoTanM \cup C^{\ret, \, \adv},   
\end{equation}
where the retarded and the advanced geodesic relations are define by 
\begin{equation} \label{eq: def_advanced_retarded_bicharacteristic_relation_NHOp}
	C^{\ret, \, \adv} := \{ \xxiyeta \in \coLightBun \times \coLightBun | \exists s \in \R : \xxi = \Phi_{s} \yeta, x \in J^{\pm} (y) \}.
\end{equation}

In this section we will inscribe the Pauli-Jordan operator   
\begin{equation} \label{eq: def_causal_propagator}
	G := G^{\ret} - G^{\adv} : \comSecE \to C_{\mathrm{sc}}^{\infty} (M; E)
\end{equation}
for $\square$ as a Fourier integral operator. 
The result has been proven by 
Duistermaat-H\"{o}rmander~\cite[Thm. 6.6.1]{Duistermaat_ActaMath_1972}  
for any scalar pseudodifferential operator of real principal type, but a systematic derivation for the bundle version is not so easy to find.  

As a preparation, we remark that there is a natural density $\dVol_{\ms C}$ on the geodesic relation $C$ and hence on the forward/backward geodesic relations $C^{\pm}$, as originally due to 
Duistermaat and H\"{o}rmander~\cite[p. 230]{Duistermaat_ActaMath_1972} 
for a generic manifold $M$, which simplifies considerably for a globally hyperbolic spacetime $(M, g)$ as reported by 
Strohmaier and Zelditch~\cite[$(52)$, Rem. 7.1]{Strohmaier_AdvMath_2021}. 
By definition~\eqref{eq: def_bicharacteristic_relation_NHOp}, for each $\xxiyeta \in C$ there is a unique $s \in \R$ such that $\xxi = \varPhi_{s} \yeta$ so that $C$ can be identified with an open subset of $\R \times \coLightBun$. 
On a globally hyperbolic spacetime $(M, g)$, the set of lightcones $\coLightBun \to M$ has the structure of a conic contact manifold where the Hamiltonian  
\begin{equation} \label{eq: def_Hamiltonian_metric}
	H_{g} : C^{\infty}(M; T^{*} M) \to \R, ~ \xxi \mapsto H_{g} \xxi := \frac{1}{2} g_{x}^{-1} (\xi, \xi) 
\end{equation}
induced by the spacetime metric $g$ vanishes identically and its Hamiltonian reduction is the the space $\mathcal{N}$ of scaled-lightlike geodesic strips. 
Note that $\mathcal{N}$ is a symplectic manifold and hence it admits the Liouville form $\dVol_{\ms \mathcal{N}}$. 
The prefactor $1/2$ in the preceding definition has been used in order to identify the relativistic Hamiltonian flow with the geodesic flow. 
Denoting by $\tilde{s}$, the dilation parameter on $\coLightBun$, the natural halfdensity on $C$ is given by 
\begin{equation}
	\sqrt{|\dVol_{\ms C}|} := \sqrt{|\rd s|} \otimes \sqrt{|\rd \tilde{s}|} \otimes \sqrt{| \dVol_{\ms \mathcal{N}}|}. 
\end{equation}  
Note, this halfdensity differs from that by Duistermaat-H\"{o}rmander by a factor of $2$ because they used the Hamiltonian flow of $g^{-1}$ to parametrise $\mathcal{N}$, in contrast to the flow of the Hamiltonian vector field $X_{g/2}$ generated by $H_{g}$. 
Moreover, this density is conserved under geodesic flow: $\pounds_{X_{g/2}} \dVol_{\ms C} = 0$.  
The densities on $C^{\pm}$ follow from the fact $\coLightBun = \dot{T}^{*}_{0, +} M \sqcup \dot{T}^{*}_{0, -} M$ in $n \geq 3$.

Employing 
Theorems~\ref{thm: exist_unique_Feynman_parametrix_NHOp},~\ref{thm: Hoermander_thm_25_2_4} and Definition~\ref{def: P_compatible_connection}, we imitate the proof for the scalar version~\cite[Thm. 6.6.1]{Duistermaat_ActaMath_1972}  
to obtain

%
%
% 
% ================================================ 
\begin{theorem} \label{thm: causal_propagator_NHOp_FIO}
	Let $E \to M$ resp. $\pi: \dotCoTanM \to M$ be a vector bundle resp. the punctured cotangent bundle over a globally hyperbolic spacetime $(M, g)$ and $\square$ a normally hyperbolic operator on $E$.  
	Then, the Schwartz kernel $G$ of the Pauli-Jordan operator for $\square$ is the Lagrangian distribution 
	\begin{subequations}
		\begin{eqnarray*}
			&& G \in I^{-3/2} \big( M \times M, C'; \Hom{E, E} \big), 
			\\ 
			&& \symb{G} = \frac{\ri}{2} \sqrt{2 \pi |\rd \mathsf{v}_{\ms C}|} \otimes \mathbbm{m} w, 
			\\ 
			&& \Char{G} \cap C = \emptyset,  
		\end{eqnarray*}
	\end{subequations}
	where $\dVol_{\ms C}$ is the natural volume form on the geodesic relation $C$, $\mathbbm{m}$ is the section of the Keller-Maslov bundle $\Maslov \to C$ 
 	constructed in~\cite[pp. 231-232]{Duistermaat_ActaMath_1972}, and $w$ is the unique element of $C^{\infty} \big( C; \pi^{*} \Hom{E, E}_{C} \big)$ that is diagonally the identity endomorphism and off-diagonally covariantly constant  
	\begin{equation*} \label{eq: symbol_causal_propagator_NHOp_covarinatly_constant}
		\nabla_{X_{g/2}}^{\ms \pi^{*} \Hom{E, E}} w = 0    
	\end{equation*}
	with respect to the $\square$-compatible Weitzenb\"{o}ck covariant derivative $\nabla_{X_{g/2}}^{\ms \pi^{*} \Hom{E, E}}$ along the geodesic vector field $X_{g/2}$.
\end{theorem}
% ================================================ 
% 
% 
% 
% 
% 
% 
% 
% 
% 
% 
% ================================================ 
\section{Hadamard states from Feynman parametrices}
% ================================================ 
In this section we will show that the unique (up to smoothing operators) Feynman parametrix $\FeynParametrix := \eqref{eq: Feynman_parametrices_NHOp}$ constructed in Theorem~\ref{thm: exist_unique_Feynman_parametrix_NHOp} can be turned into a Green's operator $\FeynGreenOp$ - the Feynman propagator for $\square$. 
We will see later that this is equivalent to constructing a Hadamard $2$-point distribution.  
As remarked by 
Duistermaat-H\"{o}rmander~\cite[Thm. 6.6.2]{Duistermaat_ActaMath_1972},  
in the case of functions there is a remarkable positivity property. 
Namely, that one can modify $\FeynParametrix$ by a smoothing operator $R$ such that
\begin{equation*}
	W + R := - \ri (\FeynParametrix - \advGreenOp) + R
\end{equation*}
is non-negative in the sense that $(W + R) (\bar{u} \otimes u) \geq 0$ for any $u \in C_{\mathrm{c}}^{\infty} (M)$. 
This means, of course that, there exists a Feynman parametrix that has this property with $R = 0$. 
Thus, there is a Feynman parametrix that satisfies $W \geq 0$ in the above sense. 
We state this here for the case of vector bundles.

% 
% 
% 
% ================================================ 
\positivityFeynmanMinusAdvNHOp*
% ================================================ 

% 
% 
% 
% ================================================ 
\begin{proof}
	% selfadjoint
	Taking the adjoint $\FeynParametrix {}^{*} \square \equiv I$ of the Feynman parametrix $\FeynParametrix$ we find that $\WFPrime{\FeynParametrix {}^{*}} \subset \varDelta \, \dotCoTanM \cup C^{-}$ by Definition~\ref{def: Feynman_parametrix}.  
	Thus $\antiFeynParametrix:=\FeynParametrix {}^{*} $ is an anti-Feynman parametrix. 
	Now differences of parametrices are solutions of the homogeneous equation up to smooth kernels. By propagation of singularities (Theorem \ref{thm: propagation_singularity_Sobolev_WF}), differences of parametrices have a wavefront set that consists of lightlike covectors and is invariant in both arguments under the geodesic flow. It follows that
	\begin{align*} 
 		\WFPrime{\FeynParametrix-\advGreenOp} \subset C^+, \quad \WFPrime{\antiFeynParametrix-\retGreenOp}\subset C^+,
 		\\
 		\WFPrime{\antiFeynParametrix-\advGreenOp} \subset C^-, \quad \WFPrime{\FeynParametrix-\retGreenOp}\subset C^-.
	\end{align*}
	Since 
	$
	\FeynParametrix + \antiFeynParametrix - (\advGreenOp + \retGreenOp) =
	\FeynParametrix-\advGreenOp + \antiFeynParametrix-\retGreenOp = \antiFeynParametrix-\advGreenOp +\FeynParametrix-\retGreenOp
 	$
	it follows that $ \WFPrime{\FeynParametrix + \antiFeynParametrix -( \advGreenOp + \retGreenOp)} = C^+ \cap C^- = \emptyset$.
	In other words, we have  
	\begin{equation} \label{eq: Feynman_plus_antiFeynman_adv_plus_ret_NHO}
		\FeynParametrix + \antiFeynParametrix \equiv \advGreenOp + \retGreenOp. 
	\end{equation}
	As $(\retGreenOp)^* =  \advGreenOp$ this implies that $(\FeynParametrix - \advGreenOp)^* + (\FeynParametrix - \advGreenOp) \equiv 0$. Therefore, $-\ri (\FeynParametrix - \advGreenOp)$ can be chosen formally self-adjoint 
	by modifying $\FeynParametrix$ by a smoothing operator.

	The correctness of the preceding relation \eqref{eq: Feynman_plus_antiFeynman_adv_plus_ret_NHO} can also be seen as follows.	The distinguished parametrices are unique and therefore coincide, modulo smoothing operators, with the ones constructed via microlocalisation. 
	For each element in the characteristic set we use microlocalisation in a neighbourhood and have to make a choice of propagator amongst the two distinguished propagators for $D_{1}$. 
	The choices for Feynman and anti-Feynman are opposites in the sense that for each such point they correspond to opposite choices. 
	The same is true for the retarted and advanced fundamental solutions. 
	Therefore the sums coincide, modulo smoothing operators, as both sums represent microlocally the sum of retarted and advanced propagator.
	% 
	% 
	% 
	
	% positivity 
	To address the positivity, first consider the operator $D_{1}$ as in Section~\ref{sec: microlocalisation_NHOp} and observe that $- \ri$ times its causal propagator $F_{1}$ is positive by~\eqref{eq: causal_fundamental_sol_partial_derivative}:   
	\begin{equation} \label{eq: positivity_causal_fundamental_sol_partial_derivative}
		- \ri F_{1} (\bar{u}^{*} \otimes u) \geq 0. 
	\end{equation}

	To account for $W$, we note that $A$ and $A^{*}$ in Theorem~\ref{thm: microlocalisation_NHOp} preserve wave front sets and thus, one obtains $(x_{0}, \xi^{0}; 0, \eta_{1} \rd y^{1}) \notin \WFPrime{W - A F_{1} A^{*}}$ from parametrix differences. Patching such microlocally constructed solutions together is however a priori not preserving positivity. Therefore we  need a refined version of microlocalisation and hence respective homogeneous Darboux theorem. 

	On an $n$-dimensional globally hyperbolic spacetime $(M, g)$ with Cauchy hypersurface $(\Sigma, h)$, the space of lightlike (unparametrised) geodesics is a $2n - 3$-dimensional smooth conic manifold which can be naturally identified with the cotangent unit sphere bundle $\mathbb{S}^{*} \Sigma \to \Sigma$ with respect to the metric $h$~\cite{Low_NA_2001}. 
	We choose $L$, a formally selfadjoint elliptic operator of order $-1/2$ with scalar principal symbol and vanishing subprincipal symbol as in the proof of Theorem~\ref{thm: microlocalisation_NHOp}. 
	One can always choose the principal symbol of $L$ such that $X_{L \square L}$ is a complete vector field~\cite[Thm. 6.4.3 and p. 234]{Duistermaat_ActaMath_1972} 
	(see also~\cite[Lem. 26.1.11 and 26.1.12]{Hoermander_Springer_2009},~\cite[Prop. 2.1]{Strohmaier_AdvMath_2021}).  
	Since the bundle of forward/backward lightlike covectors $\dot{T}_{0,\pm}^{*} M \to M$ are connected components of the bundle of lightlike covectors $\coLightBun \to M$, there exists a homogeneous diffeomorphism $\coLightBun \to (\mathbb{S}^{*} \Sigma \oplus \mathbb{S}^{*} \Sigma) \times \R \times \R_{+}$ such that $X_{L \square L}$ is mapped to the vector field $\partial / \partial s$ when the variable in $\mathbb{S}^{*} \Sigma \times \R \times \R_{+}$ are denoted by $(\cdot, s, \cdot)$. 	
	In this setting, a refined Darboux theorem yields that, for every lightlike covectors $\xxiNot$ on $(M, g)$, there is an open conic coordinate chart $\big( \cU, (x^{1}, \ldots, x^{n}; \xi_{1}, \ldots, \xi_{n}) \big)$ and a homogeneous symplectomorphism   
	\begin{equation*} \label{eq: def_bijective_symplecto}
		\kappa : \cU \to \cU', \; \xxi \mapsto 
		\kappa \xxi := \big( y^{1} \xxi, \ldots, y^{n} \xxi; \eta_1  \xxi, \ldots, \eta_{n} \xxi \big), 
	\end{equation*}
	bijectively mapping $\cU$ to an open conic neighbourhood $\cU'$ of $\kappa \xxiNot := (0, \eta_{1} \rd y^{1})$ in $\dotCoTanRn$ such that 
	\begin{equation}
		\big( \kappa^{-1} \big)^{*} \symb{P} = \xi_{1} \, \one_{\End{E}}, 
	\end{equation}
	$\Char{D_{1}} := \{ \yeta \in \Rn \times \dot{\R}^{n} ~|~ y^{1} = 0 \}$ is symmetric with respect to the plane $y^{1} = 0$ and convex in the direction of $y^{1}$-axis, and $\cU' \cap \Char{D_{1}}$ is invariant under the translation along the $y^{1}$-axis. 
	Since $\rd \kappa^{-1} (X_{L^{*} \square L}) = \one_{\matN} \partial / \partial y^{1}$, it follows that $\cU \cap \coLightBun$ is invariant under $X_{\square}$~\cite[Lem. 6.6.3]{Duistermaat_ActaMath_1972}. 

	Suppose that $\cV_{1}, \cV_{2}$ are closed conic neighbourhoods of $\xxiNot$ such that $\cV_{1} \subset \cU$ and $\cV_{2} \subset \mathrm{int} (\cV_{1})$ while $\cV_{\iota} \cap \coLightBun, \iota = 1, 2$ are invariant under the geodesic flow and that $\kappa (\cV_{\iota})$ are convex in the $y^{1}$-direction and symmetric about the plane $y^{1} = 0$. 
	Remark~\ref{rem: global_microlocalisation_NHOp} then entails that there exists a properly supported $A \in I^{-1/2} \big( M \times \Rn, \varGamma'_{1}; \Hom{\C^{N}, E} \big)$ so that $A^{*} A$ and $A A^{*}$ both are pseudodifferential operators of order $-1$ and  
	\begin{subequations} \label{eq: microlocalisation_NHOp_refined}
		\begin{eqnarray}
			&& 
			\Char{A} \cap \varGamma'_{2} = \emptyset,
			\label{eq: microlocalisation_NHOp_1_refined} 
			\\ 
			&& 
			{\Char{A^{*} A}} \cap \kappa(\cV_{2}) = \emptyset  
			,\quad
			{\Char{A A^{*}}} \cap \cV_{2} = \emptyset, 
			\label{eq: microlocalisation_NHOp_2_refined}
			\\ 
			&& 
			\ES{A^{*} \square A - D_{1}} \cap \varGamma_{2} = \emptyset, 
			\label{eq: microlocal_conjugate_NHOp_refined}
		\end{eqnarray}
	\end{subequations} 
	where $\varGamma_{1}$ and $\varGamma_{2}$ are the graphs of the restriction of $\kappa$ to $\cV_{1}$ and to $\cV_{2}$, respectively.

	It follows from Theorem~\ref{thm: exist_unique_Feynman_parametrix_NHOp} and~\eqref{eq: microlocalisation_NHOp_refined} that we can microlocally conjugate $W$ to $\ri F_{1}$ 
	(see~\cite[p. 237]{Duistermaat_ActaMath_1972} for the scalar version):  
	\begin{equation} \label{eq: FIO_6_6_14_bundle}
		\mathrm{WF}' \big( W - A (- \ri F_{1}) A^{*} \big) \cap (\cV_{2} \times \cV_{2}) = \emptyset, 
		\quad 
		\cV_{2} \cap \coLightBun \subset \big\{ \dot{T}_{0,\pm}^{*} M \big\}.  
	\end{equation} 

	% 
	% 
	% 
	% ================================================ 
	\begin{lemma} 
	\label{lem: FIO_lem_6_6_5_bundle}
		For a normally hyperbolic operator $\square$ as in Proposition~\ref{thm: positivity_Feynman_minus_adv_NHOp} and Lagrangian distributions $A$ given by~\eqref{eq: microlocalisation_NHOp_refined}, let $\varPsi \in \PsiDO{0}{M; E}$ such that $\ES{[\varPsi, \square]} \cap \coLightBun = \emptyset$ and $\ES{\varPsi} \subset \cV_{2}, \coLightBun \cap \cV_{2} \subset \dot{T}_{0, \pm}^{*} M$. 
		Then 
		\begin{equation*}
			W \varPsi \varPsi^{*} = \varPsi A (-\ri F_{1}) A^{*} \varPsi^{*} \mod C^{\infty} \big( M \times M; \Hom{E, E} \big).  
		\end{equation*}
	\end{lemma}
	% ================================================ 
	% 
	% 
	% 
	% ================================================ 
	\begin{proof}
	 	The idea of the proof is to use microlocalisation to reduce this to the special case $P=D_1$.
		To reach the conclusion our main task is then to compute the relevant commutators which is straightforward and exactly analogous to the scalar counterpart~\cite[Lem. 6.6.5]{Duistermaat_ActaMath_1972}. 
		So, we only sketch the main steps for completeness. 
		The ellipticity~\eqref{eq: microlocalisation_NHOp_1_refined} allows us to relate any properly supported $\varPsi \in \PsiDO{0}{M; E}$ and $\varPsi' \in \PsiDON{0}$ with $\ES{\varPsi} \subset \cV_{2}$ and $\ES{\varPsi'} \subset \cV'_{2}$ by 
		\begin{equation*}
			\varPsi' \equiv A^{*} \varPsi A, 
			\quad 
			\varPsi \equiv A \varPsi' A^{*}. 
		\end{equation*} 
		A direct computation shows $\ES{[\varPsi', D_{1}]} \cap \Char{D_{1}} = \emptyset$ which implies that the derivative of the symbol of $\varPsi'$ with respect to $y^{1}$ is of order $-\infty$ in a neighbourhood of $\Char{D_{1}}$. 
		Denoting the convolution by the Dirac measure at $(h, 0, \ldots, 0)$ by $\tau_{h}$, we can rewrite~\eqref{eq: causal_fundamental_sol_partial_derivative} as $F_{1} = \int \tau_{h} \rd h$ and observe that $\tau_{h} \varPsi'  - \varPsi' \tau_{h} $ is of order $-\infty$ in a conic neighbourhood of $\Char{D_{1}}$. 
		Since $\Char{D_{1}}$ is invariant under translation in the $y^{1}$-direction, one has  
		\begin{equation*}
			\forall v \in \cE'(\Rn, \CN) : \WF{[\tau_{h}, \varPsi'] v} \cap \Char{D_{1}} = \emptyset. 
		\end{equation*}
		The result follows after the integration with respect to $h$ and a few algebraic manipulations.
	\end{proof}
	% ================================================ 
	% 
	% 
	% 
	
	One observes that $\varPsi A (-\ri F_{1}) A^{*} \varPsi^{*}$ is non-negative by~\eqref{eq: positivity_causal_fundamental_sol_partial_derivative} since $A A^{*}$ is non-negative with respect to the hermitian form on $E$. 
	Thus, the only thing left to conclude Proposition~\ref{thm: positivity_Feynman_minus_adv_NHOp} is that the identity can be expressed as a sum of operators of the form $\varPsi \varPsi^{*}$ discussed in Lemma~\ref{lem: FIO_lem_6_6_5_bundle}, provided the prerequisites in this lemma are satisfied.  
	Aiming to show this, we consider a closed conic neighbourhood $\cV_{3}$ of any lightlike covector $\xxiNot$ on $(M, g)$ where $\cV_{3} \subset \mathrm{int} (\cV_{2})$ and $\cV_{3} \cap \coLightBun$ is invariant under the geodesic flow.
	One can prove that there exists $\varPsi \in \PsiDO{0}{M; E}$ with $\ES{\varPsi} \subset \cV_{2}$ such that $\varPsi$ has a non-negative (homogeneous) principal symbol equal to $1$ on $\cV_{3}$ and $\ES{[\varPsi, \square]} \cap \coLightBun = \emptyset$~\cite[Lem. 6.6.6]{Duistermaat_ActaMath_1972}.
	In other words, the hypotheses in Lemma~\ref{lem: FIO_lem_6_6_5_bundle} can be satisfied for $\varPsi$.
	Taking a suitable cover of $\coLightBun$, one can use these operators to construct a family of operators $\varPsi_{\alpha} \in \PsiDO{0}{M; E}$ satisfying 
	\begin{subequations}
		\begin{eqnarray}
			&& \ES{\varPsi_{\alpha}} \bigcap \coLightBun \subset \dot{T}_{0,\pm}^{*} M , 
			\label{eq: WF_POU_1}
			\\ 
			&& \mathrm{ES} \bigg( I - \sum_{\alpha} \varPsi_{\alpha} \varPsi_{\alpha}^{*} \bigg) \bigcap  \dot{T}_{0, \pm}^{*} M  = \emptyset, 
			\label{eq: WF_POU_2}
			\\ 
			&& \ES{[\square, \varPsi_{\alpha}]} \bigcap \coLightBun = \emptyset. 
			\label{eq: WF_commutator_NHOp_POU} 
		\end{eqnarray}
	\end{subequations}
	The construction of such a family that provides a microlocal partition of unity	is carried out in detail in~\cite[p. 238]{Duistermaat_ActaMath_1972}. 
	The proof follows the usual stategy of summing the operators and multiplying from left and right by the root of a parametrix. One can verify that the hypotheses in Lemma~\ref{lem: FIO_lem_6_6_5_bundle} are preserved under this construction. 
	To finalise the proof of Proposition~\ref{thm: positivity_Feynman_minus_adv_NHOp} we simply write 
	\begin{equation} \label{eq: Feynman_parametrix_P_mod_partial_deri}
		W 
		= \sum_{\alpha} W \varPsi_{\alpha} \varPsi_{\alpha}^{*} 
		\equiv \sum_{\alpha} \varPsi_{\alpha} A_{\alpha} \big(-\ri F_{1}) A_{\alpha}^{*} \varPsi_{\alpha}^{*},   
	\end{equation}
	which terminates the proof as every term on the right hand side of~\eqref{eq: Feynman_parametrix_P_mod_partial_deri} is nonnegative with respect ot the hermitian form on $E$. 
\end{proof}
% ================================================ 
% 
% 
% 

Of course, at the end one is interested in true solutions rather than parametrices. 
In fact there is a simple way, employing well-posedness of the Cauchy problem for $\square$,  modifying the Feynman parametrix into a Feynman propagator preserving the aforementioned positivity. 
% 
% 
% 
% ================================================ 
\existenceFeynmanGreenOpNHOp*
% ================================================ 
% 
% 
% 

We postpone to proof this to make the following observations. 
Let $(\Sigma, h)$ be a Cauchy hypersurface of $(M, g)$ and assume $f_{0}, f_{1} \in C^{\infty}_{\mathrm{c}} (\Sigma, E_{\Sigma})$. 
Since $M$ is globally hyperbolic, we will denote a point in $M$ as $(t, x) \in \R \times \Sigma$. 
Then we can define the distributions $f_{1} \otimes \delta_{\Sigma}, f_{0} \otimes \delta'_{\Sigma} \in \cE'(M; E)$ by
\begin{equation} \label{eq: def_distributional_Cauchy_data_NHOp}
	f_{1} \otimes \delta_\Sigma (\phi)  
	:= 
	\int_\Sigma \phi (x) \, f_{1} (x) \, \dVolh (x), 
	\quad  
	f_{0} \otimes \delta'_\Sigma (\phi)  
	:= 
	\int_\Sigma (\partial_{N} \phi) (x) \, f_{0} (x) \, \dVolh (x), 
\end{equation}
where $\phi \in C_{\mathrm{c}}^{\infty} (M; \bar{E}^{*})$, $N$ is a future directed unit normal vector field on $M$ along $\Sigma$, and $\dVolh$ is the Riemannian volume form on $\Sigma$; recall that $\bar{E}^{*}$ is the conjugate-dual bundle (Section~\ref{sec: adjoint_FIO}) of $E$ induced by the sesquilinear form $(\cdot|\cdot)$ on $E$. 
We denote by $\cW_\Sigma \subset \cE'(M; E)$, the span of the set of distributions of this form. By a duality argument both retarded and advanced fundamental solutions extend to continuous maps $\mathcal{E}'(M;E) \to \mathcal{D}'(M;E)$, and if $G$ is the causal propagator (Theorem~\ref{thm: causal_propagator_NHOp_FIO}) and
$f := f_{0} \otimes \delta'_\Sigma + f_{1} \otimes \delta_\Sigma$, then $u = G(f)$ is a smooth solution of $\Box u=0$ with Cauchy data $(f_0,f_1)$ on $\Sigma$ 
(see e.g.~\cite[Lem 3.2.2]{Baer_EMS_2007},~\cite[Thm. 3.20]{Jubin_LMP_2016}).

Assume that $w \Box$ and $\bar{\Box}^{*} w$ are smooth for any bidistribution $w \in \cD'(M \times M; E \boxtimes \bar{E}^*)$.  
Let us denote by $\cE'_{\mathrm{N}^*\Sigma} (M; E)$, the set of compactly supported distributions with wavefront set contained in the conormal bundle $\mathrm{N}^*\Sigma$ of $\Sigma$. 
This space is endowed with a natural 
locally convex topology~\cite[p. 125]{Hoermander_ActaMath_1971} 
with respect to which $C_{\mathrm{c}}^{\infty} (M; E)$ is sequentially dense in $\cE'_{\mathrm{N}^*\Sigma} (M; E)$ 
(see e.g.~\cite[Thm. 8.2.3]{Hoermander_Springer_2003} 
or the 
exposition~\cite[Sec. 4.3.1]{Strohmaier_Springer_2009} 
for details). 
The bidistribution $w$ can be defined as a sequentially continuous bilinear form on $\cE'_{\mathrm{N}^*\Sigma} (M; E)$.
Since $\cW_\Sigma \subset \cE'_{\mathrm{N}^*\Sigma} (M; E)$, the form $w$ is then also defined on $\cW_\Sigma$.

% 
% 
% 
% ================================================ 
\begin{lemma} \label{lem: positive_bisolution_NHOp}
	Let $\big( E \to M, (\cdot|\cdot) \big)$ be a smooth complex hermitian vector bundle over a globally hyperbolic spacetime $(M, g)$ whose Cauchy hypersurface is $\Sigma$. 
	Suppose that $\square$ is a normally hyperbolic operator on $E$ and that $w \in \cD'(M \times M, E \boxtimes \bar{E}^*)$ is a bisolution of $\square$. 
 	Then, $w$ is non-negative with respect to $(\cdot|\cdot)$ if and only if $w (\bar{f}^{*} \otimes f) \geq 0$ for any $f \in \cW_\Sigma$ where $\cW_{\Sigma}$ is the span of distributional Cauchy data~\eqref{eq: def_distributional_Cauchy_data_NHOp} of $\square$. 
\end{lemma}
% ================================================ 
% 
% 
% 
% ================================================ 
\begin{proof}
	First note that, since $w$ is a bisolution it extends to a separately sequentially continuous bilinear form on $\cE'_{\mathrm{N}^*\Sigma} (M; \bar{E}^*) \times \cE'_{\mathrm{N}^*\Sigma} (M; E)$. 
 	We denote the extension by the same letter. 
 	As this form is a bisolution it vanishes on 
 	$\bar{\Box}^* C^\infty_{\mathrm{c}} (M; \bar{E}^*) \times C^\infty_{\mathrm{c}} (M; E)$ 
 	and also on 
 	$C^\infty_{\mathrm{c}} (M; \bar{E}^*) \times \Box C^\infty_{\mathrm{c}} (M; E)$. 
 	By sequential continuity and sequential density of $C^\infty_{\mathrm{c}}$ in $\cE'_{\mathrm{N}^*\Sigma}$, it follows that $w$ vanishes on 
 	$\bar{\Box}^* \cE'_{\mathrm{N}^*\Sigma} (M; \bar{E}^{*}) \times \cE'_{\mathrm{N}^*\Sigma} (M; E)$ 
 	and on
 	$\cE'_{\mathrm{N}^*\Sigma} (M; \bar{E}^{*}) \times \Box \cE'_{\mathrm{N}^*\Sigma} (M; E)$. 
 	We will show that any 
 	$u \in C^\infty_{\mathrm{c}} (M; E)$ can be written as $u = f + \Box \phi$ for some $f \in \cW_\Sigma$ and for some $\phi \in \cE'_{\mathrm{N}^*\Sigma} (M; E)$, and the analogous statement holds for the conjugate-dual bundle and the conjugate-dual normally hyperbolic operator. 
 	This obviously implies the statement. 

 	To see that this is indeed the case, observe that $\Box (G u) = 0$ and therefore, if $(f_{0}, f_{1})$ is the Cauchy data of $G(u)$ on $\Sigma$ and $f$ the corresponding element in 
 	$\cW_{\Sigma}$, then we have $G (u) = G (f)$. 
 	This implies that $G (u - f) = 0$.
 	Hence $G^{\ret}$ and $G^{\adv}$ applied to $u - f$ are the same. 
 	Thus $\phi := \retGreenOp (u - f)$ must be a compactly supported distribution because of the support properties $G^{\ret, \adv}$ and the global hyperbolicity of $M$. 
 	By propagation of singularities (Theorem~\ref{thm: propagation_singularity_Sobolev_WF}), this distribution must be again in $\cE'_{\mathrm{N}^*\Sigma} (M; E)$ and we have $u - f = \Box \phi$. 
 	The same proof works for the bundle $\bar{E}^*$ and the operator $\bar{\Box}^*$.
\end{proof}
% ================================================ 
% 
% 
% 

Finally, we are prepared to show the existence of Feynman propagators for $\square$. 
Moreover, if there exists a hermitian form on $E$ then Feynman propagators can be chosen non-negative with respect to this form and hence implying existence of $2$-point bidistributions corresponding to Hadamard states. 

% 
% 
%  
% ================================================ 
\begin{proof}[Proof of Theorem~\ref{thm: existence_Feynman_propagator_NHOp}]
 	Let $W$ be the parametrix for $\square$ constructed in Proposition~\ref{thm: positivity_Feynman_minus_adv_NHOp}. 
 	Since the wavefront set of $W$ does not intersect with the conormal bundle $\mathrm{N}^{*} \Sigma$ of any Cauchy hypersurface $\Sigma$ in $M$, we fix a $\Sigma$ and restrict $W$ and its normal derivative to create distributional Cauchy data on $\Sigma \times \Sigma$.
 	Suppose that $E$ is endowed with a hermitian form $(\cdot|\cdot)$ such that $W$ is non-negative with respect to $(\cdot|\cdot)$. 
 	Then the restriction of $W$ on $\Sigma$ is non-negative as well. 
 	Next we construct $\omega \in \cD' (M \times M; E \boxtimes \bar{E}^*)$ employing the solution operator with the same distributional Cauchy data as $W$ on $\Sigma \times \Sigma$. 
 	This will be a bisolution of $\square$ with the property 
 	\begin{equation}
 		\left( \omega (\bar{u}^{*} \otimes u) \geq 0,  \quad \forall u \in \mathcal{E}'_{\mathrm{N}^*\Sigma} (M; E) \right)
 		\Leftrightarrow \omega \geq 0, 	
	\end{equation}
 	by the sequential density of $C_{\mathrm{c}}^{\infty} (\Sigma; E_{\Sigma})$ in  $\mathcal{E}'_{\mathrm{N}^* \Sigma} (M; E)$ and Lemma~\ref{lem: positive_bisolution_NHOp}.
 	Note that $R = \omega - W$ is smooth since it solves the inhomogeneous problem with zero Cauchy data, which entails that 
 	\begin{equation}
 		\FeynGreenOp := \FeynParametrix + \ri R
 	\end{equation}
 	is the Feynman propagator with the required positivity property.
\end{proof}
% ================================================ 
% 
% 
% 
% 
% 
% 
% 
% 
% 
% 
% ================================================ 
\section{Feynman propagators for Dirac-type operators}
% ================================================ 
In this section we will construct Feynman propagators for a Dirac-type operator. 
We recall that a first-order linear partial differential operator $D$ on a vector bundle $E \to M$ over a Lorentzian manifold $(M, g)$ is called \textit{Dirac-type}, if its principal symbol $\symb{D}$ satisfies the Clifford relation 
(see e.g.~\cite[Sec. 2.5, 3.2]{Baer_Springer_2012} for details):  
\begin{equation*}
	\big( \symb{D} \xxi \big)^{2} = g_{x}^{-1} (\xi, \xi) \, \one_{\End{E}}   
\end{equation*} 
for any $\xxi \in T^{*} M$. 
Hence, the operator $D^{2}$ is normally hyperbolic and $\symb{D}$ defines a Clifford multiplication which turns $E$ into a bundle $(E \to M, \symb{D})$ of Clifford modules over $M$. 
We equip $E$ with a \textit{sesquilinear form} $(\cdot|\cdot)$ such that $D$ is formally selfadjoint. 
Then, under the assumption mentioned in Section~\ref{sec: setup}, one has a \textit{hermitian form} $\langle  \cdot|\cdot \rangle$ on $E$ satisfying~\eqref{eq: positive_definite_sesequilinear_form_Dirac_type_op}. 
Of course, this positive-definite inner product depends on the choice of the unit normal vector field. 

If $D$ is formally selfadjoint with respect to $(\cdot|\cdot)$ then the first condition in Theorem~\ref{thm: existence_Feynman_propagator_NHOp} is satisfied and we obtain Feynman propagator $\mathsf{S}^{\Feyn}$ for $D$ immediately by setting $\mathsf{S}^{\Feyn} := D \mathsf{G}^{\Feyn}$ where $\mathsf{G}^{\Feyn}$ is the Feynman propagator for $D^{2}$. 
Nevertheless, the existence of Hadamard states cannot be deduced due to the positivity issue. 
One may wish to consider formally selfadjointness of $D$ with respect to $\langle \cdot | \cdot \rangle$ so that the hypotheses of Theorem~\ref{thm: existence_Feynman_propagator_NHOp} are satisfied resulting existence of Hadamard states. 
But this turns out to be too restrictive to encompass all Dirac-type operators on a globally hyperbolic spacetime (see Sections~\ref{sec: twisted_Dirac_op} and~\ref{sec: Rarita_Schwinger_op} for concrete examples) and so we refrain to impose the condition. 
However, we will show below that there is a more direct way to construct Feynman parametrices satisfying positivity, and hence Hadamard states. 

To begin with, recall that the manifold of lightlike covectors $\Char{D^{2}} = \coLightBun$ on a globally hyperbolic spacetime has only two connected components   
when\footnote{In 
	two dimensions, $\coLightBun$ has four connected components.}:
$n \geq 3$: the forward lightcone $\dot{T}_{0,+}^{*} M$ and the backward lightcone $\dot{T}_{0,-}^{*} M$.
Therefore, the geodesic relation $C :=$~\eqref{eq: def_bicharacteristic_relation_NHOp} has four different 
orientations\footnote{By 
	an orientation of a geodesic relation $C$, it is meant that any splitting of $C \setminus \varDelta \, \coLightBun$ into a disjoint union of open $C^{+}, C^{-} \subset C$ which are inverse relations~\cite[p. 218]{Duistermaat_ActaMath_1972}.
}
\cite[p. 218]{Duistermaat_ActaMath_1972} 
(see also~\cite[p. 540]{Radzikowski_CMP_1996})  
$(C^{+}, C^{-}), (C^{-}, C^{+}), (C^{\adv}, C^{\ret}), (C^{\ret}, C^{\adv})$ corresponding to the set $\{ \coLightBun, \dot{T}_{0,+}^{*} M, \dot{T}_{0,-}^{*} M, \emptyset \}$ of connected components of $\coLightBun$, where $C^{\pm}$ and $C^{\ret, \, \adv}$ are defined by~\eqref{eq: def_Feynman_bicharacteristic_relation_NHOp} and~\eqref{eq: def_advanced_retarded_bicharacteristic_relation_NHOp}, respectively. 
This enables to have a microlocal decomposition of Pauli-Jordan operator $G$ for $D^{2}$ as
\begin{subequations} 
	\begin{eqnarray}
		&& 
		G \equiv W^{+} - W^{-}, 
		\\ 
		&& 
		W^{\pm} D^{2} \equiv D^{2} W^{\pm} \equiv 0, 
		\\ 
		&& 
		W^{\pm} \in I^{-3/2} \big( M \times M, C^{\pm \prime}; \Hom{E, E} \big),   
	\end{eqnarray}
\end{subequations}
where $W^{+} := - \ri (\FeynParametrix - \advGreenOp), W^{-} := - \ri (\retGreenOp - \antiFeynParametrix)$ and we recall that $\equiv$ means modulo smoothing kernels. 
This decomposition follows from~\eqref{eq: Feynman_plus_antiFeynman_adv_plus_ret_NHO}. 

We look for an analogue of the preceding result for $D$. 
In particular, $\Char{D} = \Char{D^{2}} = \coLightBun$. 
On a globally hyperbolic spacetime $(M, g)$, $D$ admits unique advanced, retarded and non-unique (anti-)Feynman Green's operators:  
\begin{equation*}
	S^{\adv, \, \ret,  \, \Feyn, \, \aFeyn} := D G^{\adv, \, \ret, \, \Feyn, \, \aFeyn}.  
\end{equation*}
We decompose the Pauli-Jordan operator 
\begin{equation} \label{eq: def_causal_propagator_Dirac}
	S := S^{\ret} - S^{\adv} : \comSecE \to C_{\mathrm{sc}}^{\infty} (M; E) 
\end{equation}
for $D$ as $S^{+} + S^{-}$ up to smoothing operators, where $S^{\pm} := \sum_{\alpha} Q_{\alpha}^{\pm} S Q_{\alpha}^{\pm}$, $(Q_{\alpha}^{\pm})^{*} Q_{\alpha}^{\pm}$ are microlocal partition of unity such that $D Q_{\alpha}^{\pm} \equiv Q_{\alpha}^{\pm} D$, and $\ES{Q_{\alpha}^{\pm}} \subset \varDelta \, \dot{T}_{0,\pm} M$. 
Consequently,  
\begin{subequations}
	\begin{eqnarray}
		&& 
		S \equiv S^{+} + S^{-}, 
		\label{eq: microlocal_decomposition_causal_propagator_Dirac_type_op}
		\\ 
		&& 
		S^{\pm} D \equiv D S^{\pm} \equiv 0,  
		\label{eq: Wightman_distribution_Dirac_op_mod_smooth} 
		\\ 
		&& 
		S^{\pm} \in I^{-1/2} \big( M \times M, C^{\pm \prime}; \Hom{E, E} \big), 
		\label{eq: Wightman_distribution_Dirac_FIO}
	\end{eqnarray}
\end{subequations}
by an application of~\eqref{eq: product_Lagrangian_dist_bundle}. 
The operator $Q_{\alpha}^{\pm}$ can be constructed by choosing any $[q_{\alpha}^{\pm}] \in S^{0 - [\infty]} \big( \dotCoTanM, \Hom{E, E} \big)$ as its left total symbol such that $\mathrm{esssupp} (q_{\alpha}^{\pm}) \subset \dot{T}_{0,\pm}^{*} M$; in other words $q_{\alpha}^{\pm}$ is of order $-\infty$ near $y \in M$ and in a conic neighbourhood of $\eta$ for all $\yeta$ in the complement of $\mathrm{esssupp} (q^{\pm})$. 
Furthermore, we pick $\supp{q_{\alpha}^{\pm}}$ slightly away from the projection of $\ES{Q_{\alpha}^{\pm}}$ on $M$ so that $[D, Q_{\alpha}^{\pm}]$ is smooth. 
We note that 

% 
% 
% 
% ================================================ 
\begin{lemma}
	As in the terminologies of Theorem~\ref{thm: Hadamard_bisolution_Dirac_type_op}, let $S$ be the Pauli-Jordan operator for $D$. 
	Then $\ri S$ is non-negative with respect to $(\cdot|\cdot)$.  
\end{lemma}
% ================================================ 
% 
% 
% 
% ================================================ 
\begin{proof} 
	Let $u, v \in \secE$ having compact $\supp{u} \cap \supp{v}$ and let $K$ be a compact set  having a smooth boundary $\partial K$ with an outward unit normal vector $N$ and volume element $\rd \mathsf{v}$. 
	Then the Green-Stokes formula yields 
	(see e.g.,~\cite[(1.7)]{Dimock_AMS_1982} 
	or~\cite[(9)]{Baer_Springer_2012})
	\begin{equation*}
		\int_{K} \big( (Du|v)_{x} - (u|Dv)_{x} \big) \dVolg (x) = - \ri \int_{\partial K} (\symb{D} (N^{\flat}) u|v)_{x'} \, \rd \mathsf{v} (x'),  
	\end{equation*} 
	where $\dVolg$ is the Lorentzian volume form on $(M, g)$. 
	Setting $u := S^{\ret} w$ and $v := Sw$ for any $w \in \comSecE$ the above formula entails $(w | \ri S w) \geq 0$ due to the compactness of $\supp{S^{\ret} u} \subset J^{+} (\supp{u})$. 
\end{proof}
% ================================================ 
% 
% 
% 

Therefore, $(u | {\ri} S^{\pm} u) \geq 0$ by~\eqref{eq: microlocal_decomposition_causal_propagator_Dirac_type_op}. 
Furthermore, 
\begin{equation}
	\WFPrime{S^{\adv, \, \ret, \, \Feyn, \, \aFeyn}} = \WFPrime{G^{\adv, \, \ret, \, \Feyn, \, \aFeyn}} \subset \varDelta \dotCoTanM \cup C^{\adv, \, \ret, \, \Feyn, \, \aFeyn}
\end{equation}
by~\eqref{eq: WF_advanced_retarded_Green_op_NHOp} and~\eqref{eq: def_Feynman_bicharacteristic_relation_NHOp}. 
We observe that the Lagrangian distribution $S^{\Feyn} - S^{\adv} \in I^{-1/2} \big( M \times M, C^{+ \prime}; \Hom{E, E} \big)$ is a bisolution of the homogeneous Dirac equation. 
Then, it follows from~\eqref{eq: Wightman_distribution_Dirac_op_mod_smooth} and~\eqref{eq: Wightman_distribution_Dirac_FIO} that $S^{+} \equiv S^{\Feyn} - S^{\adv}$. 
Therefore $\big( u \big| \ri (S^{\Feyn} - S^{\adv}) u \big)$ is non-negative. 
Hence, we arrive at the assertion which is the analogue of Proposition~\ref{thm: positivity_Feynman_minus_adv_NHOp} for Dirac-type operators. 

As before, in order to turn $\ri (S^{\Feyn} - S^{\adv})$ into an exact distributional bisolution $\omega$, we are going employ the well-posedness of the Cauchy problem for $D$ on a globally hyperbolic spacetime~\cite[Thm. 2]{Muehlhoff_JMP_2011}.  
For any $f \in C_{\mathrm{c}}^{\infty} (\Sigma; E |_{\Sigma})$, we define distributional Cauchy data $f \otimes \delta_{\Sigma}$ by 
\begin{equation} \label{eq: def_distributional_Cauchy_data_Dirac_type_op}
	f \otimes \delta_{\Sigma} (\phi) := \int_{\Sigma} \phi (x) \, \big( \symb{D} (x, N^{\flat}) \, f (x) \big) \, \dVolh (x) 
\end{equation}
for any $\phi \in C^{\infty} (M; E^{*})$, where we have used the notation $(t, x) \in \R \times \Sigma \cong M$. 
By $\mathcal{X}_{\Sigma} \subset \cE_{\mathrm{N}^{*} \Sigma}' (M; E)$, we denote the set of distributions of the above form where $\cE_{\mathrm{N}^{*} \Sigma}' (M; E)$ is as before. 
As in the previous section, any bidistribution $w \in \cD' (M \times M; E \boxtimes E^{*})$ such that $wD$ and $Dw$ are smooth in distributional sense, can be defined as a sequential continuous bilinear form on $\mathcal{X}_{\Sigma}$.  

% 
% 
% 
% ================================================ 
\begin{lemma}
	As in the terminologies of Theorem~\ref{thm: Hadamard_bisolution_Dirac_type_op}, let $w \in \cD' (M \times M; E \boxtimes \bar{E}^{*})$ be a bisolution of $D$. 
	Then $(u|wu) \geq 0$ if and only if $w (\bar{u}^{*} \otimes u) \geq 0$ for all $u \in \mathcal{X}_{\Sigma}$. 
\end{lemma}
% ================================================ 
% 
% 
% 
% ================================================ 
\begin{proof}
	The proof is completely analogous to the proof of Lemma~\ref{lem: positive_bisolution_NHOp}  with $\square$ replaced by $D$, and the Cauchy data~\eqref{eq: def_distributional_Cauchy_data_NHOp} and Pauli-Jordan operator $G$ (\eqref{eq: def_causal_propagator}), replaced by the Cauchy data~\eqref{eq: def_distributional_Cauchy_data_Dirac_type_op} and Pauli-Jordan operator $S$ (\eqref{eq: def_causal_propagator_Dirac}), respectively. 
\end{proof} 
% ================================================ 
% 
% 
% 

Following exactly the route taken from Proposition~\ref{thm: positivity_Feynman_minus_adv_NHOp} to Theorem~\ref{thm: existence_Feynman_propagator_NHOp}, the construction yields a Feynman propagator that satisfies the positivity condition.
% 
% 
% 
% ================================================ 
\HadamardBisolutionDiracTypeOp*
% ================================================ 
% 
% 
% 
% 
% 
% 
% 
% 
% 
% 
% ================================================ 
\section{Examples}
% ================================================ 
We list below several operators arising in the context of quantum field theories on a globally hyperbolic spacetime $(M, g)$. 
% ================================================ 
\subsection{Covariant Klein-Gordon operator}
% ================================================ 
Let $E \to M$ be a trivial line bundle, that is, its sections are just complex-valued smooth functions $C^{\infty} (M)$ on $M$. 
The covariant Klein-Gordon operator is defined by 
\begin{equation*}
	\square := - \div \circ \grad + \mathrm{m}^{2} + \lambda \mathsf{R} : C^{\infty} (M) \to C^{\infty} (M), 
\end{equation*}
where $\mathrm{m}^{2} \in \R_{+}$ is a parameter, physically interpreted as mass-squared of a linear Klein-Gordon field, $\lambda$ is a coupling and $\mathsf{R}$ is the Ricci scalar of $M$. 
The special case of $\square$ when $\mathrm{m}, \lambda = 0$ is called the d'Alembert/relativistic wave operator. 
Since the natural inner product on $C^{\infty} (M)$ is positive-definite with respect to which $\square$ is formally selfadjoint 
(see e.g.~\cite[Example 2.13]{Baer_Springer_2012}), 
Theorem~\ref{thm: existence_Feynman_propagator_NHOp} applies and positive Feynman propagators can be constructed using this method.   
% 
% 
% 
% 
% 
% 
% 
% 
% 
% 
% ================================================ 
\subsection{Connection d'Alembert operator} 
% ================================================ 
Let $E \to M$ be a smooth complex vector bundle and $\nabla^{E}$ a connection on $E$. 
We endow $M$ with the Levi-Civita connection $\nabla^{T^{*} M}$ on the cotangent bundle $T^{*}M$. 
These connections induce a connection $\nabla := \nabla^{T^{*} M} \otimes \one_{E} + \one_{T^{*} M} \otimes \nabla^{E}$ on the vector bundle $T^{*} M \otimes E \to M$. 
The connected d'Alembert operator is then defined by the composition of following three maps  
\begin{equation*}
	\square : C^{\infty} (M; E) 
	\stackrel{\nabla^{E}}{\to} C^{\infty} (M; T^{*} M \otimes E) 
	\xrightarrow{\nabla} C^{\infty} (M; T^{*} M \otimes T^{*} M \otimes E) 
	\xrightarrow{\tr \otimes I} C^{\infty} (M; E), 
\end{equation*} 
where $\tr_{g}$ is the metric trace (as defined after~\eqref{eq: Weitzenboeck_connection_NHO}) and $I$ is the identity operator on $E$. 
A straightforward computation shows that this operator is normally hyperbolic 
(see e.g.~\cite[Example 2.14]{Baer_Springer_2012}). 
Then, Theorem~\ref{thm: existence_Feynman_propagator_NHOp} entails that $\square$ admits a Feynman propagator whereas the conclusion for positivity depends on the positivity of the inner product with respect to which $\square$ is formally selfadjoint. 
% 
% 
% 
% 
% 
% 
% 
% 
% 
% 
% ================================================ 
\subsection{Hodge-d'Alembert operator} 
\label{sec: Hodge_d_Alembert_op}
% ================================================ 
Let $E := \wedge^{k} T^{*} M$ be the bundle of $k$-forms. 
The Hodge-d'Alembert operator is defined by  
\begin{equation*}
	\Box := - (\updelta \rd + \rd \updelta) : C^{\infty} (M; \wedge^{k} T^{*} M) \to C^{\infty} (M; \wedge^{k} T^{*} M),   
\end{equation*}
where $\rd$ is the exterior derivative and $\updelta$ is coexterior differential. 
By the Bochner formula, the Weitzenb\"ock connection is the Levi-Civita connection on $k$-forms and hence this connection is to be used for microlocalisation.
It can be shown that $\Box$ is a normally hyperbolic operator 
(see e.g.~\cite[Example 2.15]{Baer_Springer_2012}). 
There thus exists a Feynman propagator. 
Since the natural inner product with respect to which this operator is formally selfadjoint is not positive-definite unless $k=0$ or $k=n$, we cannot conclude non-negativity directly. For example for $k=1$ non-negativity is only expected on a subset as is usual for gauge theories. 
% 
% 
% 
% 
% 
% 
% 
% 
% 
% 
% ================================================ 
\subsection{Proca operator} 
\label{sec: Proca_op}
% ================================================ 
The Proca operator acting on covectors $A$ on $M$ is defined as   
\begin{equation*}
	P := \updelta \rd + \mathrm{m}^{2} : C^{\infty} (M; T^{*}M) \to C^{\infty} (M; T^{*}M), 
	\quad \mathrm{m} \in \R \setminus \{0\},  
\end{equation*}
which is not normally hyperbolic. 
Nevertheless, $P$ is equivalent to the normally hyperbolic operator $\updelta \rd + \rd \updelta + \mathrm{m}^{2}$ on $C^{\infty} (M; T^{*}M)$ together with the Lorenz constraint $\updelta A = 0$. 
Therefore, a Feynman propagator for $\updelta \rd + \rd \updelta + \mathrm{m}^{2}$ can be constructed by this method. The positivity issue is same as in Section~\ref{sec: Hodge_d_Alembert_op}. 
% 
% 
% 
% 
% 
% 
% 
% 
% 
% 
% ================================================ 
\subsection{Twisted Dirac operators} 
\label{sec: twisted_Dirac_op}
% ================================================ 
Let $S \to M$ be a spinor bundle over a globally hyperbolic spacetime $(M, g)$ 
admitting\footnote{A  
		spacetime $(M, g)$ admits a spin-structure if and only if the second Stiefel–Whitney class of $M$ vanishes. 
		We note that a spin structure always exists on a $4$-dimensional globally hyperbolic spacetime.
}  
a spin-structure. 
The Levi-Civita connection $\nabla^{\mathrm{LC}}$ on $(M, g)$ induces a spin connection $\nabla^{S}$ on $S$. 
Let $E \to M$ be a smooth vector bundle endowed with a connection $\nabla^{E}$.  
These two connections induce another connection $\nabla := \nabla^{S} \otimes \one_{E} + \one_{S} \otimes \nabla^{E}$ on the twisted spinor bundle $S \otimes E \to M$ and the twisted Dirac operator is defined by 
(see, e.g.~\cite{Baum_1981} for details) 
\begin{equation*} \label{eq: def_twisted_Dirac_op}
	\slashed{D} := - \ri \mathsf{c} \circ g^{-1} \circ \nabla : C^{\infty} (M; S \otimes E) \to C^{\infty} (M; S \otimes E), 
\end{equation*} 
where $\upgamma : TM \to \End{S}$ is the Clifford mapping, pointwise defined by  
\begin{equation*}
	\upgamma (X) \, \upgamma (Y) + \upgamma (Y) \, \upgamma (X) = 2 g (X, Y) \mathbbm{1}_{\End{S}} 
\end{equation*}
for any $X, Y \in T_{x} M$ and the Clifford multiplication $\mathsf{c} : TM \otimes S \otimes E \to S \otimes E$ is given by $X \otimes \psi \otimes u \mapsto (\upgamma (X) \psi) \otimes u$. 
The 
Schr\"{o}dinger~\cite{Schroedinger_GRG_2020}-Lichnerowicz~\cite{Lichnerowicz_1963} 
formula entails  
(see e.g.~\cite[Example 1]{Muehlhoff_JMP_2011}) 
\begin{equation*}
	\slashed{D}^{2} = \nabla^{*} \nabla + \frac{\mathsf{R}}{4} + \mathsf{F}, 
\end{equation*}
where $\mathsf{R}$ is the scalar curvature of $M$ and $\mathsf{F}$ is the Clifford multiplied curvature of $\nabla^{E}$. 
Clearly, $\slashed{D}^{2}$ is a normally hyperbolic operator and thus $\slashed{D}$ is of  Dirac-type. 
This formula also shows that the Weitzenb\"ock connection is the twisted spin-connection $\nabla$ on the twisted spinor bundle and therefore it induces the $\slashed{D}$-compatible connection on $\Hom{S \otimes E, S \otimes E} \to T^{*} M$, which is to be used for microlocalisation. 
Since $\slashed{D}$ is formally selfadjoint with respect to the natural inner product $(\cdot| \cdot)$ on $S \otimes E$, a Feynman propagator for $\slashed{D}$ exists by Theorem~\ref{thm: existence_Feynman_propagator_NHOp} yet the existence of Hadamard states cannot be concluded due to the fact that $(\cdot | \cdot)$ is not positive-definite unless $(M, g)$ is a Riemannian manifold. 
However, since $(M, g)$ is globally hyperbolic, there exist smooth spacelike Cauchy hypersurfaces $\Sigma_{t} := \{t\} \times \Sigma$ with future-directed unit vector field $N$ normal to $\Sigma_{t}$ so that $\upgamma (N) \cdot$ defines the Clifford multiplication to achieve the hermitian form $\langle \cdot | \cdot \rangle$ satisfying~~\eqref{eq: positive_definite_sesequilinear_form_Dirac_type_op}. 
Then, Theorem~\ref{thm: Hadamard_bisolution_Dirac_type_op} is applicable and we have a Feynman propagator with the desired positivity with respect to $(\cdot|\cdot)$.  
% 
% 
% 
% 
% 
% 
% 
% 
% 
% 
% ================================================ 
\subsection{Rarita-Schwinger operator}
\label{sec: Rarita_Schwinger_op}
% ================================================ 
Let $S \to M$ be a complex spinor bundle over a Lorentzian spin manifold $(M, g)$ and $\upgamma$ the Clifford multiplication given by 
$T^{*} M \otimes S \ni (\xi, u) \mapsto \upgamma (\xi) u \in S$. 
One has the representation theoretic splitting    
\begin{equation*}
	T^{*} M \otimes S = \iota (S) \oplus S^{3/2}, \quad S^{3/2} := \ker \upgamma. 
\end{equation*}
Here, the embedding $\iota$ of $S$ into $T^{*} M \otimes S$ is locally defined by 
$\iota (u) := - e_{i} \otimes \upgamma (e^{i}) u / n$ 
where $\{e_{i}\}_{1 \leq i \leq n}$ is an orthonormal basis of $TM$. 
Suppose that $\slashed{D} := \ri (\mathbbm{1} \otimes \upgamma) \circ \nabla$ is the twisted Dirac operator (cf. Section~\ref{sec: twisted_Dirac_op}) on $T^{*} M \otimes S$. 
Then, the Rarita-Schwinger operator is defined as 
(see e.g.~\cite[Def. 2.25]{Baer_Springer_2012})  
\begin{equation} \label{eq: def_Rarita_Schwinger_op}
	R := (I - \iota \circ \upgamma) \circ \slashed{D} : C^{\infty} (M; S^{3/2}) \to C^{\infty} (M; S^{3/2}). 
\end{equation}
The characteristic set $\Char{R}$ of $R$ coincides with the set of lightlike covectors in dimensions $n \geq 3$  
(see e.g.~\cite[Lem. 2.26]{Baer_Springer_2012}) 
and $R$ is a formally selfadjoint differential operator whose Cauchy problem is well-posed when $(M, g)$ is globally hyperbolic albeit $R^{2}$ is \textit{not} a normally hyperbolic operator
(see e.g~\cite[Rem. 2.27]{Baer_Springer_2012}). 
Moreover, any hermitian form satisfying~\eqref{eq: positive_definite_sesequilinear_form_Dirac_type_op} does not exist in $n \geq 3$~\cite[Example 3.16]{Baer_Springer_2012}. 

Originally, Rarita and Schwinger (in Minkowski spacetime) considered the twisted Dirac operator $\slashed{D}$ restricted to $S^{3/2}$ but not projected back to $S^{3/2}$~\cite[(1)]{Rarita_PR_1941} 
(see also, e.g.~\cite[Sec. 2]{Homma_CMP_2019} for Riemannian 
and~\cite[Rem. 2.28]{Baer_Springer_2012} 
for Lorentzian spin manifolds), 
that is,   
\begin{equation} \label{eq: constraint_Rarita_Schwinger}
	\slashed{D} |_{C^{\infty} (M; S^{3/2})} : C^{\infty} (M; S^{3/2}) \to C^{\infty} (M; T^{*} M \otimes S)  
\end{equation}
in order to ensure the correct number of propagating degrees of freedom for spin-$3/2$ fields     
(see, for instance, the reviews~\cite{Sorokin_AIP_2005, Rahman_2015} for physical motivation and  different approaches used in Physics literature). 
The corresponding Rarita-Schwinger operator is then an overdetermined system and this constrained system limits possible curvatures of the spacetime~\cite{Gibbons_JPA_1976}   
(see e.g.~\cite[p. 856]{Homma_CMP_2019},~\cite{Hack_PLB_2013}): 
$(\mathsf{Ric} - \mathsf{R} g / n)^{*} u = 0$, 
where $\mathsf{Ric}$ and $\mathsf{R}$ are respectively the Ricci tensor and Ricci scalar curvature of $M$. 
In other words, the corresponding Rarita-Schwinger field exists only on Einstein spin manifolds. 
However, the Rarita-Schwinger operator corresponding to the restricted twisted spin-Dirac operator~\eqref{eq: constraint_Rarita_Schwinger} does \textit{not} admit a Green's operator~\cite[Rem. 2.28]{Baer_Springer_2012}.  

Since the Cauchy problem for $R$ is well-posed on a globally hyperbolic spacetime, the Rarita-Schwinger operator~\eqref{eq: def_Rarita_Schwinger_op} admits unique advanced and retarded propagators. 
Albeit $\Char{R} = \coLightBun = \Char{\slashed{D}}$, the existence of a (anti-)Feynman propagator cannot be concluded from Theorem~\ref{thm: Hadamard_bisolution_Dirac_type_op} as $R$ is not a Dirac-type operator. 
% 
% 
% 
% 
% 
% 
% 
% 
% 
% 
% ================================================ 
\subsection{Higher spin operators}
% ================================================ 
The straightforward attempts to generalise  
Dirac operator on Minkowski spacetime for arbitrary spin~\cite{Dirac_PRSA_1936} 
in 
curved spacetimes\footnote{Not 
	necessarily be globally hyperbolic.}
leads to difficulties
(see, e.g,~\cite[p. 324]{Illge_AnnPhys_1999} for a panoramic view 
and the reviews~\cite{Sorokin_AIP_2005, Rahman_2015}). 
A crucial advancement came through 
Buchdahl operator (in Riemannian manifold)~\cite{Buchdahl_JPA_1982} 
(see~\cite[Exam. 2.24]{Baer_Springer_2012} for a Lorentzian formulation) 
whose square turns out to be a normally hyperbolic operator~\cite[p. 8]{Baer_Springer_2012}, 
yet the minimum coupling principle seems violated and a "by hand" proposal is required in the original idea of Buchdahl. 
These minor imperfections were cured by 
W\"{u}nsch~\cite{Wuensch_GRG_1985},  
by 
Illege~\cite{Illge_ZAA_1992, Illge_CMP_1993} 
and by 
Illege and Schimming~\cite{Illge_AnnPhys_1999} 
for the massive case deploying the $2$-spinor formalism in $4$-dimensional curved spacetimes. 
In particular, the square of Buchdahl operator (as modified by W\"{u}nsch and Illge) is a normally hyperbolic operator on certain twisted bundles. 
In contrast, there are a few open questions for the massless case~\cite{Frauendiener_JGP_1999} 
(see also the reviews~\cite{Sorokin_AIP_2005, Rahman_2015} for the contemporary status and other formulations used in Physics literature). 
Hence, the existence of a Feynman propagator is evident either by Theorem~\ref{thm: existence_Feynman_propagator_NHOp} or Theorem~\ref{thm: Hadamard_bisolution_Dirac_type_op} but the issue of positivity is non-conclusive at this stage. 
% 
% 
% 
% 
% 
% 
% 
% 
% 
%  
% ================================================ 
\section*{Acknowledgement}  
We would like to acknowledge the anonymous referees for their suggestions. 
We are indebted to Chris Fewster for his careful reading and pointing out a few typos.
OI was funded by the Leeds International Doctoral Studentship during the entire period of this work at the University of Leeds, UK.
% ================================================ 
% 
% 
% 
% 
% 
% 
% 
% 
% 
% 
% ================================================ 
\appendix 
% ================================================ 
% ================================================ 
\section{Symbols}  
\label{sec: symbol}
% ================================================ 
The multiindex notation will be used in this section, that is, for a multiindex $\alpha \in \N_{0}^{d}, \N_{0} := \N \cup \{0\}, d \in \N$ we set 
$|\alpha| := |\alpha_{1}| + \ldots + |\alpha_{d}|$ 
and we define partial derivatives $D_{x}^{\alpha} := (-\ri)^{|\alpha|} \partial^{|\alpha|} / \partial (x^{1})^{\alpha_{1}} \ldots \partial (x^{n})^{\alpha_{n}}$. 

% 
% 
% 
% ================================================ 
\begin{definition} \label{def: symbol_mf}
	Let $U$ be an open subset of an $n$-dimensional Euclidean space $\Rn$. 
	The set $S_{1, 0}^{m} (U \times \Rd)$ of symbols of order $m \in \R$ and type $(1,0)$, is defined as the set of all complex-valued smooth functions $a$ on $U \times \Rd$ such that, for every compact set $K \subset U$ and all multiindices $\alpha \in \N_{0}^{d}, \beta \in \N_{0}^{n}$, the estimate 
	(see e.g.~\cite[Def. 18.1.1]{Hoermander_Springer_2007})   
	\begin{equation*}
		|D_{x}^{\beta} \partial_{\theta}^{\alpha} a \, \xtheta| \leq c_{\alpha, \beta; K} (1 + |\theta|)^{m - |\alpha|} 
	\end{equation*}
	is valid for some constant $c_{\alpha, \beta; K}$ for any $x \in K$ and any $\theta \in \Rd $.   
\end{definition}
% ================================================ 
% 
% 
% 

The above symbol class, known as the Kohn-Nirenberg symbol class, is too general for this exposition. 
In particular, it is sufficient to restrict our attention to a subclass of $S_{1, 0}^{m}$, known as the classical symbol class. 
In order to introduce this subclass we recall that a function $a$ on $U \times \Rd$ of degree $k$ is called positively homogeneous if $a (x, \lambda \theta) = \lambda^{k} a (x, \theta)$ for all $|\theta| > 1$ and $\lambda > 1$. 

% 
% 
% 
% ================================================ 
\begin{definition} \label{def: polyhomogeneous_symbol_mf}
	Let $U$ be an open subset of an $n$-dimensional Euclidean space $\Rn$ and let $m$ be any real number. 
	The set $S^{m} (U \times \Rn)$ of polyhomogeneous symbols is defined as the set of all $a \in S_{1, 0}^{m} (U \times \Rd)$ such that 
	(see e.g.~\cite[Def. 18.1.5]{Hoermander_Springer_2007}) 
	\begin{equation*} \label{eq: polyhomogeneous_symbol}
		a \xtheta \sim \sum_{k \in \N_{0}} a_{m-k} \xtheta \, \chi_{k} (\theta),    
	\end{equation*}
	where $a_{m-k}$ is positively homogeneous of degree $m - k$ when $|\theta| > 1$ and $\chi_{k}$ is a smooth function on $\Rd$ that vanishes identically near $\theta = 0$ such that $\chi_{k} (\theta) = 1$ if $|\theta| \geq 1$.

	Here $\sim$ symbolises the asymptotic summation, by which it is meant that, for any $a \in S^{\mu_{0}} (U \times \Rd)$ such that $\supp{a} \subset \bigcup \supp{a_{k}}$, one has 
	(see e.g.~\cite[Prop. 18.1.3]{Hoermander_Springer_2007}) 
	\begin{equation*}
		a - \sum_{k=0}^{N-1} a_{k} \in S_{1, 0}^{\mu_{\ms N}} (U \times \Rd) 
	\end{equation*} 
	for every $N \in \N$, where $(a_{k})_{k \in \N_{0}} \in S_{1, 0}^{m_{k}} (U \times \Rd)$ with $m_{k} \to - \infty$ as $k \to \infty$ and $\mu_{\ms N} := \max_{k \geq N} m_{k}$.

	We set $S^{\infty} := \bigcup_{m \in \R} S^{m}, S^{-\infty} := \bigcap_{m \in \R} S^{m}$, and $S^{m - [m']} := S^{m} / S^{m-m'}$.
\end{definition}
% ================================================ 
% 
% 
% 

Note that for any element $a$ in $S^{m} (U \times \Rd)$, an equivalence class $[a]$ in the quotient space $S^{m - [1]} (U \times \Rd)$ can be identified with the leading order term $a_{m}$ in Definition~\ref{def: polyhomogeneous_symbol_mf}. 
Thus, the space $S^{m -[1]} (U \times \Rd)$ is naturally identified with the space of smooth homogeneous functions on $U \times \Rd$ of degree $m$. 

Recall, a smooth manifold $\mathscr{M}$ endowed with a smooth, proper, and free action of the multiplication group $\R_{+}$ is called a conic manifold. 
Canonical relations are example of conic manifolds.  
It is evident that $S^{m - [1]} (\mathscr{M})$ can be defined analogously 
(see~\cite[p. 87]{Hoermander_ActaMath_1971} 
and, e.g.~\cite[Def. 21.1.8]{Hoermander_Springer_2007} for details). 

The generalisation to (complex) matrix-valued symbols $S^{m} \big( \cdot, \matN \big), N \in \N$ is straightforward. 
We refer, for instance, the monograph~\cite[Sec. 1.5.1-1.5.3, 4.4.1-4.4.3]{Scott_OUP_2010} for details.
% 
% 
% 
% 
% 
% 
% 
% 
% 
% 
% ================================================ 
\section{Fourier integral operators associated with symplectomorphisms}  
\label{sec: FIO_symplecto_bundle}
% ================================================ 
Let $T^{*} M \to M, T^{*} N \to N$ be the cotangent bundles over manifolds $M, N$ and $\pr_{\ms M}, \pr_{\ms N} : M \times N \to M, N$ are the projection maps.
The simplest canonical relation $\varGamma' \subset \dotCoTanM \times \dotCoTanN$ with respect to the symplectic form $\pr_{\ms M}^{*} \omega_{\ms M} + \pr_{\ms N}^{*} \omega_{\ms N}$ is defined by the twisted graph $\varGamma'$ of a homogeneous symplectomorphism $\varkappa : \dotCoTanN \to \dotCoTanM$, where $\omega_{\ms M}, \omega_{\ms N}$ are symplectic forms on $T^{*} M, T^{*} N$. 
This enforces $\dim M = \dim N =: n$.  
In this case, it is always possible to choose coordinates $y = (y^{i})$ on the image of $y_{0} \in N$ under some local chart such that~\cite[Prop. 25.3.3]{Hoermander_Springer_2009} 
\begin{equation}
	\varGamma_{\varphi} = \{ (x, \rd_{x} \varphi; \grad_{\eta} \varphi, \eta) \}, 
	\quad 
	\det \frac{\partial^{2} \varphi}{\partial x^{i} \partial \eta_{j}} (x, \eta) \neq 0,  
\end{equation} 
by making use of a non-degenerate phase function $\varphi (x, \eta)$ on an open conic neighbourhood of $(x_{0}, \eta^{0}) \in M \times \dot{T}^{*}_{y_{0}} N$. 

As a consequence, the simplest class of Fourier integral operators $C_{\mathrm{c}}^{\infty} (N; F) \to \cD' (M; E)$ are those whose Schwartz kernels are elements in $I^{m} \big( M \times N, \varGamma'; \Hom{F, E} \big)$ where $E \to M, F \to N$ are smooth complex vector bundles and $m \in \R$.  
As explained in Section~\ref{sec: FIO}, any element $A$ of this space of Lagrangian distributions can be locally represented as a matrix $\big( I^{m} (\Rn \times \Rn, \varGamma'_{\varphi}) \big)_{k}^{r}$ of entries~\cite[pp. 169-173]{Hoermander_ActaMath_1971}  
(for details, see e.g.~\cite[Sec. 25.3]{Hoermander_Springer_2009},~\cite[pp. 461-465]{Treves_Plenum_1980}) 
\begin{equation} \label{eq: Hoermander_25_3_2_prime}
	A_{k}^{r} (x, y) \equiv \int_{\Rn} \re^{\ri (\varphi (x, \eta) - y \cdot \eta)} a_{k}^{r} (x, y, \eta) \, \frac{\rd \eta}{(2 \pi)^{n}},      
\end{equation}  
where $k = 1, \ldots, \rk F; r = 1, \ldots, \rk E$ and the total symbol $a_{k}^{r}$ is of homogeneous of degree $m$, having support in the interior of a small conic neighbourhood of $(x_{0}, \eta^{0})$ contained in the domain of definition of $\varphi$. 

In order to describe the principal symbol of $I^{m} \big( M \times N, \varGamma'; \Hom{F, E} \big)$, one notes that $\varGamma$ is naturally a symplectic manifold with respect to the symplectic form $\omega_{\ms \varGamma} := \Pr_{\ms M}^{*} \omega_{\ms M} = \Pr_{\ms N}^{*} \omega_{\ms N}$, where $\Pr_{\ms M}, \Pr_{\ms N}$ are the projections from $\varGamma$ to $T^{*} M, T^{*} N$. 
The natural symplectic half-density $\varOmega^{\nicefrac{1}{2}} \varGamma$ (notationally suppressed in Section~\ref{sec: principal_symbol_FIO}) in the principal symbol $\symb{A} \xxiyeta \in S^{m + n/2} \big( \varGamma; \Maslov \otimes \widetilde{\mathrm{Hom}} (F, E) \otimes \varOmega^{\nicefrac{1}{2}} \varGamma \big)$ of $A$ can be factored out so that the order $m + (n + n) / 4$ of the halfdensity valued principal symbol is reduced to $m$ as $\varOmega^{\nicefrac{1}{2}} \varGamma$ is of order $n/2$. 
Therefore, the principal symbol map  
\begin{subequations}
	\begin{equation} \label{eq: symbol_FIO_symplecto_bundle}
		\sigma_{\cdot} : I^{m - [1]} \big( M \times N, \varGamma'; \Hom{F, E} \big) \to S^{m - [1]} \big( \varGamma; \Maslov \otimes \widetilde{\mathrm{Hom}} (F, E) \big)    
	\end{equation}
	is locally given by  
	(see e.g.~\cite[p. 27]{Hoermander_Springer_2009}) 
	\begin{equation} \label{eq: symbol_FIO_symplecto}
		(\sigma_{A})_{k}^{r} (x, \eta) := a_{k}^{r} \left( x^{i}, \parDeri{\eta_{i}}{\varphi}; \eta_{i} \right) \, \left| \det \left[ \frac{\partial^{2} \varphi}{\partial x^{i} \partial \eta_{j}} \right] \right|^{-1/2} \mathbbm{m} \mod S^{m-1} (\cdot),    
	\end{equation} 
\end{subequations}
where $\mathbbm{m}$ is the contribution of the Keller-Maslov bundle $\Maslov \to \varGamma \circ \varLambda$ as explained below. 

The Keller-Maslov bundle $\Maslov \to C$ over a conic Lagrangian submanifold $C \subset \dotCoTanM$ is a complex line bundle obtained from some principal bundle with structure group $\nicefrac{\Z}{4 \Z}$~\cite[p. 148]{Hoermander_ActaMath_1971} 
(see also, e.g.~\cite[Def. 21.6.5]{Hoermander_Springer_2007}). 
It is trivial as a vector bundle.  
For our purpose, it is constructed as follows. 
Let $\varphi$ be a non-degenerate phase function for $C$, whose fibre-critical manifold is $\varSigma$. 
Employing $\Hess \varphi$, one has the integer-valued map: 
$\varSigma \ni \xtheta \mapsto \sgn \big( \Hess_{\theta} \varphi \, \xtheta \big) \in \Z$. 
This function can be somewhat discontinuous as the Hessian can be singular at some elements of $\varSigma$. 
We now consider an open conic 
(Leray\footnote{This means that finite intersection of $\mathcal{L}_{\alpha}$'s are either empty or diffeomorphic to the open ball.}-) 
covering $\{ \mathcal{L}_{\alpha} \}$ of $C$, indexed by a countable set with the corresponding non-degenerate phase functions $\varphi_{\alpha}$ and the fibre-critical manifolds $\varSigma_{\alpha}$. 
It follows from~\eqref{eq: necessarily_sufficient_equivalent_clean_phase_function} that the  mapping~\cite[$(3.2.15)$]{Hoermander_ActaMath_1971}
(for details, see e.g.~\cite{Meinrenken_ReptMathPhys_1992},~\cite[Sec. 5. 13]{Guillemin_InternationalP_2013},~\cite[pp. 408-412]{Treves_Plenum_1980})
\begin{align*}
	\mathfrak{s}_{\alpha \beta} &: \varSigma_{\alpha} \cap \varSigma_{\beta} \to \Z, ~ (x; \theta^{\alpha}, \theta^{\beta}) \mapsto 
	\\ 
	& 
	\mathfrak{s}_{\alpha \beta}(x; \theta^{\alpha}, \theta^{\beta}) 
	:= \frac{1}{2} 
	\Big( 
	\sgn \big( \Hess_{\theta} \varphi_{\beta} \, (x; \theta^{\beta}) \big) - d_{\beta} 
	- 
	\sgn \big( \Hess_{\theta} \varphi_{\alpha} \, (x; \theta^{\alpha}) \big) + d_{\alpha} 
	\Big) 
\end{align*}
is constant for all connected intersections $\varSigma_{\alpha} \cap \varSigma_{\beta}$. 
Let $\jmath_{\alpha \beta} : \varSigma_{\alpha} \cap \varSigma_{\beta} \to \mathcal{L}_{\alpha} \cap \mathcal{L}_{\beta}$ be the homogeneous immersion. 
Thereby, 
$\tau_{\alpha \beta} := \re^{\nicefrac{\ri}{2} \pi \mathfrak{s}_{\alpha \beta} \circ \jmath_{\alpha \beta}^{-1}} : \mathcal{L}_{\alpha} \cap \mathcal{L}_{\beta} \to \dot{\C}$, 
clearly satisfies the cocycle property together with $|\tau_{\alpha \beta}| = 1$. 
Hence, the collection of non-zero complex numbers $\{ \tau_{\alpha \beta} \}$ for all $\alpha$ and $\beta$ such that $\mathcal{L}_{\alpha}$ and $\mathcal{L}_{\beta}$ are not disjoint, defines our transition function and the global construction of $\Maslov$ is achieved by taking the disjoint union 
$\Maslov := \bigsqcup_{\alpha} (\mathcal{L}_{\alpha} \times \C)$  
modulo the equivalence relation
\begin{equation*} \label{eq: def_equivalence_relation_Maslov_bundle}
	\big( (x, \xi^{\alpha}), c_{\alpha} \big) \sim \big( (x, \xi^{\beta}), c_{\beta} \big) 
	\Leftrightarrow 
	(x, \xi^{\alpha}) = (x, \xi^{\beta}) \in \mathcal{L}_{\alpha} \cap \mathcal{L}_{\beta}, c_{\alpha} = \tau_{\alpha \beta} c_{\beta}. 
\end{equation*}

We remark that the composition $AB$ of properly supported Lagrangian distributions 
$A \in I^{m} \big( M \times \tilde{N}, \varGamma', \Hom{\tilde{F}, E} \big)$ 
and 
$B \in I^{m'} \big( \tilde{N} \times N, \varLambda'; \Hom{F, \tilde{F}} \big)$ 
is always well-defined and  
$AB \in I^{m + m' + e/2} \big( M \times N, (\varGamma \circ \varLambda)'; \Hom{F, E} \big)$ 
defines a Fourier integral operator associated to the graph (canonical relation) $\varGamma \circ \varLambda$ of the composition of symplectomorphisms $\dotCoTanN \ni \yeta \mapsto (\tilde{y}, \tilde{\eta}) \mapsto \xxi \in \dotCoTanM$. 
In this case, the principal symbol is given by~\cite[p. 180]{Hoermander_ActaMath_1971}  
(see also~\cite[(6.11), p. 465]{Treves_Plenum_1980}) 
\begin{equation} \label{eq: product_symbol_FIO_symplecto}
	\symb{AB} 
	= 
	\sum_{(z, \zeta) \,|\, (x, \xi; z, \zeta) \in \varGamma, (z, \zeta; y, \eta) \in \varLambda} \symb{A} (x, \xi; z, \zeta) \, \big( \symb{B} (z, \zeta; y, \eta) \big). 
\end{equation}
If the respective vector bundles are hermitian then the algebra of Lagrangian distributions is a $*$-algebra. 

Note that a pseudodifferential operator $P : C_{\mathrm{c}}^{\infty} (M; F) \to \secE$ of order $m$, is a Fourier integral operator associated with the twisted graph $\varGamma' = \{ (x, \xi; x, - \xi) \in \dotCoTanM \times \dotCoTanM \}$ of the identity homogeneous symplectomorphism on $\dotCoTanM$. 
In other words, $P \in I^{m} \big(M \times M, (\varDelta \, \dotCoTanM)'; \Hom{F, E} \big)$ where $\varDelta \, \dotCoTanM : \dotCoTanM \hookrightarrow \dotCoTanM \times \dotCoTanM$ is the diagonal embedding.
% ================================================ 
% 
% 
% 
% 
% 
% 
% 
% 
% 
% 
% =======================================  
\section{Parametrices for elliptic Fourier integral operators}
% ======================================= 
In this section we will introduce the notion of ellipticity~\cite[p. 186]{Duistermaat_ActaMath_1972} 
for a Fourier integral operator and show that an approximate inverse always exists for such an operator. 

% 
% 
% 
% =======================================
\begin{definition} \label{def: elliptic_FIO}
	Let $E \to M, F \to N$ be complex smooth vector bundles over manifolds $M, N$. 
	Suppose that $\varGamma$ is the graph of a homogeneous symplectomorphism from $\dotCoTanN$ to $\dotCoTanM$ and that $\Maslov \to \varGamma$ is the Keller-Maslov bundle over $\varGamma$. 
	A Lagrangian distribution $A \in I^{m} \big( M \times N, \varGamma'; \Hom{F, E} \big)$ is called non-characteristic at $(x_{0}, \xi^{0}; y_{0}, \eta^{0}) \in \varGamma$ if its principal symbol $a \in  S^{m} \big( \varGamma; \Maslov \otimes \widetilde{\mathrm{Hom}} (F, E) \big)$ has an inverse $\in S^{-m} \big( \varGamma; \Maslov^{-1} \otimes \widetilde{\mathrm{Hom}} (E, F) \big)$ in a conic neighbourhood of $(x_{0}, \xi^{0}; y_{0}, \eta^{0})$. 
	$A$ is called elliptic if it is non-characteristic at every point of $\varGamma$. 
	The complement of non-characteristic points is called \sout{by} the characterisitic set $\Char{A}$ of $A$~\cite[Def. 25.3.4]{Hoermander_Springer_2009}.   
\end{definition}
% =======================================
% 
% 
% 

We note that~\eqref{eq: Hoermander_25_3_2_prime} and~\eqref{eq: symbol_FIO_symplecto} imply that the non-characteristic points belongs to $\WFPrime{A}$. 
If $\varGamma^{-1}$ is also a graph and $A$ is elliptic, properly supported, then $A$ has a unique parametrix $G$, i.e., 
\begin{equation} \label{eq: parametrix_elliptic_FIO}
	GA - I_{F} \in \PsiDO{- \infty}{N; F}, 
	\qquad 
	AG - I_{E} \in \PsiDO{- \infty}{M; E}. 
\end{equation}
A proof of the microlocal version of this claim is going to be presented shortly after one devises a variant of Lagrangian distributions where the closedness assumption on the Lagrangian submanifold has been relaxed.  

Let $C \subset \dotCoTanM \times \dotCoTanN$ be a conic Lagrangian submanifold which is not necessarily closed and let $K \subset C$ a conic subset which is closed in $\dotCoTanMN$. 
By $I^{m} \big( M, K; \Hom{F, E} \big)$, one denotes the set of all matrices with elements~\eqref{eq: 25_1_3_clean_Hoemander} together with the additional condition that the restriction of $a^{r}_{k}$ (appearing in~\eqref{eq: 25_1_3_clean_Hoemander}) to some conic neighbourhood in $\R^{\nM} \times \R^{\nN} \times \R^{d}$ of the pullback of $C \setminus K$ by the Lagrangian fibration (terms as defined in~\eqref{eq: local_twisted_canonical_relation_phase_function}) $\varSigma_{\varphi} \ni \xytheta \mapsto (x, \rd_{x} \varphi; y, \rd_{y} \varphi) \in C_{\varphi}$, is in the class $S^{-\infty} (\cdot)$. 
Then, the analogue of the isomorphism~\eqref{eq: def_symbol_map} reads   
\begin{equation} \label{eq: def_symbol_map_FIO_microlocal}
	I^{m - [1]} \big( M \times N, K'; \Hom{F, E} \big) \cong S^{m - [1]} \big( K; \Maslov \otimes \widetilde{\mathrm{Hom}} (F, E) \big), 
\end{equation}  
where $S^{m} \big( K; \Maslov \otimes \widetilde{\mathrm{Hom}} (F, E) \big)$ denotes the set of $a \in S^{m} \big( C; \Maslov \otimes \widetilde{\mathrm{Hom}} (F, E) \big)$ such that $a \in S^{-\infty}$ on $C \setminus K$~\cite[p. 187]{Duistermaat_ActaMath_1972}. 

% 
% 
% 
% =======================================
\begin{theorem} \label{thm: existence_parametrix_FIO}
	Let $E \to M$ and $F \to N$ are smooth complex vector bundles over manifolds $M, N$. 
	Suppose that $\varGamma$ is the graph of a homogeneous symplectomorphism $\varkappa$ from an open conic subset $\cV \subset \dotCoTanN$ into $\dotCoTanM$ and that $K \subset \varGamma$ is a conic subset which is closed in $\dotCoTanMN$.  
	Then, for any conic subset $\cU \subset \cV$ such that $\cU$ (resp. $\varkappa (\cU)$) is closed in $\dotCoTanN$ (resp. $\dotCoTanM$), if $A \in I^{m} \big( M \times N, K'; \Hom{F, E} \big), m \in \R$ is non-characteristic on $\{ \big( \varkappa (\cU), \cU \big) \} \subset \varGamma$ then there exists an elliptic $G \in I^{-m} \big( N \times M, K^{-1 \prime}; \Hom{E, F} \big)$ such that
	\begin{equation*} 
		\ES{GA - I_{F}} \cap \cU = \emptyset, 
		\quad 
		\ES{AG - I_{E}} \cap \varkappa (\cU) = \emptyset, 
	\end{equation*}
	where $I_{E} \in \PsiDO{0}{M; E}$ and $I_{F} \in \PsiDO{0}{N; F}$ are identity operators.  
	The parametrix $G$ is unique in the sense that $\big( \cU, \varkappa (\cU) \big) \not\subset \WFPrime{G - \tilde{G}}$ for any other parametrix $\tilde{G}$ of $A$. 
\end{theorem}
% =======================================
%
%
% 

% =======================================
\begin{proof}
	We will follow the same arguments used to proof the scalar-version of this assertion~\cite[Prop. 5.1.2]{Duistermaat_ActaMath_1972}.  
	Let $a$ be the principal symbol of $A$ in the sense of~\eqref{eq: def_symbol_map_FIO_microlocal}.  
	By hypotheses, there exists a $b \in S^{-m} \big( K; \Maslov^{-1} \otimes \widetilde{\mathrm{Hom}} (E, F) \big)$ such that $ba = \symb{I_{F}}$ in a conic neighbourhood of $\varDelta \cU$. 
	By utilising appropriate microlocal partition of unities $\varPsi$ (resp. $\varPhi$) subordinated to $\cU$ (resp. $\varkappa (\cU)$), one chooses $b = 0$ outside of a sufficiently small conic neighbourhood of $\{ \big( \varkappa (\cU), \cU \big) \}$, where $\cU$ is identified with $\varDelta \cU$ (via projection). 
	Then, we have a properly supported $B_{0} \in I^{-m} \big( N \times M, K^{-1 \prime}; \Hom{E, F} \big)$ whose principal symbol is $\symb{\varPsi} b \symb{\varPhi} = b$, and by the composition of Lagrangian distributions, there exists a properly supported $R \in \PsiDO{-1}{N; F}$ such that $B_{0} A = I_{F} - R$.  
	Next, we want to invert $I_{F} - R$ by making use of the Neumann series: 
	$(I_{F} - R)^{-1} = \sum_{k \in \N_{0}} R^{k}$.  
	Set   
	$B_{k} := R^{k} B_{0} \in I^{-m - k} \big( N \times M, K^{-1 \prime}; \Hom{E, F} \big)$ 
	and then 
	\begin{equation*}
		I_{F} 
		= 
		(I_{F} - R) \sum_{k=0}^{\infty} R^{k}  
		= 
		\sum_{k=0}^{N-1} R^{k} (I_{F} - R) + \sum_{k=N}^{\infty} R^{k} (I_{F} - R) 
		= 
		\sum_{k=0}^{N-1} B_{k} A + R^{N}. 
	\end{equation*}
	Let $G$ be defined by the asymptotic summation: $G :\sim \sum_{k \in \N_{0}} B_{k}$.  Then inserting the last equation one obtains   
	\begin{equation*}
		GA - I_{F} 
		=
		\left(G - \sum_{k=0}^{N-1} B_{k} \right) A - R^{N} 
		\sim  
		\sum_{k = N}^{\infty} B_{k} A - R^{N} 
		\in \PsiDO{- N}{N; F}  
	\end{equation*}
	for every $N \in \N$, which in turn proves the first part of the theorem as a right parametrix can be constructed analogously. 

	To prove the uniqueness, we suppose that $\tilde{G}$ is another right parametrix for $A$. 
	Then  
	\begin{equation*}
		\big\{ \big( \cU, \varkappa (\cU) \big) \big\} \not\subset \WFPrime{G} = \WFPrime{GA \tilde{G}} = \WFPrime{\tilde{G}} 
	\end{equation*}
	and similarly for the left parametrix. 
\end{proof}
% ======================================= 
% 
% 
% 
% 
% 
% 
% 
% 
% 
% 
% ================================================ 
\section{Egorov theorem}  
% ================================================ 
In this section, we will present a vector bundle version of the 
Egorov theorem~\cite{Egorov_UMN_1969}, 
slightly general than in the exiting literature~\cite{Dencker_JFA_1982},~\cite[Thm. 3.2]{Bolte_CMP_2004},~\cite[Thm. 1.7]{Kordyukov_MPAG_2005},~\cite[Thm. 6.1]{Kordyukov_JGP_2007},~\cite[Prop. 3.3]{Jakobson_CMP_2007},~\cite[Prop. A.3]{Gerard_CMP_2015}.  

% 
% 
% 
% =============================================
\begin{theorem}[Egorov's theorem] \label{thm: Egorov}
	Let $E \to M$ (resp. $F \to N$) be a smooth complex vector bundle and $\pi_{\ms M}: \dotCoTanM \to M$ (resp. $\pi_{\ms N} : \dotCoTanN \to N$) the punctured cotangent bundle over a manifold $M$ (resp. $N$). 
	Assume that $\varkappa : \dotCoTanN \to \dotCoTanM$ is a homogeneous symplectomorphism whose graph is denoted by $\varGamma$. 
	Suppose that $A \in I^{m} \big( M \times N, \varGamma'; \Hom{F, E} \big), B \in I^{-m} \big( N \times M, \varGamma^{-1 \prime}; \Hom{E, F} \big), m \in \R$ are properly supported Lagrangian distributions and that $P \in \PsiDO{m'}{M; E}, m' \in \R$ is a properly supported pseudodifferential operator having a scalar principal symbol $\symb{P}$. 
	Then, $BPA$ is a properly supported pseudodifferential operator $\PsiDO{m'}{N; F}$ whose principal symbol is 
	\begin{equation*}
		\symb{BPA} = \symb{BA} \cdot \varkappa^{*} \symb{P}, 
	\end{equation*}
	Here $\varkappa^{*} \symb{P}$ is the pull-back of $\symb{P}$ to $\dotCoTanN$ understood as scalar section of $\pi_{\ms N}^* \End{F}$.
\end{theorem}
% ============================================= 
% 
% 
% 
% ============================================= 
\begin{proof}
	We will use the same strategy employed for proving the scalar version of the principal symbol formula~\cite[Thm. 25.3.5]{Hoermander_Springer_2009}. 
	By repeated applications of the composition~\eqref{eq: product_symbol_FIO_symplecto} of Lagrangian distributions,  we have 
	$PA \in I^{m+m'} \big( M \times N, \varGamma'; \Hom{F, E} \big)$ and  
	$BPA \in \PsiDO{m'}{N; F}$.  
	In order to compute $\symb{PA}$, one lifts $\symb{P}$ to $\varGamma$ via the pullbacks of the projector $\dotCoTanM \times \dotCoTanN \to \dotCoTanM$ followed by the inclusion $\varGamma \hookrightarrow \dotCoTanM \times \dotCoTanN$. 
	Then $\symb{PA} = \symb{P} \symb{A}$. 
	Equivalently, one can lift $\varkappa^{*} \symb{P}$ to $\dotCoTanM \times \dotCoTanN$ via the projector $\dotCoTanM \times \dotCoTanN \to \dotCoTanN$ and then consider it as a homomorphism on $\varGamma$ as before. 
	Thus $PA - AR \in I^{m + m' - 1} \big( M \times N, \varGamma'; \Hom{F, E} \big)$ if $R \in \PsiDO{m'}{N; F}$ having the principal symbol $\varkappa^{*} \symb{P}$. 
	As $\symb{P}$ is scalar, therefore $BPA - BAR \in \PsiDO{m' - 1}{N; F}$ which entails the claim. 
\end{proof}
% =======================================
% 
% 
% 
% 
% 
% 
% 
% 
% 
% 
% =======================================  
\section{Products of operators with vanishing principal symbol}
% ======================================= 
In this section we will compute the principal symbol of the product of a pseudodifferential operator with a Lagrangian distribution when the principal symbol of the pseudodiferential operator vanishes. 
It turns out that Lie derivative plays a pivotal role in this regard so we recall a few rudimentary formulae. 
Let $u$ be a halfdensity valued section of a vector bundle $E \to M$ over a manifold $M$ and $\nabla$ a connection on $E$. 
Recall, the \textit{Lie derivative} of $u$ along a vector field $X$ on $M$ is defined by 
$\pounds_{X} u := \nicefrac{\rd}{\rd s} \big|_{s=0} (\varXi_{s}^{*} u)$  
where $\varXi_{s}^{*}$ is the pullback via the flow $\varXi_{s} : M \to M$ of $X$ and the parallel transport map $\hat{\varXi}_{s}$ induced by $\nabla$. 
We need an expression in local coordinates $(x^{1}, \ldots, x^{n})$ on $M$, $X = X^{i} \partial / \partial x^{i}$, and a local frame $(e_{1}, \ldots, e_{\rk E})$ in $E$. 
Then $\sqrt{|\rd x|}$ is a nowhere-vanishing section of the halfdensity bundle over $M$. 
Given a section $u = u^{k} e_{k} \otimes \sqrt{|\rd x|}$ 
the Leibniz rule then entails that 
(see e.g~\cite[(25.2.11)]{Hoermander_Springer_2009})
\begin{equation} \label{eq: Lie_derivative_density_coordinate}
	\pounds_{X} u 
	= 
	\bigg( X^{i} \parDeri{x^{i}}{u^{r}} + \frac{1}{2} \mathrm{div} (X) \, u^{r} \bigg) e_{r} \otimes \sqrt{|\rd x|} + u^{r} \pounds_{X} (e_{r}) \otimes \sqrt{|\rd x|}.
\end{equation}

% 
% 
% 
% ======================================= 
\begin{theorem} \label{thm: Hoermander_thm_25_2_4}
	Let $E, \dotCoTanM \to M$ (resp. $F, \dotCoTanN \to N$) be a complex smooth vector bundle and the punctured cotangent bundle over a manifold $M$ (resp. $N$) and $m, m' \in \R$. 
	Assume that $A \in I^{m} (M \times N, C'; \Hom{F, E})$ is a Lagrangian distribution associated with the homogeneous canonical relation $C$ from $\dotCoTanN$ to $\dotCoTanM$. 
	Suppose that $P \in \varPsi \mathrm{DO}^{m'} (M; E)$ be a properly supported pseudodifferential operator with a scalar principal symbol $\symb{P}$ and a subprincipal symbol $\subSymb{P}$, and that $\symb{P}$ vanishes on the projection of $C$ in $\dotCoTanM$. 
	Then, $PA \in I^{m + m'- 1} (M \times N, C'; \Hom{F, E})$ and its principal symbol is  
	\begin{equation*}
		\symb{PA} = \left( - \ri \pounds_{X_{P}} + \subSymb{P} \right) \symb{A}, 
	\end{equation*}
	where $\pounds_{X_{P}}$ is the Lie derivative along the Hamiltonian vector field $X_{P}$ of $\symb{P}$, lifted to a function on $\dotCoTanM \times \dotCoTanN$, so $X_{P}$ is tangential to $C$. 
\end{theorem}
% ======================================= 
% 
% 
% 

% =======================================
\begin{proof}
	Since the principal symbol is locally defined, our strategy is to make use of the partition of unity to boil down the statements in the level of manifolds and employ the scalar-version of this 
	statement~\cite[Thm. 5.3.1]{Duistermaat_ActaMath_1972} 
	(see also~\cite[Thm. 25.2.4]{Hoermander_Springer_2009},~\cite[pp. 451-454]{Treves_Plenum_1980}). 
	The differences in the proof between the scalar and the bundle versions are essentially bookkeeping, yet we provide the main steps for completeness.    
	Locally, the transition functions of $\Maslov$ are constant so the Maslov factor can be ignored while computing the Lie derivative. 
	Moreover, $\symb{A}$ is a matrix-valued density on $C'$ and thus the last term in~\eqref{eq: Lie_derivative_density_coordinate} does not contribute. 
	It is always possible to choose a local coordinate $(x^{i}, \xi_{i}; y^{j}, \eta_{j})$ around $\xxiyeta \in \dotCoTanM \times \dotCoTanN$ such that $x^{i} \xi_{i} + y^{j} \eta_{j} - H$ is a non-degenerate phase function for $C$, where $H$ is a smooth positively homogeneous $\R$-valued function of degree $1$ on an open conic neighbourhood of $(\xi, \eta) \in \dot{T}^{*}_{x} M \times \dot{T}^{*}_{y} N$~\cite[Thm. 21.2.16]{Hoermander_Springer_2007},~\cite[Lem. 25.2.5]{Hoermander_Springer_2009}). 
	Then   
	\begin{equation*}
		A^{l}_{k} (y, z) 
		\equiv 
		(2 \pi)^{- 3 (\nM + \nN) / 4} \int_{\R^{\nM + \nN}}  
		\re^{\ri (\eta_{i} y^{i} + \zeta_{j} z^{j} - H (\eta, \zeta))} \, a^{l}_{k} (\eta, \zeta) \, \rd \eta \, \rd {\zeta}, 
	\end{equation*}
	where $a^{l}_{k} \in S^{m - (\nM + \nN) / 4} (\R^{\nM + \nN})$ having support in a conic neighbourhood of $(y^{i} = \partial H / \partial \eta_{i}, z^{j} = \partial H / \partial \zeta_{j})$. 
	For any $v \in C_{\mathrm{c}}^{\infty} (N; F)$, it follows from the definition of the Fourier transformation that	
	\begin{equation*}
		\widehat{A^{l}_{k} v^{k}} (\xi) 
		=  
		(2 \pi)^{\nM - \frac{3}{4} (\nM + \nN)} \int_{\R^{\nN}} \re^{- \ri H (\xi, \zeta)}  
		a^{l}_{k} (\xi, \zeta) \, \hat{v}^{k} (- \zeta) \, \rd \zeta, 
	\end{equation*}  
	where we have made use of the Fubini's theorem for oscillatory integrals~\cite[$(1.2.4)$]{Hoermander_ActaMath_1971}. 
	Since $P$ is properly supported, one obtains   
	\begin{equation*}
		(P^{r}_{l} A^{l}_{k} v^{k}) (x)
		=
		(2 \pi)^{- \frac{3}{4} (\nM + \nN)} 
		\int_{\R^{\nM}} \rd \xi \int_{\R^{\nN}} \rd \zeta \, \re^{\ri x \cdot \xi - \ri H (\xi, \zeta)} 
		(\totSymb{P})^{r}_{l} (x, \xi) \, a^{l}_{k} (\xi, \zeta) \, \hat{v}^{k} (- \zeta),
	\end{equation*}
	where $(\totSymb{P})^{r}_{l}$ is the total symbol of $P$ in the chosen coordinate-charts and bundle-charts. 
	The preceding equation entails that the Schwartz kernel of $PA$ is 
	\begin{equation*}
		(P^{r}_{l} A^{l}_{k}) (x, y) 
		= 
		(2 \pi)^{- \frac{3}{4} (\nM + \nN)}
		\int_{\R^{\nM}} \rd \xi \int_{\R^{\nN}} \rd \eta \, \re^{\ri (x \cdot \xi + y \cdot \eta - H)} (\totSymb{P})^{r}_{l} (x, \xi) \, a^{l}_{k} (\xi, \eta).  
	\end{equation*}
	One can decompose $(\totSymb{P})^{r}_{l} = p^{r}_{l} + \tilde{p}^{r}_{l}$ where $(p^{r}_{l}) := (\symb{P})^{r}_{l}$ and $\tilde{p}^{r}_{l}$ is the remainder terms. 
	By hypothesis, $p = 0$ on the projection of $C$ in $\dotCoTanM$.
	By means of Taylor's series one can write 
	\begin{equation*} 
		p^{r}_{l} (x, \xi) 
		= \parDeri{\xi_{i}}{(x \cdot \xi + y \cdot \eta - H)} f^{r}_{l, i} (x, \xi, \eta),    
	\end{equation*}
	where $f^{r}_{l, i}$ is homogeneous of degree $m'$ with respect to $(\xi, \eta)$, given by the mean value theorem: 
	\begin{equation*}
		f^{r}_{l,i} (x, \xi, \eta) 
		= 
		\int_{0}^{1} \frac{\partial p^{r}_{l}}{\partial x^{i}} \Big( \uplambda \, x + (1 - \uplambda) \frac{\partial H}{\partial \xi}, \xi \Big) \rd \uplambda. 
	\end{equation*} 
	In particular, $f^{r}_{l,i} |_{C} = (\partial_{x^{i}} p^{r}_{l}) (\partial_{\xi_{i}} H, \xi)$. 
	It is always possible to assume that $a$ vanishes in a neighbourhood of $0$ and then an integration by parts gives for $PA$:   
	\begin{equation*}
		P^{r}_{l} A^{l}_{k}  
		=    
		(2 \pi)^{- \frac{3}{4} (\nM + \nN)}  
		\int_{\R^{\nM + \nN}} \re^{\ri (x \cdot \xi + y \cdot \eta - H)} \,  
		\Big( \tilde{p}^{r}_{l} a^{l}_{k} 
		- 
		D_{\xi_{i}} (f^{r}_{l, i} a^{l}_{k}) \Big) \rd \xi \, \rd \eta. 
	\end{equation*}
	Let $\hat{a}^{l}_{k}$ be the leading order term in $a^{l}_{k}$.
	Then $\mathsf{a}^{l}_{k} = {\hat{a}}^{l}_{k} (\xi, \eta) \sqrt{| \rd \xi| \, | \rd \eta |}$ is the principal symbol of $A$ and the preceding equation implies  
	\begin{equation*}
		(\symb{P A})^{r}_{k}  
		=  
		\Big( \tilde{p}^{r}_{l} \mathsf{a}^{l}_{k} 
		- D_{\xi_{i}} (f^{r}_{l, i}) \mathsf{a}^{l}_{k} 
		- f^{r}_{l, i} D_{\xi_{i}} (a^{l}_{k}) \Big)_{x = \frac{\partial H}{\partial \xi}}   
		\sqrt{| \rd \xi| \, | \rd \eta|} 
	\end{equation*}
	whenever $(\xi, \eta)$ are taken as coordinates on $C$. 
	We will suppress the bundle indices from afterwards as there is no scope for confusion. 

	By definition, $X_{P} = \nicefrac{\partial p}{\partial \xi_{i}} \, \nicefrac{\partial}{\partial x^{i}} - \nicefrac{\partial p}{\partial x^{i}} \,  \nicefrac{\partial}{\partial \xi_{i}}$. 
	So, $\Char{P}$ is a hypersurface near $\xxi$ and $X_{P}$ spans the symplectic-orthogonal (in $\dotCoTanM$) of its tangent space. 
	Now we pullback $p \xxi$ via the projector $\dotCoTanM \times \dotCoTanN \to \dotCoTanM$ to consider it as a function $p \xxiyeta$ on $\dotCoTanM \times \dotCoTanN$. 
	The restriction of this function $p \xxiyeta$ on $C'$ is only a function of $(\xi, \eta)$ as $x^{i} = \partial H / \partial \xi_{i}$ and $y^{j} = \partial H / \partial \eta_{j}$ on $C'$. 
	Thus, $X_{P}$ on $C$ must be of the form 
	\begin{equation*}
		- \parDeri{x^{i}}{} \big( p (x, \xi) \big) \parDeri{\xi_{i}}{} 
		= - f_{i} (x, \xi, \eta) \parDeri{\xi_{i}}{}, 
		\quad 
		x^{i} =  \parDeri{\xi_{i}}{H}, 
	\end{equation*}
	in our chosen parametrisation of $C$ and it is pointwise tangent to $C$.
	We compute 
	\begin{equation*}
		\pounds_{X_{P}} \mathsf{a} 
		=  
		- f_{i} \Big(\parDeri{\xi}{H}, \xi, \eta \Big) \, D_{\xi_{i}} (\mathsf{a}) 
		- (D_{\xi_{i}} f_{i}) \Big(\parDeri{\xi}{H}, \xi, \eta \Big) \, \mathsf{a}  
		- \frac{\ri}{2} \frac{\partial^{2} p}{\partial x^{j} \partial \xi_{j}} \Big(\parDeri{\xi}{H}, \xi \Big)  
	\end{equation*}
	and insert this into the hindmost expression of $\symb{PA}$ to reach our goal:
	\begin{equation*}
		\symb{PA}  
		= (- \ri \pounds_{X_{P}} + \subSymb{P}) \mathsf{a} 
		\mod S^{m + m' - 2} (\ldots),   
	\end{equation*}
	where we have used 
	\begin{equation*}
		\frac{\partial^{2} p}{\partial x^{j} \xi_{j}} (x, \xi) 
		= 
		\parDeri{\xi_{j}}{f_{j}} (x, \xi, \eta)
		- \frac{\partial^{2} H}{\partial \xi_{j} \partial \xi_{i}} (\xi, \eta) \, \parDeri{x^{j}}{f_{i}} (x, \xi, \eta) 
		+ \Big( x^{i} - \parDeri{\xi_{i}}{H} (\xi, \eta) \Big) \frac{\partial f_{i}}{\partial x^{j} \partial \xi_{j}} (x, \xi, \eta)  
	\end{equation*}
	when evaluated at $x = \partial H / \partial \xi$. 
\end{proof}
% ======================================= 
% 
% 
% 
% 
% 
% 
% 
% 
% 
% 
% =======================================  
\section{A few technical results}
% ======================================= 
In this appendix, we collect the following results used in different places in this paper. 
The first result has been proven for scalar pseudodifferential operators~\cite[Thm. 5.3.2]{Duistermaat_ActaMath_1972} 
(see also~\cite[Lem. 26.1.16]{Hoermander_Springer_2009}),  
which we tailor for normally hyperbolic operators. 

% 
% 
% 
% ================================================ 
\begin{lemma} \label{lem: Hoermander_lem_26_1_16}
	Let $\square$ be a normally hyperbolic operator on a smooth complex vector bundle $E \to M$ over a globally hyperbolic spacetime $(M, g)$. 
	Assume that $C^{\pm}$ are the forward and backward geodesic relations~\eqref{eq: def_Feynman_bicharacteristic_relation_NHOp}. 
	If $B \in I^{m +1} \big( M \times M, C^{\pm \prime}; \Hom{E, E} \big), m \in \R$ then one can find $A \in I^{m} \big( M \times M, C^{\pm \prime}; \Hom{E, E} \big)$ such that $\square A - B$ is smoothing.   
\end{lemma}
% ================================================ 
% 
% 
% 
% ================================================
\begin{proof}
	Our task is to solve the equation $\square A \equiv B$ for a given $B$. 
	Let $A_{0} \in I^{m} \big( M \times M, C^{\pm \prime}; \Hom{E, E} \big)$. 
	Since $\symb{\square} = g^{-1}$ vanishes on the lightcone bundle $\coLightBun$ the Lagrangian distribution $\square A_{0}$ is of order $m+1$. 
	By Theorem~\ref{thm: Hoermander_thm_25_2_4}, its principal symbol is given by $\symb{\square A_{0}} = \left( - \ri \, \pounds_{X_{\square}} + \subSymb{\square} \right) \symb{A_{0}}$ and $\square A_{0} - B \in I^{m} (\ldots)$ if 
	\begin{equation*}
		\left( - \ri \, \pounds_{X_{\square}} + \subSymb{\square} \right) \symb{A_{0}} = \symb{B}.  
	\end{equation*}
	Identifying halfdensities with functions and making use of the fact that the Keller-Maslov bundles (Appendix~\ref{sec: FIO_symplecto_bundle}) $\Maslov^{\pm} \to C^{\pm}$ are trivial as vector bundles, we write $\symb{A_{0}} = a_{0} \mathbbm{m}$ and $\symb{B} = b \mathbbm{m}$ where $\mathbbm{m}$ is a non-vanishing section of $\Maslov^{\pm}$ and $a_{0}, b$ are scalar symbols of degree respectively $m, m+1$. 
	The preceding equation then reads $(- \ri X_{\square} + \subSymb{\square}) a_{0} = b$ by~\eqref{eq: Lie_derivative_density_coordinate}. 
	In other words, inserting~\eqref{eq: HVF_NHOp} and~\eqref{eq: subprincipal_symbol_NHOp} in the last equation yield  
	\begin{equation} 
		\left( - \frac{\partial g^{\mu \nu}}{\partial x^{i}} \xi_\mu \xi_\nu \frac{\partial}{\partial \xi_{i}} 
		+ 2 g^{ij} \xi_{j} \frac{\partial}{\partial x^{i}} 
		+ 2 g^{ij} \Gamma_{i}^i \xi_{j}  \right) a_{0} = b.  
	\end{equation}
	This is a transport equation along the  vector field $X_{\square}$. 
	The equation therefore has a unique solution for a given initial condition and $\WFPrime{A_{0}} \subset C^{\pm}$. 

	We will now apply the same line of argument to $B_{1} := B - \square A_{0} \in I^{m} (\ldots)$, that is, one considers $B_{k+1} := B_{k} - \square A_{k} \in I^{m-k} (\ldots)$ for $k = 0, 1, \ldots$, where $B_{0} := B$ and $A_{k} \in I^{m-k} \big( M \times M, C^{\pm \prime}; \Hom{E, E} \big)$. 
	Then $B - B_{k+1} = \square (A_{0} + \ldots + A_{k})$. 
	Let us now set $A :\sim A_{0} + \ldots + A_{k} + \ldots$ in the sense of asymptotic summation (Definition~\ref{def: polyhomogeneous_symbol_mf}). 
	It then follows that $\square A - B \in I^{- \infty} (\ldots)$. 
\end{proof} 
% ================================================
% 
% 
% 

The next result is a variant of H\"{o}rmander's square root construction~\cite[Prop. 2.2.2]{Hoermander_ActaMath_1971} for vector bundles. 

% 
% 
% 
% ================================================ 
\begin{proposition} \label{prop: FIO_2_2_2_matrix}
	Let $\big( E \to M, (\cdot|\cdot) \big)$ be a smooth complex vector bundle over a manifold $M$, endowed with a sesquilinear form $(\cdot|\cdot)$. 
	Suppose that $P \in \PsiDO{m}{M; E}$ is a properly supported pseudodifferential operator of order $m \in \R_{+}$, formally selfadjoint with respect to $(\cdot|\cdot)$, and elliptic in a conic neighbourhood $\cU$ of $\xxiNot \notin \ES{P}$ in $\dotCoTanM$. 
	If the principal symbol of $P$ is given by 
	\begin{equation*}
		p = q^{*} q  
	\end{equation*} 
	for some $q \in C^{\infty} (\cU, \End{E})$ then one can find a properly supported and formally selfadjoint (with respect to $(\cdot|\cdot)$) $Q \in \PsiDO{m/2}{U; E}$ such that 
	\begin{equation*}
		\xxiNot \notin \ES{P - Q^{*} Q}, \qquad \xxiNot \notin \ES{Q},  
	\end{equation*}
	where $U$ is the projection of $\cU$ on $M$. 
\end{proposition}
% ================================================ 
% 
% 
% 
% ================================================ 
\begin{proof}
	The hypothesis on $p$ implies that $q$ is homogeneous of degree $m/2$ which in turns entails that $q \in S^{m/2} (\cU, \End{E})$; cf. Definition~\ref{def: polyhomogeneous_symbol_mf}. 
	We define a properly supported $Q_{0} \in \PsiDO{m/2}{U; E}$ whose principal symbol is $q$. 
	Without lose of generality $Q_{0}$ can be taken selfadjoint, otherwise one can just replace $Q_{0}$ by $(Q_{0} + Q_{0}^{*}) / 2$ without changing the principal symbol. 
	Then $P - Q_{0}^{*} Q_{0} \in \PsiDO{m-1}{U; E}$. 

	Now we are left with estimations for lower order terms and will show that it is always possible to obtain properly supported and formally selfadjoint $Q_{k} \in \PsiDO{m/2-k}{U; E}$ for all $k \in \N$, such that  
	\begin{equation*}
		R_{k} := P - (Q_{0} + \ldots + Q_{k})^{*} (Q_{0} + \ldots + Q_{k}) \in \PsiDO{m-1-k}{U; E}. 
	\end{equation*}
	We proceed inductively. 
	Clearly, $k = 1$ has been checked.    
	Observe that, if $Q_{k}$s have been chosen accordingly then   
	\begin{equation*}
		P - |Q_{0} + \ldots + Q_{k}|^{2} 
		= 
		R_{k} - (Q_{k}^{*} Q_{0} + Q_{0}^{*} Q_{k}) - Q_{k}^{*} Q_{k} + \ldots \in \PsiDO{m-1-k}{U; E}. 
	\end{equation*}
	Since $R_{k}$ is formally selfadjoint, the principal symbol of $R_{k} - R_{k}^{*}$ is $2 \ri \, \Im (\symb{R_{k}})$ modulo $S^{m - 2 - k} (\cU, \End{E})$. 
	Then the desired operators $B_{k}$ are achieved if we set  
	\begin{equation*}
		\symb{R_{k}} = \symb{Q_{k}}^{*} q + q^{*} \symb{Q_{k}}  
	\end{equation*}
	on $\cU$, that is, in the region where $q (x, \hat{\xi})$ is invertible in $\End{E}_{x}$ for each $x \in U$ and $\hat{\xi} := \xi / \| \xi \| \in \mathbb{S}^{\rk (E) - 1}$. 
	Finally the result entails by constructing the total symbol of $Q$ in the sense of asympototic summation (Definition~\ref{def: polyhomogeneous_symbol_mf}); in other words: $Q :\sim \sum_{k=0}^{\infty} Q_{k}$.
\end{proof}
% ================================================ 
% 
% 
% 
% 
% 
% 
% 
% 
% 
% 
% =======================================

\end{document}